\documentclass[11pt,letterpaper,reqno]{amsart}
\usepackage[T1]{fontenc}

\usepackage{amsfonts,amsmath,latexsym,epsfig,enumitem,amssymb,bbm}

\usepackage[usenames,dvipsnames]{color}
\usepackage[unicode,colorlinks=true,linkcolor=Red,citecolor=Green]{hyperref}
\usepackage{cleveref}

\crefformat{section}{\S#2#1#3}
\crefformat{subsection}{\S#2#1#3}
\crefformat{appendix}{\S#2#1#3}

\newtheorem{theorem}{Theorem}[section]
\newtheorem*{theorem*}{Theorem}
\newtheorem{lemma}[theorem]{Lemma}
\newtheorem{claim}[theorem]{Claim}
\newtheorem*{claim*}{Claim}
\newtheorem{cor}[theorem]{Corollary}
\newtheorem{prop}[theorem]{Proposition}

\renewcommand{\le}{\leqslant}
\renewcommand{\ge}{\geqslant}
\renewcommand{\leq}{\leqslant}
\renewcommand{\geq}{\geqslant}

\newcommand{\parag}[1]{\vspace{2mm}

\noindent{\bf #1} }

\theoremstyle{remark}

\newcommand{\definition}
      {\medskip\noindent {\bf Definition:\hspace{0em}}}

\def\qed{\ifvmode\mbox{ }\else\unskip\fi\hskip 1em plus 10fill$\Box$}

\makeatletter
\def\Ddots{\mathinner{\mkern1mu\raise\p@
\vbox{\kern7\p@\hbox{.}}\mkern2mu
\raise4\p@\hbox{.}\mkern2mu\raise7\p@\hbox{.}\mkern1mu}}
\makeatother

\newcommand{\mb}{\mathbf}
\newcommand{\mbt}[1]{\widetilde{\mathbf{#1}}}

\newcommand{\ms}{\mathsf}
\newcommand{\sep}{\,;}

\def\R{\mathbb R}
\def\Z{\mathbb Z}

\def\F{\mathbb F}

\def\N{\mathbb N}
\def\E{\mathbb E}

\def\mc{\mathcal }

\newcommand{\slope}{\mathrm{slope}}

\def\T{\mathbb T}
\DeclareMathOperator{\Err}{Err}

\DeclareMathOperator{\Area}{Area}
\DeclareMathOperator{\Dir}{Dir}
\DeclareMathOperator{\supp}{supp}

\title{ Lower bounds for incidences }

\author{Alex Cohen}
\thanks{Department of Mathematics, Massachusetts Institute of Technology, Cambridge, MA. Email: {\tt alexcoh@mit.edu}. Research supported by an NSF GRFP Fellowship and a Hertz Foundation fellowship. Part of this work was completed while AC was an intern at Microsoft Research New England.}

\author{Cosmin Pohoata}
\thanks{Department of Mathematics, Emory University, Atlanta, GA. Email: {\tt cosmin.pohoata@emory.edu}. Research supported by NSF Award DMS-2246659.}

\author{Dmitrii Zakharov}
\thanks{Department of Mathematics, Massachusetts Institute of Technology, Cambridge, MA. Email: {\tt  zakhdm@mit.edu}. Research supported by the Jane Street Graduate Fellowship.}

\date{}

\begin{document}
\begin{abstract}
Let $p_1,\ldots,p_n$ be a set of points in the unit square and let $T_1,\ldots,T_n$ be a set of $\delta$-tubes such that $T_j$ passes through $p_j$. We prove a lower bound for the number of incidences between the points and tubes under a natural regularity condition (similar to Frostman regularity). As a consequence, we show that in any configuration of points $p_1,\ldots, p_n \in [0,1]^2$ along with a line $\ell_j$ through each point $p_j$, there exist $j\neq k$ for which $d(p_j, \ell_k) \lesssim n^{-2/3+o(1)}$. 

It follows from the latter result that any set of $n$ points in the unit square contains three points forming a triangle of area at most $n^{-7/6+o(1)}$. This new upper bound for Heilbronn's triangle problem attains the high-low limit established in our previous work arXiv:2305.18253.

\end{abstract}

\maketitle 

\section{Introduction}\label{sec:introduction}

\subsection{Main results}
Given a set of points $P$ and a set of geometric objects $L$ in $\R^d$, an \textit{incidence} is a pair $(p, \ell) \in P\times L$ where $p$ lies on $\ell$. We denote by $I(P, L)$ the number of incidences in $P\times L$. 

Proving upper bounds for incidences is a fundamental problem in combinatorics and analysis. For example, the Szemer\'edi-Trotter theorem \cite{ST83} says that if $P$ is a set of points and $L$ a set of lines in $\R^2$,  
\begin{align}\label{eq:ST}
    I(P, L) \lesssim |P|^{2/3} |L|^{2/3} + |P| + |L|. 
\end{align}
The notation $X \lesssim Y$ means $X \leq CY$ for some universal constant $C > 0$. This is a sharp bound---see \cite{guthST} and the references therein. In a recent breakthrough result (building on papers by Bourgain \cite{bourgain2003erdHos}, Orponen-Shmerkin \cite{OrponenShmerkinFurstenberg2023,OrponenShmerkinABC2023}, and Orponen-Shmerkin-Wang \cite{OrponenShmerkinWang}), Ren and Wang \cite{RenWang2023} proved the Furstenberg set conjecture, a continuous analogue about points and tubes. 

This present paper is about \textit{lower bounds} for incidences between a set of points $P$ and $\delta$-tubes $\T$ (a $\delta$-tube is the $\delta$-neighborhood of a line). If the points and tubes are placed uniformly at random in the unit square, the expected number of incidences is around $\delta |P||\T|$. However, if the points all lie in the top half of $[0,1]^2$ and the tubes all lie in the bottom half, there are no incidences. In order to rule out this example we consider a configuration of points $p_1, \ldots, p_n$ along with a $\delta$-tube $T_j$ through each point (by this we mean the central line $\ell_{T} \subset T$ passes through $p$). There are $n$ trivial incidences because $p_j \in T_j$, and we are interested in finding nontrivial incidences. The constraint that $T_j$ passes through $p_j$ arises naturally in applications and forces the points and tubes to be spatially mixed together.
Here is one of our main theorems.
\begin{theorem}\label{thm:extra_inc_one_tube}
For all $\varepsilon > 0$ the following holds for $\delta < \delta_0(\varepsilon)$. Let $p_1,\ldots, p_n$ be a set of points in $[0,1]^2$ along with a $\delta$-tube $T_j$ through each point. If $n \geq \delta^{-3/2-\varepsilon}$, there is some nontrivial incidence $p_j \in T_k$ where $j \neq k$. 
\end{theorem}
The following corollary is equivalent to \Cref{thm:extra_inc_one_tube}.  
\begin{cor}\label{cor:dist_P_L}
Let $n \geq 2$ be an integer and let $p_1,\ldots, p_n$ be a set of points in $[0,1]^2$ along with a line $\ell_j$ through each point. Then there is some $j\neq k$ for which $d(p_j, \ell_k) \lesssim_{\varepsilon} n^{-2/3+\varepsilon}$. 
\end{cor}
By a subsampling argument, we also get the following incidence lower bound:
\begin{cor}\label{cor:incidence_lower_bound}
    Let $p_1, \ldots, p_n$ be a set of points in $[0,1]^2$ along with a $\delta$-tube $T_j$ through each point. Then for $P= \{p_1, \ldots, p_n\}$ and $\T = \{T_1, \ldots, T_n\}$ we have $I(P, \T) \gtrsim_\varepsilon \delta^{3/2+\varepsilon} |P| |\T|$.
\end{cor}

Indeed, note that for a uniformly random set of points $P' \subset P$ such that $|P'| \geq \delta^{-3/2-\varepsilon}$ and the corresponding set of $\delta$-tubes $\T' \subset \T$ through these points, Theorem \ref{thm:extra_inc_one_tube} tells us that there is at least one extra incidence between $P'$ and $\T'$. Moreover, by applying Theorem \ref{thm:extra_inc_one_tube} to every subset $P'' \subset P'$ of size $|P'|/2$ we find at least $|P'|/2$ distinct extra incidences between $P'$ and $\T'$. On the other hand, the expected number of extra incidences determined by $P'$ and $\T'$ is at most $I(P,\T)\frac{|P'|}{|P|}\frac{|\T'|}{|\T|}$. 

We can generalize \Cref{thm:extra_inc_one_tube} by putting a whole family of tubes through each point. In order to state this generalization we need some notions that will be important throughout the paper. We say a set of points $P \subset [0,1]^2$ is a $(\delta, t, C)$-set if the points are $\delta$-separated\footnote{Often the definition of a $(\delta, t, C)$-set does not include $\delta$-separation, but this is a technically useful condition for dealing with incidence counts.} and satisfy
\begin{equation}\label{eq:points_frostman}
    |P \cap B_w(x_0)| \leq C w^t |P|\qquad \text{for every $w$-ball $B_w(x_0)$, $w \in [\delta, 1]$.}
\end{equation}
Applying this inequality to a ball $B_{\delta}(p)$ centered at a point of $P$ yields $|P| \geq \frac{1}{C}\delta^{-t}$.
We used the Euclidean metric on the set of points to define $(\delta, t, C)$-sets. Here is a natural metric on the space of lines. We let $d(0, \ell)$ be the distance to the origin, $\theta(\ell) \in \R/\pi\Z$ be the angle, and we set
\begin{equation}\label{eq:line_metric_all_slope}
    d(\ell_1, \ell_2) = |d(0, \ell_1) - d(0, \ell_2)| + |\theta(\ell_1) - \theta(\ell_2)|.
\end{equation}
Throughout this paper we restrict attention to lines with $d(0, \ell) \leq 10$ because we only care about lines that intersect $[-1,1]^2$. 
We say a set of lines $L$ is a $(\delta, t, C)$-set if they are $\delta$-separated and satisfy 
\begin{equation}\label{eq:lines_frostman}
    |L \cap B_w(\ell_0)| \leq C w^t |L|\qquad \text{for every $w$-ball $B_w(\ell_0)$, $w \in [\delta, 1]$.}
\end{equation}
The \textit{Frostman regularity} hypotheses \eqref{eq:points_frostman} and \eqref{eq:lines_frostman} bound how much points/lines concentrate in balls/tubes. It is often better to think in terms of covering numbers. Let $|P|_w$ be the minimal number of $w$-balls needed to cover $P$. If $P$ is a non-empty $(\delta, t, C)$-set then $|P|_w \gtrsim w^{-t}$ for all $w \geq \delta$. Moreover, if $P' \subset P$ is a subset with $|P'| \gtrsim |P|$, then \eqref{eq:points_frostman} implies $|P'|_w \gtrsim w^{-t}$ as well. The Frostman hypothesis gives a robust lower bound on the covering number at every scale. 
\bigskip\\
Let $P$ be a set of points in the unit square and for each $p \in P$ let $\T_p$ be a $(\delta, s, C)$-set of tubes through $p$. The set of tubes $\T_p$ is determined by the set of directions $\Theta_p \subset S^1$ that they span. If the tubes $\T_p$ are a $(\delta, s, C)$-set then $\Theta_p$ is a $(\delta/20, s, 20C)$-set as a subset of $S^1$. 

\begin{theorem}\label{thm:exists_extra_incidence}
Fix $t \in [1, 2]$ and $s \in [0,1]$ such that $2t+s>3$. There exists $\eta(t, s) > 0$ such that the following holds for all $\delta < \delta_0(t, s)$. 
Let $P \subset [0,1]^2$ be a set of $\delta^{-t}$ many points. For each $p \in P$ let $\T_p$ be a $(\delta, s, \delta^{-\eta})$-set of $\delta$-tubes through $p$ and let $\T = \bigsqcup_{p} \T_p$. Then there is some nontrivial incidence between $P$ and $\T$, meaning there is a point $p \in P$ and a tube $T \in \T \setminus \T_{p}$ so that $p \in T$.  
\end{theorem}
By taking $s = 0$, \Cref{thm:exists_extra_incidence} implies \Cref{thm:extra_inc_one_tube}. 
Using the framework we set up to prove \Cref{thm:exists_extra_incidence}, we can naturally recover a weaker version of an incidence lower bound that was proved earlier this year by D{\k a}browski, Goering, and Orponen \cite{dabrowski2024}. See \S\ref{sec:prior_work_R2} for a more detailed comparison.
\begin{theorem}\label{thm:inc_lower_bd_t_set}
Fix $t \in [1,2]$ and $s \in [0,1]$ so that $t+s > 2$. For every $\varepsilon > 0$ there exists $\eta(t, s, \varepsilon) > 0$ such that the following holds for all $\delta < \delta_0(t,s,\varepsilon)$. Let $P \subset [0,1]^2$ be a $(\delta, t, \delta^{-\eta})$-set. For each $p \in P$ let $\T_p$ be a $(\delta, s, \delta^{-\eta})$-set of tubes through $p$, and let $\T = \bigsqcup_p \T_p$. Then 
\begin{align*}
    I(P, \T) \geq \delta^{1+\varepsilon} |P|\, |\T|.
\end{align*}
\end{theorem}
Notice that $\delta |P|\, |\T|$ is the roughly the expected number of incidences if the points and tubes are distributed uniformly at random. 

\subsection{Background on Heilbronn's triangle problem}
Given an integer $n \geq 3$, Heilbronn's triangle problem asks for the smallest number $\Delta = \Delta(n)$ such that in every configuration of $n$ points in the unit square one can find three among them which form a triangle of area at most $\Delta$. It is easy to prove $\Delta \lesssim 1/n$: just split the unit square into fewer than $n/2$ strips and choose three points in the same strip. In 1951 Roth \cite{Roth1951} showed that $\Delta = o(1/n)$. Roth's proof consisted of a remarkable density increment argument that quantitatively showed $\Delta \lesssim n^{-1} (\log \log n)^{-1/2}$. In 1971 Schmidt \cite{Schmidt1971} improved on Roth's upper bound by showing $\Delta \lesssim n^{-1} (\log n)^{-1/2}$ with a much simpler argument. In the subsequent year, Roth \cite{Roth1972_1} came back to Heilbronn's problem and proved $\Delta \lesssim n^{-\mu}$ holds for an absolute constant $\mu > 1$. Roth wrote a few papers developing his method \cite{Roth1972_2,Roth1973} eventually getting $\mu \approx 1.117$. In 1981 Koml\'os, Pintz and Szemer\'edi \cite{KomlosUpperBound} improved this exponent to $\mu = 8/7 - o(1) \approx 1.14285$. More precisely they showed that $\Delta \leq \exp(c\sqrt{\log n})n^{-8/7}$ holds for some absolute constant $c>0$. 
In \cite{CPZ} we recently improved their estimate by a polynomial factor. 
\begin{theorem}[\cite{CPZ}] \label{old} 
For sufficiently large $n$, any set of $n$ points in the unit square contains three points forming a triangle of area at most $$\Delta \leq n^{-8/7-1/2000}.$$
\end{theorem}
Roth's 1972 paper introduced a two-step approach to the problem that all subsequent work has followed. 
First he related the triangle problem to an incidence lower bounds problem. Let $P \subset [0,1]^2$ be the set of points. Pick some parameters $u >0$ and $w>0$, and let 
\begin{align*}
    \T = \{\text{the $w$-tube $T_{p_1,p_2}$ connecting $p_1$ and $p_2$ for all $p_1, p_2 \in P$ with $|p_1-p_2| \leq u$}\}. 
\end{align*}
There is a trivial incidence between the points $p_1,p_2$ and the tube $\T_{p_1,p_2}$. If there is any nontrivial incidence $p_3 \in T_{p_1,p_2}$, then $p_1,p_2,p_3$ forms a triangle of area $\leq uw$. Roth proved that there are lots of incidences between $P$ and $\T$, so in particular there is some nontrivial incidence. First he proved an \textit{initial estimate} showing that there are lots of incidences at a large scale $w_i \sim n^{-\varepsilon}$, and then an \textit{inductive step} showing that there are still lots of incidences at a much smaller final scale $w_f$. 

In order to execute the inductive step Roth proved a multiscale inequality about incidences---we will state the version of this inequality we proved in \cite{CPZ}. The number of incidences between $P$ and $L$ at scale $w$ is 
\begin{equation}\label{eq:hard_inc_count}
   \#\{(p, \ell) \in P\times L\, :\, d(p, \ell) \leq w\}.
\end{equation}
For technical reasons, we use a \textit{smoothed} incidence count 
\begin{equation*}
    I(w) = I(w; P, L) = \sum_{(p,\ell)\in P\times L} \eta(d(p, \ell)/w)
\end{equation*}
where $\eta$ is a certain smooth bump function (see \cref{subsec:notation}, if we took $\eta = \chi_{[-1,1]}$ we would get the count \eqref{eq:hard_inc_count}).
For $P$ a set of points and $L$ a set of lines, let 
\begin{align*}
    \ms M_{w\times w}(P) &= \max_{x_0} |P\cap B_w(x_0)|, \\ 
    \ms M_{1\times w}(L) &= \max_{\ell_0} |L\cap B_w(\ell_0)|. 
\end{align*}
These numbers measure how much $P$ and $L$ concentrate in subregions. 
\begin{theorem}[High-low inequality]\label{thm:high_low_ineq}
Let $P \subset [-1, 1]^2$ be a set of points and $L$ a set of lines intersecting $[-1,1]^2$. We have 
\begin{align*}
    \Bigl|\frac{I(w)}{w|P||L|} - \frac{I(w/2)}{(w/2)|P||L|}\Bigr| &\lesssim \Bigl(\frac{\ms M_{w\times w}(P)}{|P|} \frac{\ms M_{1\times w}(L)}{|L|}w^{-3}\Bigr)^{1/2}.
\end{align*}
\end{theorem}
For completeness we include a proof in \S\ref{sec:high_low_proof}. The normalized incidence count $\frac{I(w)}{w|P||L|}$ is $\sim 1$ if $P$ and $L$ are placed uniformly at random (when $P$ and $L$ are sufficiently large). The left hand side is the change in the normalized incidence count, and the right hand side is small if $P$ and $L$ do not concentrate too much in subregions. In 1972 Roth introduced a version of this inequality in order to prove lower bounds for incidences. In 2017 Guth, Solomon, and Wang \cite{GuthSolomonWang2019} independently proved a version of this inequality in order to prove upper bounds for incidences. We call it the high-low inequality because of their paper.

Roth observed that if the point set $P$ has extra concentration in subregions of $[0,1]^2$, we can exploit this by zooming into those subregions and inductively finding a small triangle there. If $P$ does not have too much concentration, we can exploit that through the high-low inequality. 
In his 1971 paper Roth \cite{Roth1972_1} introduced an initial estimate and an inductive step to get the first power saving bound. In his subsequent papers Roth \cite{Roth1972_2,Roth1973} kept the same initial estimate and improved the inductive step. In 1981, Komlos, Pintz, and Szemer\'edi \cite{KomlosUpperBound} kept the initial estimate and executed a sharp inductive step to prove $\Delta \lesssim n^{-8/7+o(1)}$. 

In our paper \cite{CPZ} we gave a new approach to the initial estimate. Our new approach related the initial estimate to direction set estimates from projection theory. Using classical direction set estimates we recovered the $8/7$ exponent, and using a very strong recent improvement due to Orponen, Shmerkin, and Wang \cite{OrponenShmerkinWang} we improved the bound. 

We also proved in \cite[Theorem 1.2]{CPZ} that if $P$ is \textit{homogeneous}, meaning any $n^{-1/2}\times n^{-1/2}$ square has $\lesssim 1$ point, then $P$ has a triangle of area at most $n^{-7/6+o(1)}$. In this special case it is easy to prove an initial estimate, and the bottleneck is the inductive step. The high-low method has a barrier at the exponent $7/6$ due to a Szemer\'edi-Trotter type example, see \cite[\S3]{CPZ}.

\subsection{Application to Heilbronn's triangle problem} \label{HeilbronnApp}
\Cref{cor:dist_P_L} leads to a new bound for Heilbronn's problem that achieves the high-low barrier described in the last section. This is a new approach to Heilbronn's problem that departs from Roth's setup. At a high level, we solve a more natural and general incidence lower bounds problem and then apply our results to Heilbronn's triangle problem. 
\begin{theorem}\label{thm:heilbronn_result}
For any $\varepsilon > 0$, every set of $n$ points in the unit square contains a triangle of area 
\begin{align*}
    \Delta \lesssim_{\varepsilon} n^{-7/6+\varepsilon}.
\end{align*}
\end{theorem}
\begin{proof}
In any set of $n \geq 2$ points in the unit square there are two points with distance $<5\sqrt{2}/\sqrt{n}<10/\sqrt{n}$. To see this split $[0,1]^2$ into an even grid of $\frac{5}{\sqrt{n}}\times \frac{5}{\sqrt{n}}$ squares and use the pigeonhole principle. 

Let $P\subset [0,1]^2$ be a set of $n$ points. Choose two points $p_1, p_1'$ with $d(p_1,p_1') \leq 10/\sqrt{n}$. Remove them and find two more points $p_2,p_2'$ with $d(p_2,p_2') \leq 10/\sqrt{n-2}$. Continuing in this way we can find $\lfloor n/4\rfloor $ pairs $\{(p_j,p_j')\}_{j=1}^{\lfloor n/4\rfloor }$ with $d(p_j,p_j') \leq 20/\sqrt{n}$. 

Let $\ell_j$ be the line through $p_j, p_j'$. 
By \Cref{cor:dist_P_L} there is some $j\neq k$ such that $d(p_j, \ell_k) \lesssim_{\varepsilon} n^{-2/3+\varepsilon}$. The triangle formed by $p_k, p_k', p_j$ has area $\lesssim_{\varepsilon} n^{-7/6+\varepsilon}$. 
\end{proof}

Our proof is different from the framework Roth setup in 1972. That said, we still use an initial estimate and a high-low inductive step to prove \Cref{thm:extra_inc_one_tube}. In the context of Heilbronn's problem, the main advantages of the new approach are that we more efficiently exploit extra concentration of the point/line set in subregions, and we add flexibility to the initial estimate. 

\subsection{Frostman regularity over rectangles and the main incidence lower bound}\label{subsec:intro_statement_main_thm}
The starting point for this paper is the observation that \Cref{thm:extra_inc_one_tube} is symmetric under blowing up into subrectangles. 

Let $R$ be a $u\times uw$ rectangle in $[-1,1]^2$ with $\slope(R) \in [-1,1]$. We say a point-line pair $p\in \ell$ lies in $R$ if $p\in R$ and $|\slope(\ell) - \slope(R)| \leq w$. Although this isn't precisely true, we can think of this slope condition as saying that $\ell$ intersects the short ends of $R$ and not the long ends. Let $\psi_R$ be the affine isomorphism mapping $R \mapsto [-1,1]^2$. 
Given a set of point-line pairs $X = \{(p_j, \ell_j)\}_{j=1}^n$ with $p_j \in \ell_j$, we define the \textit{blowup} into $R$ by 
\begin{align*}
    X^{R} = \psi_{R}(X\cap R)
\end{align*}
where $X\cap R$ denotes the point-line pairs lying in $R$ in the sense described above.
If $T$ is a $\delta$-tube passing through $R$ and having $|\slope(T) - \slope(R)| \leq w$, then $\psi_R(T)$ is a $\delta/uw$ tube. 
It follows that if $X$ has no nontrivial incidences at scale $\delta$, then $X^R$ has no nontrivial incidences at scale $\delta/uw$, see \Cref{fig:rectrescale}.
\begin{figure}
    \centering
    \includegraphics[width=0.7\linewidth]{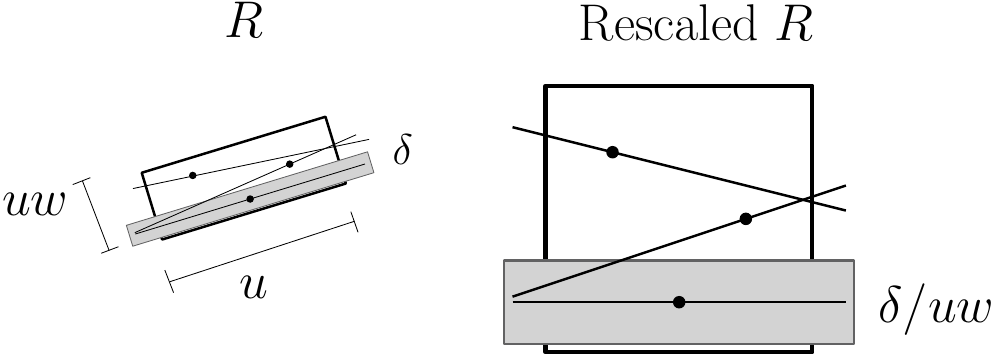}
    \caption{If we rescale a $u\times uw$ rectangle to the unit square, a $\delta$-tube pointing in the long direction of the rectangle maps to a $\delta/uw$-tube.}
    \label{fig:rectrescale}
\end{figure}
This symmetry leads us to analyze a set of point-line pairs in terms of concentration over rectangles of all side ratios. That is one of the main new ideas of our paper. 

Let $p = (a,b) \in [-1,1]^2$ be a point and $\ell$ a line through $p$ with slope $c \in [-1,1]$. We parametrize this pair with the triple $(a,b,c) \in  [-1,1]^3$ and denote our configuration space by $\Omega = [-1,1]^3$. In this notation we let $\mb X = \{\omega_1, \ldots, \omega_n\}$ be a set of point-line pairs, and for $R$ a $u\times uw$ rectangle we let 
\begin{align*}
    \mb R = R \times [\slope(R) - w, \slope(R)+w] \subset \Omega
\end{align*}
be the lift of $R$ to the configuration space. To be precise, for $(a_0, b_0, c_0) \in [-1,1]^3$ we define the \textit{phase space rectangle centered at $(a_0, b_0, c_0)$} by
\begin{align}\label{eq:intro_phase_space_rect}
    \mb R = \mb R_{u\times uw\times w}(a_0, b_0, c_0) = \{(a_0 + t, b_0 + c_0t + r, c_0+s)\, :\, (t,r,s)\in [-u,u]\times [-uw,uw]\times [-w,w]\}
\end{align}
Now the set of pairs lying in $R$ is just the set intersection $\mb X \cap \mb R$. The blowup of $\mb X$ into $\mb R$ is given by 
\begin{align*}
    \mb X^{\mb R} &= \{\psi_{\mb R}(a,b,c)\, :\, (a,b,c) \in \mb X\cap \mb R\}, \\
    \psi_{\mb R} &= \text{The affine map taking $\mb R$ to $[-1,1]^3$.}
\end{align*}
Our goal is to prove that for any $\gamma > 3/2$, if $|\mb X| \geq \delta^{-\gamma}$ then the point set $P[\mb X]$ and the line set $L[\mb X]$ have a nontrivial incidence at scale $\delta$ (as long as $\delta < \delta_0(\gamma))$. If we are trying to prove this statement at scale $\delta_*$, we might suppose inductively that we already proved it for all $\delta > 2\delta_*$. Let $\mb X \subset \Omega$ be a set of size $\delta_*^{-\gamma}$ and suppose $\mb X$ has no nontrivial incidence at scale $\delta_*$. Then for any $u\times uw\times w$ rectangle $\mb R$, the blowup $\mb X^{\mb R}$ has no nontrivial incidence at scale $\delta_* / uw$. By the inductive hypothesis $|\mb X\cap \mb R| \leq (\delta_*/uw)^{-\gamma}$, so 
\begin{equation}\label{eq:concentration_X_squares}
    |\mb X \cap \mb R| \leq u^{\gamma}w^{\gamma}|\mb X|. 
\end{equation}
This regularity condition controls how much $\mb X$ concentrates in rectangles with all side ratios. 
As a special case we can choose $\mb R$ to be a $u\times u\times 1$ square and we find $P[\mb X]$ is a $(\delta, \gamma, C)$-set. If we choose $\mb R$ to be a $1\times u\times u$ tube we find $L[\mb X]$ is a $(\delta, \gamma, C)$-set. This means that the high-low error (the right hand side of \Cref{thm:high_low_ineq}) is bounded by 
\begin{align*}
    \Bigl(\frac{\ms M_{w\times w}(P)}{|P|} \frac{\ms M_{1\times w}(L)}{|L|}w^{-3}\Bigr)^{1/2} \lesssim w^{\gamma-3/2}.
\end{align*}
If $\gamma > 3/2$ then the right hand side is small, so the high-low error is small, and we have a chance to execute the inductive step. 

Inspired by inequality \eqref{eq:concentration_X_squares} we say a set $\mb X \subset \Omega$ is a $(\delta, \alpha, \beta, C)$-set if it is $\delta$-separated in $\R^3$ and satisfies the $(\alpha, \beta)$-Frostman condition 
\begin{align*}
    |\mb X\cap \mb R| \leq C u^{\alpha}w^{\beta} |\mb X|\qquad \text{for every $u\times uw\times w$ rectangle $\mb R$ with $uw\geq \delta$.}
\end{align*}
The high-low error is $w^{(\alpha+\beta-3)/2}$, so if $\alpha + \beta > 3$ we can hope to use an inductive step to find lots of incidences. 

Our main theorem is an incidence lower bound under this Frostman condition. Given a set $\mb X \subset \Omega$, we define the incidence count at scale $\delta$ to be 
\begin{align*}
    I(\delta; \mb X) = \sum_{\omega_1, \omega_2 \in \mb X} \eta(d(p_{\omega_1}, \ell_{\omega_2})/\delta)
\end{align*}
where $\eta$ is the smooth bump function described in \cref{subsec:notation}. Here if $\omega = (a,b,c)$ is a phase space point, $p_{\omega} = (a,b)$ is the point defined by $\omega$ and $\ell_{\omega} = \{(a+t,b+ct)\, :\, t \in \R\}$ is the line through $p_{\omega}$ with slope $c$. 
\begin{theorem}\label{thm:incidence_lower_bd}
Let $\alpha, \beta \in [1,2]$ satisfy $\alpha+\beta > 3$, and let $\varepsilon > 0$. There exists $\eta = \eta(\alpha, \beta, \varepsilon) > 0$ such that the following holds for all  $\delta < \delta_0(\alpha, \beta, \varepsilon)$. Suppose $\mb X \subset \Omega$ is a $(\delta, \alpha, \beta, \delta^{-\eta})$-set. Then 
\begin{align*}
    I(\delta; \mb X) \geq \delta^{1+\varepsilon}|\mb X|^2.
\end{align*}
\end{theorem}

In \cref{sec:space_pt_line_pairs} we setup our phase space formalism in more detail, and in \cref{subsec:pf_thm_inc_lower_bd} we use this formalism we restate the definition of a $(\delta, \alpha, \beta, C)$-set and the statement of \Cref{thm:incidence_lower_bd}.
Notice that $\delta|\mb X|^2$ is the expected incidence count\footnote{We count $P[\mb X]$ and $L[\mb X]$ with multiplicity, so $|\mb X| = |P[\mb X]| = |L[\mb X]|$.} if the points and lines are placed uniformly at random. We remark that if $\mb X$ is a $(\delta, \alpha, \beta, \delta^{-\eta})$-set for $\alpha+\beta>3$ then the matching upper bound $I(\delta; \mb X) \le \delta^{1-\varepsilon} |\mb X|^2$ follows from the high-low inequality. On the other hand, for any $\gamma < 3/2$ there exists a $(\delta, \gamma, \gamma, C)$-set $\mb X$ with $I(\delta; \mb X) \ge \delta^{1-\varepsilon} |\mb X|^2$, $\varepsilon = \varepsilon(\gamma)>0$.\footnote{Such an ${\mb X}$ can be defined as the set of incident point-line pairs in the Szemer\'edi--Trotter example, see e.g. \cite[Section 3]{CPZ}.}

To prove \Cref{thm:exists_extra_incidence} using \Cref{thm:incidence_lower_bd} we blowup into a rectangle in which $\mb X$ is $(\alpha,\beta)$-Frostman. To do so we choose the $u\times uw\times w$ rectangle $\mb R$ which maximizes the quantity $|\mb X\cap \mb R| u^{-\alpha}w^{-\beta}$, and then find that $\mb X^{\mb R}$ is $(\alpha, \beta)$-Frostman. Part of the work is checking that the maximizing rectangle $\mb R$ is not too small and that the blowup set $\mb X^{\mb R}$ is large. 

\subsection{Finite fields}\label{sec:finite_fields}
In the finite field setting, Vinh \cite{Vinh} proved the following theorem in 2011. 
\begin{theorem}[\cite{Vinh}]\label{thm:vinh}
Let $\F_q$ be a finite field, $P \subset \F_q^2$ a set of points, and $L$ a set of lines in $\F_q^2$. Then 
\begin{equation*} \label{eq:vinh}
    \Bigl|\, I(P, L) - \frac{1}{q}|P||L|\, \Bigr| \leq q^{1/2}|P|^{1/2} |L|^{1/2}.
\end{equation*}
\end{theorem}

Vinh proved Theorem \ref{thm:vinh} in order to establish an analogue of the Szemer\'edi-Trotter upper bound \eqref{eq:ST} over $\mathbb{F}_{q}^{2}$ in the case where $|P|$ and $|L|$ are large. However, Theorem \ref{thm:vinh} also gives a lower bound for the number of incidences. In hindsight Vinh's estimate can be seen as a high-low estimate over finite fields.
Remember that in the continuous setting a set of points and lines have no incidences if the points concentrate in the upper half of the square and lines concentrate in the lower half. Over finite fields points cannot concentrate in macroscopic subregions, which is why any large enough set of points and lines must have incidences over $\F_q$. 

Let $p_1, \ldots, p_n$ be a set of points in $\F_q^2$ and let $\ell_j$ be a line through $p_j$ for each $j$. Vinh's estimate implies that if $n > q^{3/2}+q$, there is some $j\neq k$ for which $p_j \in \ell_k$. This is a finite field analogue of Theorem \ref{thm:extra_inc_one_tube} and is sharp up to constants because of the following example.
Let $p$ be a prime and choose $q = p^2$. For $a \in \F_q$ let $\bar a= a^p$ be the Galois conjugate and let $N(a) = a \bar a = a^{p+1}$ be the norm. Consider the affine version of the {\em Hermitian unital} in $\F_q^2$,
$$
P = \{ (a, b) \in \F_q^2:~ N(a) + N(b) = 1 \}.
$$
We have $|P| \sim q^{3/2}$, and for each point $x = (a,b) \in P$ there exists a unique `tangent' $\F_q$-line $\ell_x \subset \F_q^2$ such that $\ell_x \cap P = \{x\}$. This tangent line $\ell_{x}$ is given by 
$$
\ell_x = \{(a + t \bar b, b - t \bar a), ~ t \in \F_q\}.
$$
Letting $L = \left\{\ell_{x}:\ x \in P\right\}$, we thus get a configuration of points $P$ and lines $L$ in $\mathbb{F}_{q}^{2}$ with the desired property. We refer to \cite{PZskew} for the proof that $\ell_{x} \cap P = \left\{x\right\}$.
The Hermitian unital was also recently used by Mattheus and Verstraete \cite{mattheus2023asymptotics} to give lower bound constructions for the off-diagonal Ramsey number $R(4,t)$. 
The Hermitian unital example helps explain why it is hard to improve on \Cref{thm:extra_inc_one_tube} using the high-low method. 
  
\subsection{Prior work}\label{sec:prior_work_R2}
In the continuous setting D{\k a}browski, Goering, and Orponen \cite{dabrowski2024} recently proved an incidence lower bounds theorem which is related to the present paper. 
Say $N \subset \R^2$ is an $s$-Nikod\'ym set if for every $x \in N$, there is $s$-dimensional family of lines through $x$ which do not intersect any other point of $N$. 
\begin{theorem}[\cite{dabrowski2024}]\label{thm:nikodym}
Any $s$-Nikod\'ym set $N \subset \R^2$ has $\dim_H N \leq 2-s$. 
\end{theorem}
The paper \cite{dabrowski2024} proceeds by proving a discretized version of the result and then taking a limit. The discretized version is morally similar to our \Cref{thm:inc_lower_bd_t_set}, but their result is stronger because they find a large density of $1$-separated incidences. By contrast, our result can be used to find $\delta^{\varepsilon}$-separated incidences after combining the lower bound with an incidence upper bound on a smaller scale. Under the hood though, the proof of \Cref{thm:nikodym} and \Cref{thm:inc_lower_bd_t_set} is quite similar. They use a similar notion to \textit{direction stability}, which they call \textit{tightness}, to prove an initial estimate. Then they use Sobolev regularity in place of the high-low estimate to prove an inductive step. Our work was done independently of theirs. We note that this circle of ideas relates to the special case \cref{sec:horizontal_good_line} of \Cref{thm:incidence_lower_bd}, and the general case is harder and requires different ideas (in particular, the idea of considering phase space rectangles of all side ratios). 

Also in the continuous setting, Kuca, Orponen, and Sahlsten \cite{Kuca2023} proved the following:
\begin{theorem}[\cite{Kuca2023}]\label{thm:continuous_sarkozy}
There exists an $\varepsilon > 0$ such that if $K \subset \R^2$ has Hausdorff dimension $\geq 2-\varepsilon$, there exists an $x \in K$ and a $z \neq 0$ so that $x+(z,z^2)\in K$.  
\end{theorem}
A discretized version of this Theorem gives an incidence lower bound for a set of points along with a parabola through each point. The methods are quite different from those previously discussed---it does not involve an initial estimate and inductive step and it is specific to the regime very close to two dimensions. This result was significantly generalized by Bruce and Pramanik \cite{bruce2023}.

\subsection{Paper outline}
In \cref{sec:pf_sketch} we sketch the proof of our main incidence lower bounds result, \Cref{thm:incidence_lower_bd}. In \cref{sec:space_pt_line_pairs} we setup the phase space formalism for analyzing point-line pairs. In \cref{sec:two_step_inc_lower_bd} we setup our incidence lower bounds framework and reduce proving incidence lower bounds to proving something about Lipschitz functions. In \cref{sec:pf_main_results} we state the main Lipschitz function theorem of the paper, \Cref{thm:lip_func_theorem}, and using this Theorem we prove the main result \Cref{thm:incidence_lower_bd}, the main consequence \Cref{thm:exists_extra_incidence}, and the auxiliary result \Cref{thm:inc_lower_bd_t_set}. In \cref{sec:pf_lip_func_thm} we prove \Cref{thm:lip_func_theorem} about Lipschitz functions. In Appendix \ref{sec:misc_proofs} we prove several auxiliary results used throughout the paper, and in Appendix \ref{sec:heilbronn_kgon} we show how \Cref{thm:extra_inc_one_tube} relates to Heilbronn's problem for $k$-gons. 

\subsection{Notation}\label{subsec:notation}
We use $A\lesssim B$ to mean $A \leq CB$ for a universal constant $C$. We use $A \sim B$ to mean $A \leq CB$ and $B \leq CA$ for a universal constant $C$. We use $A\sim_K B$ to mean $A \leq CKB$ and $B \leq CKA$ for a universal constant $C$. 

Throughout the paper we use the uniform convergence topology when dealing with Lipschitz functions. 

Let $P \subset [-1,1]^2$ be a set of points and $L$ a set of lines. Following \cite{CPZ}, we define a smoothed incidence count between $P$ and $L$ as follows. We construct a certain smooth function $\eta: \R \to [0,1]$ satisfying 
\begin{align*}
    \supp \eta &\subset [-1/2 - 1/10, 1/2+1/10] \\ 
    \eta(x) &= 1 \qquad\text{for $x\in [-1/2+1/10, 1/2-1/10]$}\\ 
    \int \eta &= 1. 
\end{align*}
The function $\eta$ is specially designed to make \Cref{thm:high_low_ineq}, the high-low inequality, work. 
We define a smoothed incidence count
\begin{equation}\label{eq:smoothed_incidence_count}
    I(w; P, L) = \sum_{(p, \ell) \in P\times L} \eta(d(p, \ell)/w).
\end{equation}
Given a phase-space set $\mb X \subset \Omega = [-1,1]^3$, we define 
\begin{align*}
    I(w; \mb X) = \sum_{\omega_1, \omega_2 \in \mb X} \eta(d(p_{\omega_1}, \ell_{\omega_2})/w).
\end{align*}
Here if $\omega = (a,b,c)$ is a phase space point, $p_{\omega} = (a,b)$ is the point defined by $\omega$ and $\ell_{\omega} = \{(a+t,b+ct)\, :\, t \in \R\}$ is the line through $p_{\omega}$ with slope $c$. We also use the normalized incidence count
\begin{align*}
    B(w\sep \mb X) = \frac{I(w\sep \mb X)}{w|\mb X|^2},
\end{align*}
which is normalized so that if $\mb X$ is chosen uniformly at random, $B(w; \mb X)\sim 1$ (for large enough $w$). 
In this paper we let $P[\mb X]$ denote the underlying point set of $\mb X$ counted with multiplicity and $L[\mb X]$ denote the underlying line set counted with multiplicity, so $I(w; \mb X) = I(w; P[\mb X], L[\mb X])$.

\subsection*{Acknowledgements}
Thanks to Tuomas Orponen and Damian D{\k a}browski for stimulating conversations about their recent paper \cite{dabrowski2024} and connections to the present work. Thanks to Henry Cohn and Larry Guth for helpful conversations during the writing of the paper. Thanks to Noah Kravitz for helpful comments about the exposition. Many thanks to an anonymous referee for detailed comments that improved the paper.

\section{Proof sketch}\label{sec:pf_sketch}
\subsection{Setup}
We sketch the proof of \Cref{thm:incidence_lower_bd}, the main result of our paper. 

Recall from \cref{subsec:intro_statement_main_thm} that $\Omega = [-1,1]^3$ is our phase space of point-line pairs and $\mb X \subset \Omega$ is a $(\delta, \alpha, \beta, C)$ set, meaning it is $\delta$-separated in $\Omega$ and satisfies the $(\alpha, \beta)$-Frostman condition 
\begin{align}\label{eq:frostman_cond_sketch}
    |\mb X\cap \mb R| \leq C u^{\alpha}w^{\beta} |\mb X|
\end{align}
for every $u\times uw\times w$ phase space rectangle $\mb R$ with $uw\geq \delta$.
The goal is to prove a lower bound on the normalized incidence count, $B(\delta \sep \mb X) \geq \delta^{\varepsilon}$.

\subsection{Two-step strategy for finding incidences}\label{subsec:two_step_strat_sketch}
To prove \Cref{thm:incidence_lower_bd} we will carefully choose a $u\times uw\times w$ phase space rectangle $\mb R$ such that we can execute an initial estimate and an inductive step for the blowup $\mb X^{\mb R}$. The rectangle $\mb R$ will be pretty large, $uw\sim \delta^{o(1)}$. The two steps are
\begin{itemize}
    \item \textbf{Initial estimate.} For some large initial scale $w_i \sim \delta^{o(1)}$, $B(w_i\sep \mb X^{\mb R}) \gtrsim 1$.
    \item \textbf{Inductive step.} By the high-low inequality, 
    \begin{align*}
        |B(w_i\sep \mb X^{\mb R}) - B(\delta/uw\sep \mb X^{\mb R})| \lesssim (\log 1/\delta) \sup_{w_i < w < \delta/uw} \Bigl(\frac{\ms M_{w\times w}(P[\mb X^{\mb R}])}{|P[\mb X^{\mb R}]|} \frac{\ms M_{1\times w}(L[\mb X^{\mb R}])}{|L[\mb X^{\mb R}]|}w^{-3}\Bigr)^{1/2}.
    \end{align*}
    If the right hand side is small, we find $B(\delta/uw\sep \mb X^{\mb R})\gtrsim 1$ as well. We will have $|\mb X^{\mb R}| \gtrsim \delta^{o(1)} |\mb X|$, so we find that $B(\delta \sep \mb X) \gtrsim \delta^{o(1)}$ as well. 
\end{itemize}
The first step of the proof is technical but crucial. Given a set $\mb X \subset \Omega$, we extract a subset ${\mb X'} \subset {\mb X}$ which is \textit{uniform}. This means that for every $u\times uw\times w$ phase space rectangle $\mb R$ centered at a point of $\mb X$, 
\begin{align*}
    |\mb X\cap \mb R| \sim r(u, w)
\end{align*}
for some function $r(u,w)$. 
In practice we can only ask for this property on a relatively sparse set of scales, for instance those of the form $\delta^{x}\times \delta^{x+y}\times \delta^{y}$ where $x,y \in m^{-1}\Z_{\geq 0}$. 
Actually, uniformity is a stronger property that involves a larger set of rectangles in $\Omega$. 

From now on we assume $\mb X$ is uniform. Let $|\mb X|_{u\times v\times w}$ be the minimal number of $u\times v\times w$ phase space rectangles needed to cover $\mb X$. 
We define the \textit{branching function} $f(x,y)$ of $\mb X$ to be the covering numbers on a log-log scale,
\begin{align*}
    \delta^{-f(x,y)} = |\mb X|_{\delta^{x}\times \delta^{x+y}\times \delta^{y}}.
\end{align*}
For uniform sets, the branching function completely describes how $\mb X$ concentrates in subrectangles. 
If $\mb R$ is a $\delta^{x}\times \delta^{x+y}\times \delta^y$ rectangle that intersects $\mb X$ significantly then the blowup $\mb X^{\mb R}$ has the branching function 
\begin{align*}
    f(x',y'\sep x, y) = f(x+x',y+y') - f(x,y). 
\end{align*}
Throughout the paper we use the notation $(\bullet \sep x, y)$ to denote the value of some quantity after blowing up into a $\delta^x\times \delta^{x+y}\times \delta^y$ rectangle.
The hypothesis that $\mb X$ is a $(\delta, \alpha, \beta, C)$-set implies 
\begin{align}\label{eq:branching_func_lower_bd_sketch}
    f(x,y) \geq \alpha x + \beta y - o(1)\qquad \text{for $x+y\leq 1$.}
\end{align}
Let $\slope(\mb X)$ be the set of slopes of lines in $L[\mb X]$. 
Uniformity of $\mb X$ implies that $\slope(\mb X)$ is uniform, meaning that for every interval $I \subset [-1,1]$ of length $w$ centered at a point in $\slope(\mb X)$, 
\begin{align*}
   |\slope(\mb X)\cap I| \sim r(w)
\end{align*}
for some function $r(w)$. 
Moreover, the same property holds after blowing up into any subrectangle. We define a direction number $d(t\sep x, y)$ so that 
\begin{align*}
    \delta^{-d(t\sep x, y)} &= |\slope(\mb X^{\mb R})|_{\delta^t}
\end{align*}
where $\mb X^{\mb R}$ is the blowup into a $\delta^{x}\times \delta^{x+y}\times \delta^y$ rectangle and $|\cdot|_{\delta^t}$ is the $\delta^t$-covering number. 
Again because of uniformity, these covering numbers are the same for every one of these blowups. 
The next claim gives an initial estimate in terms of the direction numbers of $\mb X$. 
\begin{claim*}[Initial estimate]\label{claim:direction_stability_initial_estimate}
We have 
\begin{align*}
    B(\delta^t\sep \mb X) \gtrsim \delta^{d(t\sep 0, 0) - d(t\sep t, 0)}.
\end{align*}
Moreover, let $x,y \in (0, 1/5)$. Let $\mb R$ be a $\delta^{x}\times \delta^{x+y}\times \delta^y$ rectangle in $\Omega$ which intersects $\mb X$ significantly, and let $\mb X^{\mb R}$ be the blowup into $\mb R$. Then 
\begin{align*}
    B(\delta^t\sep \mb X^{\mb R}) \gtrsim \delta^{d(t\sep x,y) - d(t\sep x+t,y)}.
\end{align*}
\end{claim*}

\begin{proof}[Proof sketch]
Let $p \in P[\mb X]$ and let $Q$ be the $\delta^t\times \delta^t$-square centered at $p$.
Let 
\begin{align*}
    S_{big} &= \{c \in \delta^t\Z\, :\, [c, c+\delta^t] \cap \slope(\mb X) \neq \emptyset\} \\ 
    S_{little} &= \{c \in \delta^t\Z\, :\, [c, c+\delta^t] \cap \slope(\mb X\cap Q) \neq \emptyset\}. 
\end{align*}
Think of $S_{big}$ as the slope set of $\mb X$ up to $\delta^t$ tolerance and $S_{little}$ as the slope set of $\mb X \cap \mb Q$ up to $\delta^t$ tolerance. We have $S_{little} \subset S_{big}$ and 
\begin{align*}
    |S_{big}| &\sim \delta^{-d(t\sep 0, 0)}, \\ 
    |S_{little}| &\sim \delta^{-d(t\sep t, 0)}.
\end{align*}
Select a line $\ell \in L[\mb X]$ uniformly at random, and round $\slope(\ell)$ down to the nearest value in $\delta^t\Z$. This rounded down slope is roughly uniformly distributed over $S_{big}$ because $\slope(\mb X)$ is uniform. The chance that the rounded down slope lands in $S_{little}$ is
\begin{align*}
    \frac{|S_{little}|}{|S_{big}|} = \delta^{d(t\sep 0, 0) - d(t\sep t, 0)}.
\end{align*}
If the slope does land in $S_{little}$, there is a $\gtrsim \delta^t$ chance that $\ell$ passes through $Q$. Overall, the probability that $\ell$ passes through $Q$ is $\gtrsim \delta^{t+d(t\sep 0, 0) - d(t\sep t, 0)}$. We can think of $B$ as 
\begin{align*}
    B(\delta^t) = \E_{p\in P} \frac{\Pr_{\ell\in L}[d(p, \ell) \leq \delta^t]}{\delta^t},
\end{align*}
so $B(\delta^t) \gtrsim \delta^{d(t\sep 0, 0) - d(t\sep t, 0)}$. 
\end{proof}

Let $\mb R$ be a $\delta^{x}\times \delta^{x+y}\times \delta^y$ rectangle which intersects $\mb X$ significantly, and let $\mb X^{\mb R}$ be the blowup. The high-low error of $\mb X^{\mb R}$ at scale $\delta^s$ is given by 
\begin{align*}
    \Bigl(\frac{\ms M_{\delta^s\times \delta^s}(P[\mb X^{\mb R}])}{|\mb X^{\mb R}|} \frac{\ms M_{1\times \delta^s}(L[\mb X^{\mb R}])}{|\mb X^{\mb R}|}\delta^{-3s}\Bigr)^{1/2} \lesssim \delta^{\frac{1}{2}[f(s, 0\sep x, y) + f(0, s\sep x, y) - 3s]}. 
\end{align*}
We are ready to combine the initial estimate and inductive step. Say a triple $(t\sep x, y)$ is \textit{direction stable} with parameter $\rho$ if 
\begin{align}\label{eq:direction_stable}
    d(t\sep x, y) - d(t\sep x+t,y) &\leq \rho t. 
\end{align}
If $(t\sep x, y)$ is direction stable, then after blowing up into a $\delta^x\times \delta^{x+y}\times \delta^y$ rectangle $B(\delta^t \sep \mb X^{\mb R}) \gtrsim \delta^{\rho t}$, giving a good initial estimate. 
Say a triple $(t;x, y)$ is \textit{high-low regular} with parameter $\rho$ if 
\begin{align} \label{eq:slope_controlled}
    \frac{1}{2}[f(s, 0\sep x, y) + f(0, s\sep x, y) - 3s] &\geq 2\rho t \qquad \text{for all $t \leq s \leq 1-x-y$} .
\end{align}
If a triple is high-low regular then the error term in the high-low estimate is $\lesssim \delta^{2\rho t}$ for all scales between $\delta^t$ and $\delta^{1-x-y}$. We say a triple is {\em effective} if it is direction stable and high-low regular with parameter $\rho = o(1)$. If a triple is effective then we can prove an incidence lower bound by blowing up into a $\delta^{x}\times \delta^{x+y}\times \delta^y$ rectangle, getting an initial estimate at scale $\delta^t$, and then using the high-low estimate to find lots of incidences at the bottom scale. 

The rest of the argument is proving there exists an effective triple, which is a problem about the Lipschitz functions $f(x,y)$ and $d(t;x, y)$. 
Here is all the information about $f$ and $d$ that we need. 
\begin{itemize}
    \item For $x', y' \geq 0$ we have 
    \begin{align*}
        0 \leq f(x+x',y+y') - f(x,y) \leq 2(x'+y') + o(1). 
    \end{align*}
    In other words, $f$ is monotone and Lipschitz. 

    \item For $t', x', y' \geq 0$, the direction function satisfies
    \begin{align*}
        0 \leq d(t+t'\sep x, y) - d(t\sep x, y) &\leq t' + o(1), \\ 
        0 \leq d(t\sep x, y)  - d(t\sep x+x', y) &\leq 2x' + o(1), \\ 
        |d(t\sep x, y+y') - d(t\sep x, y)| &\leq 2y' + o(1).
    \end{align*}
    In particular, $d$ is monotone with respect to $t$ and $x$ and Lipschitz with respect to all three parameters. 

    \item The direction function satisfies 
    \begin{align*}
        d(t+s\sep x, y) \geq d(t\sep x, y) + d(s\sep x, y+t) - o(1),
    \end{align*}
    see \Cref{fig:directions} for an illustration.

    \begin{figure}
    \centering
    \includegraphics[width=0.4\linewidth]{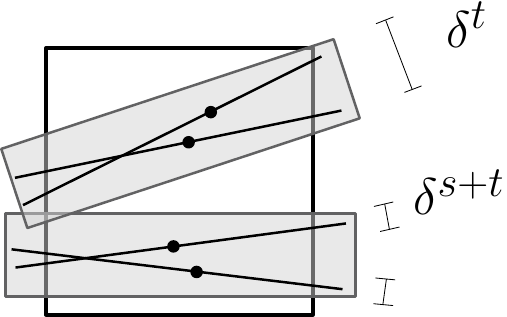}
    \caption{Each $1\times \delta^t$ rectangle contributes $\delta^{-d(s\sep 0, t)}$ directions on scale $\delta^{t+s}$. Since there are at least $\delta^{-d(t\sep 0,0)}$ many $1\times \delta^t$-tubes with non-overlapping directions, this gives a lower bound on the total number of directions.}
    \label{fig:directions}
\end{figure}
    
    \item For technical reasons, we also need to consider the dual direction function $d^\vee(t\sep x, y)$ which satisfies the same properties as $d(t\sep x, y)$ but with the roles of $x$ and $y$ swapped. This function has a simple meaning: $d^\vee(t\sep 0,0)$ is the $\delta^t$-covering number of the set of $x$-coordinates of $P[\mb X]$.
\end{itemize}
Using these properties the goal is to find some triple $(t\sep x, y)$ that is direction stable \eqref{eq:direction_stable} and high-low regular \eqref{eq:slope_controlled}.

In reality $f = f(x,y,z)$ is a function of three variables rather than just two variables, although we ignore the $z$ variable in this section. The $z$ variable contains direction information and $d$ is defined in terms of $f(x,y,z)$. In addition to being monotonic and Lipschitz, $f$ satisfies a \textit{submodularity} property as well. All the properties of $f$ and $d$ listed above follow from combining monotonicity, Lipschitz, and submodularity of $f$. 

Our branching function analysis is delicate, so before launching into the argument we describe two important examples of $(\delta, \alpha, \beta, C)$-sets.

\subsection{Examples of $(\delta, \alpha, \beta, C)$-sets}\label{section-examples}

\begin{enumerate}
    \item First we need a definition. A set $P \subset [-1,1]^2$ is  $(\delta, t, C)$-regular if it is $\delta$-separated and satisfies the Ahlfors-David regularity condition
    \begin{equation*}
        C^{-1} w^{t} |P| \leq |P \cap B_w(x_0)| \leq C w^{t} |P| \qquad \text{for all $x_0 \in P$ and $w\in [\delta, 1]$.}
    \end{equation*}
    Let $P$ be a $(\delta, t, C)$-regular set and for each $p \in P$, let $L_p$ be a $(\delta, s, C)$-set of lines. Assume $t+s > 2$ and let $\mb X \subset \Omega$ be the corresponding set in phase space, after passing to a subset we may assume $\mb X$ is uniform. 
    
    Because of Ahlfors-David regularity,  
    \begin{align*}
        f(x,0) &= tx + o(1). 
    \end{align*}
    By the $s+t>2$ regime of the Furstenberg set estimate (see \cref{subsec:limiting_branching_func} \cref{exampleitem:info_L_furstenberg}), the number of $\delta$-tubes needed to cover $L$ is $\gtrsim \delta^{-s-1}$. Because we assumed $P$ is Ahlfors-David regular, we can apply this estimate inside of any $\delta^x\times \delta^x$ sub-square. It follows that for any $x+y\leq 1$, 
    \begin{align*}
        f(x,y) &\geq t x + (s+1)y - o(1).
    \end{align*}
    This branching function lower bound combined with uniformity implies that $\mb X$ is $(\alpha, \beta)$-Frostman with constant $\delta^{-\varepsilon}$ (meaning $\mb X$ satisfies \eqref{eq:frostman_cond_sketch}). Because $\mb X$ is $\delta$-separated, it is a $(\delta, t, s+1, \delta^{-\varepsilon})$-set. 

    \item \label{example2} Let $\varphi: [-1,1]^2 \to [-1,1]$ be a smooth function with $|(1, \varphi(x,y)) \cdot \nabla \varphi(x,y)| \sim 1$ everywhere. Let 
    \begin{align*}
       \mb X = \{(x,y,\varphi(x,y))\, :\, (x,y) \in \delta \Z^2 \cap [-1,1]^2\}. 
    \end{align*}  
    The underlying point set is a $\delta$-grid, and the line through the point $(x, y)$ has slope $\varphi(x,y)$. 
    We claim 
    \begin{align*}
        f(x,y) = \max\{2x+y,x+2y\} + o(1),
    \end{align*}
    so $\mb X$ is a $ (\delta, \alpha, \beta, C)$-set for any $\alpha, \beta\ge 1$ with $\alpha+\beta=3$. In particular it is a $(\delta, \frac{3}{2}, \frac{3}{2}, C)$-set which is the case we focus on. 
    \bigskip\\
    Here is a quick justification. For concreteness we will take $\varphi(x,y) = x$. Let $\omega = (x_0, y_0, x_0)$ be fixed, and let $\mb R$ be the $u\times uw\times w$ phase space rectangle centered at $\omega$ (see \eqref{eq:intro_phase_space_rect}). A point $(x_1, y_1, x_1)$ lands in $\mb R$ if 
    \begin{align*}
        |x_0 - x_1| &\leq \min\{u,w\} &&\text{so the $x$-coordinate and slopes are close enough,}\\ 
        |y_1 - (y_0 + x_0(x_1-x_0)) | &\leq uw && \text{so the vertical height is close enough.}
    \end{align*}
    There are $\sim \min\{u^2w,uw^2\}\delta^{-2}$ many points in $\mb X$ satisfying these constraints, so the covering number by these rectangles is $\sim \max\{u^{-2}w^{-1}, u^{-1}w^{-2}\}$. 
\end{enumerate}

The above examples are important and they motivate our analysis of branching functions. However, it turns out that for branching functions that are either smooth, piecewise linear, or have other natural regularity properties (with fixed regularity parameters as $\delta \to 0$), it is much easier to prove \Cref{thm:incidence_lower_bd}. A challenge in the analysis is that Lipschitz functions can be very rough, and we do not have examples describing this difficulty. 

\subsection{The analysis of Lipschitz functions}\label{subsec:pf_sketch_analysis_lip}
The examples described above motivate a definition. Say $(x,y)$ is an $(\alpha, \beta)$-good scale for the branching function $f$ if 
\begin{align*}
    f(x',y'\sep x,y) = f(x+x',y+y') - f(x,y) \geq \alpha x' + \beta y'. 
\end{align*}
In other words, the blowup of $\mb X$ into a $\delta^{x}\times \delta^{x+y}\times \delta^y$ rectangle is $(\alpha, \beta)$-Frostman. In the first example, $(x, 0)$ is a $(t, s+1)$-good scale for every $x \in [0, 1]$. In the second example, $(x, x)$ is a $(\frac{3}{2}, \frac{3}{2})$-good scale for every $x \in [0,1/2]$. 
Both examples have a whole line of good scales. It turns out that in general we can reduce to that case, see \cref{sec:reduce_lip_func_good_line}. 

\parag{Case 1: There is a horizontal line of good scales.}
Suppose that $(x,0)$ is an $(\alpha, \beta)$-good scale for every $x\in [0,1]$. 
Let $\rho \le (\alpha + \beta - 3)/50$ and $t = \rho/50$. Let $m = \lfloor 0.1/t\rfloor $, and consider the telescoping sum 
\begin{align*}
    (d(t\sep 0, 0) - d(t\sep t, 0)) + \dots + (d(t\sep (m-1)t, 0) - d(t\sep mt, 0)) = d(t\sep 0, 0) - d(t\sep mt, 0).
\end{align*}
Each of the $m$ summands on the left is nonnegative, and the right hand side is $\leq t$. So there must be some $0 \leq j \leq m-1$ so that $d(t\sep jt, 0) - d(t\sep (j+1)t, 0) \leq \frac{1}{m} t < \rho t$. The triple $(t\sep jt, 0)$ is direction stable with parameter $\rho$, and it is also high-low regular with parameter $\rho$ because of the assumption that $(jt, 0)$ is an $(\alpha,\beta)$-good scale. Thus $(t; jt, 0)$ is effective, as desired.

\parag{Case 2: There is a diagonal line of good scales.} Suppose that $(x,x)$ is an $(\alpha, \beta)$-good scale for every $x\in [0,1/2]$. 
In the last case we searched for an effective triple based at a good scale, so high-low regularity came for free. The current case is harder because we need to locate a scale that is both direction stable and high-low regular. Indeed, in Example \ref{section-examples} (\ref{example2}), every $(\frac{3}{2}, \frac{3}{2})$-good scale $(x, x)$ is high-low regular but not direction stable (for any $t$). 
We look at candidate triples of the form $(t\sep x,x+t)$, which are parametrized by 
\begin{align*}
    \ms B = \{(x,t)\in [0,0.1]\times (0,0.1]\}. 
\end{align*}
Let $\rho > 0$ be a small constant and let $E_{hlr}$ be the points in $\ms B$ corresponding to a high-low regular triple, $E_{ds}$ the points corresponding to a direction stable triple. Given $(x,t) \in \ms B$, recall that $(t; x,x+t)$ is high-low regular with parameter $\rho$ if 
\begin{align*}
    \frac{1}{2}[f(s, 0\sep x, x+t) + f(0, s\sep x, x+t) - 3s] &\geq 2\rho t \qquad \text{for all $t \leq s \leq 1-x-(x+t)$},
\end{align*}
and $(t; x,x+t)$ is direction stable if 
\begin{align*}
    d(t\sep x, x+t) - d(t\sep x+t,x+t) &\leq \rho t. 
\end{align*}
The key idea is to consider the measure 
\begin{align*}
    \mu = \frac{1}{t}\,dx\,dt 
\end{align*}
on $\ms B$. This measure integrates to infinity and is concentrated near $t = 0$. For each $\varepsilon > 0$, let  
\begin{align*}
   \mu_{\varepsilon} = \frac{1_{\{(x,t)\, :\, t \geq \varepsilon\}}}{\mu(\{(x,t)\, :\, t \geq \varepsilon\})} \mu 
\end{align*}
be the normalized restriction of $\mu$ to $\{t \geq \varepsilon\}$. 
We prove that 
\begin{itemize}
    \item For small enough $\rho$, $E_{hlr}$ has positive density under the measure $\mu$. This means that $\mu_{\varepsilon}(E_{hlr}) \geq c - o_{\varepsilon}(1)$ for some constant $c > 0$. 
    \item For all $\rho$, $E_{ds}$ has density one under the measure $\mu$. This means that $\mu_{\varepsilon}(E_{ds}) \geq 1 - o_{\varepsilon}(1)$. 
\end{itemize}
Combining these two facts proves $E_{hlr}$ and $E_{ds}$ have an intersection point for some choice of $\rho$. 

We prove $E_{hlr}$ has positive density one $x$-coordinate at a time. If 
\begin{align}\label{eq:slope-assumption}
    f(0, t\sep x,x) = s_x\,  t 
\end{align}
was a linear function of $t$ for each $x$, then every $(x,t) \in \ms B$ would be high-low regular. Actually, this is the motivation for considering triples of the form $(t\sep x, x+t)$. We must have $s_x \geq \beta$ because we assumed $(x,x)$ is an $(\alpha, \beta)$-good point. We have 
\begin{align*}
    f(t, 0\sep x,x+t) + f(0, t\sep x,x+t) - 3t &= f(t,t\sep x, x) - 2 f(0, t\sep x, x) + f(0, 2t\sep x, x) - 3t \\ 
    &\geq (\alpha+\beta-3)t.
\end{align*}
where the second and third terms cancel because we assumed $f(0, t\sep x, x)$ is a linear function in $t$ and we use that $(x,x)$ is an $(\alpha,\beta)$-good point to get $f(t,t;x,x) \ge (\alpha + \beta)t$. 
For $t' \geq t$ we can estimate the high-low error using the hypothesis that $(x+t,x+t)$ is a good point,
\begin{align*}
    [f(t', 0\sep x,x+t) + f(0, t'\sep x,x+t)] - [f(t, 0\sep x,x+t) + f(0, t\sep x,x+t)] &= \\
    = f(t'-t,0; x+t,x+t) + f(0,t+t'; x,x) - f(0,2t; x,x) &\geq (\alpha +\beta)(t'-t).
\end{align*}
This shows that under the simplifying assumption (\ref{eq:slope-assumption}) any triple of the form $(t; x, x+t)$ is high-low regular. 
For general $f$, we can write 
\begin{align*}
    f(0, t\sep x, x) = s_x(t)t
\end{align*}
where the slope $s_x$ now depends on the point $t$. 
Say $(x, t)$ is \textit{slope minimal} if 
\begin{align*}
   s_x(t+t') \geq s_x(t) - \gamma \qquad \text{for all $t \leq t' \leq H t$}
\end{align*}
where $\gamma, H$ are parameters. The computation above implies that if $(x,t)$ is slope minimal then $(x,t')\in E_{hlr}$ for $t\leq t' \leq (1+\eta)t$ and some $\eta >0$. By taking successive minima of the function $s_x(t)$, it is easy to find lots of slope minimal scales, and overall we get a lower bound on the $\mu$-density of $E_{hlr}$. The main takeaway is that our choice of candidate scales along with the hypothesis that $(x, x)$ is a good line lets us find high-low regular triples one $x$-coordinate at a time. 

We move on to proving $E_{ds}$ has density one. 
We show 
\begin{align*}
   \int_{\ms B} t^{-1} (d(t\sep x,x+t) - d(t\sep x+t,x+t))\, d\mu_{\varepsilon} = o_{\varepsilon}(1). 
\end{align*}
The integral on the left hand side is $\geq \rho\, \mu_{\varepsilon}(E_{ds}^c)$, so the estimate implies $\mu_{\varepsilon}(E_{ds}^c) = o_{\varepsilon}(1)$ and $E_{ds}$ has $\mu$-density $1$. In order to estimate this integral we have to make use of some cancellation. 
To do so we use the crucial direction set inequality 
\begin{align*}
    d(t+s\sep x, y) \geq d(t\sep x, y) + d(s\sep x,y+t).
\end{align*}
With this inequality in hand, we can estimate 
\begin{align*}
    d(t\sep x,x+t) - d(t\sep x+t,x+t) \leq d(2t\sep x, x) - d(t\sep x, x) - d(t\sep x+t,x+t). 
\end{align*}
The right hand side has the advantage that it only involves direction numbers on the good line.  We have 
\begin{align*}
    \int_{\ms B} t^{-1} (d(2t\sep x, x) - d(t\sep x, x) - d(t\sep x+t,x+t)) d\mu_{\varepsilon} &\leq \int_{\ms B} t^{-1} (d(2t\sep x, x) - 2d(t\sep x, x)) d\mu_{\varepsilon} + o_{\varepsilon}(1) \\ 
    &\leq o_{\varepsilon}(1).
\end{align*}
The first inequality comes from integrating out the $x$-coordinate, and the second inequality comes from integrating out the $t$-coordinate and using change of variables.

\parag{General good lines.}
The argument we described in Case 2 works whenever there is a line of good scales $\{(x, mx)\, :\, 0 \leq x \leq 1/10\}$ with slope $m \in (0, 1]$. In order to deal with good lines that have slope $> 1$ we use point-line duality. The dual notion for the slope set of $L[\mb X]$ is the set of $x$-coordinates of $P[\mb X]$. By replacing lines with points and slopes with $x$-coordinates, the argument works in the same way. 

\subsection{The space of branching functions}
In recent years, careful analysis of branching functions has become a significant tool in fractal geometry. An early example of this method can be found in Keleti and Shmerkin's~\cite{KeletiShmerkin} paper on the Falconer distance problem. While prior work has used branching functions of a single variable to study sets of points, it is new to apply branching functions of multiple variables to study sets of point/line pairs.

Our branching functions are Lipschitz functions of three variables. We define the set of functions $\mc L$ to be limits of branching functions of phase-space sets, see \cref{subsec:limiting_branching_func}. The functions in $\mc L$ have a few basic properties: monotonicity, Lipschitz regularity, and submodularity. Our main Lipschitz functions result, \Cref{thm:lip_func_theorem}, is a fact about $\mc L$ that just uses these properties. Deeper facts from continuous incidence geometry, such as the Furstenberg set estimate, can also be stated as information about the space $\mc L$. We believe many interesting statements and conjectures in continuous two dimensional incidence geometry can be phrased as information about $\mc L$ and we hope this will be a useful framework for further study.

\section{The space of point-line pairs}\label{sec:space_pt_line_pairs}

\subsection{Metrics}
Let $\Omega = [-1,1]^3$. 
Given $\mb \omega = (a, b, c) \in \Omega$ we let $p_\omega = (a, b)$ and $\ell_\omega = \{ (a+t,b+ct) \, :\, ~t \in \R \}$. We think of $\Omega$ as the space of point-line pairs $p\in \ell$ where $\ell$ makes a slope $c \in [-1,1]$ with the $x$-axis. 
Given a subset $\mb X \subset \Omega$ we associate to it the set of points $P[\mb X] = \{p_\omega\}_{\omega \in \mb X}$ and the set of lines $L[\mb X] = \{\ell_\omega\}_{\omega \in \mb X}$ (if there are collisions, we treat these as multisets). 

Our plan is to study how subsets in $\Omega$ interact with rectangles ${\mb R}$ of various aspect ratios $u\times v\times w$, their covering, rescaling and uniformity properties. A convenient approach to doing this is to introduce a family of metrics $d_{u\times v\times w}$ on $\Omega$ so that rectangles ${\mb R}$ correspond to unit balls in these metrics. This allows us to prove geometric properties of these rectangles via direct computations with the corresponding metrics.

Say a triple of scales $u\times v\times w$ is \textit{admissible} if $v \geq uw$. In what follows we denote $\mb \omega_j = (a_j, b_j, c_j)$. 
For any admissible $u\times v\times w$ define an asymmetric distance function on $\R^3$ by 
\begin{equation}\label{eq:dist_func_defn}
    d_{u\times v\times w}(\mb \omega_0 \to \omega_1) = \max\Bigl\{\frac{|a_0-a_1|}{u},  \frac{|b_1 - (b_0+c_0(a_1-a_0))|}{v}, \frac{|c_0-c_1|}{w}\Bigr\}. 
\end{equation}
When the triple $u\times v\times w$ is clear from the context we simply write $d(\mb \omega_0 \to \omega_1)$. Notice that $d$ is homogeneous in the sense that 
\begin{equation}\label{eq:dist_homogeneity}
    d_{u\times v\times w}(\omega\to (\omega + \lambda \omega')) = |\lambda|\, d_{u\times v\times w}(\omega \to (\omega + \omega')),\qquad \lambda \in \R.
\end{equation} 
Sometimes we use the notation $d_{1\times \square \times \square}$, $d_{1\times 1\times \square}$, etc. Parse the $\square$ as $\infty$ in \eqref{eq:dist_func_defn}, or in other words, take the $\max$ in \eqref{eq:dist_func_defn} over a subset of the three quantites. For instance, 
\begin{align*}
    d_{1\times \square\times \square}(\omega_0\to \omega_1) &= |a_0-a_1| = \text{$x$-coordinate distance between $p_{\omega_0}$ and $p_{\omega_1}$}, \\ 
    d_{\square\times \square\times 1}(\omega_0\to \omega_1) &= |c_0-c_1| = \text{slope distance between $\ell_{\omega_0}$ and $\ell_{\omega_1}$}\\ 
    d_{\square\times 1\times \square}(\omega_0\to \omega_1) &= |b_1 - (b_0+c_0(a_1-a_0))| = \text{height from the line $\ell_{\omega_0}$ to the point $p_{\omega_1}$}.
\end{align*}
Usefully, $d_{1\times \square \times \square }$ and $d_{\square\times \square \times 1}$ are symmetric and transitive. On the other hand $d_{\square\times 1\times \square}$ does not satisfy symmetry or transitivity. However we have the following substitutes. This lemma uses admissibility of a triple $u \times v\times w$. 
\begin{lemma}\label{lem:approx_sym_trans}\leavevmode
\begin{enumerate}[label=(\roman*) ]
    \item We have an approximate transitivity property for $d_{\square\times1\times\square}$,
    \begin{align*}
    d_{\square\times 1\times \square}(\omega_0 \to \omega_2) &\leq d_{\square\times 1\times \square}(\omega_0 \to \omega_1) + d_{\square\times 1\times \square}(\omega_1\to \omega_2) + \\ 
    &\qquad d_{\square\times \square\times 1}(\omega_0\to \omega_1)d_{1\times \square\times \square}(\omega_1\to \omega_2).
    \end{align*}
    As a consequence, for any admissible scale $u\times v\times w$,
    \begin{equation*}
        d(\omega_0 \to \omega_2) \leq d(\omega_0 \to \omega_1) + d(\omega_1\to \omega_2) + d(\omega_0 \to \omega_1)d(\omega_1\to \omega_2).
    \end{equation*}

    \item There is an approximate symmetry property for $d_{\square\times1\times\square}$
    \begin{equation*}
        |d_{\square\times 1\times \square}(\omega_0 \to \omega_1) - d_{\square\times 1\times \square}(\omega_1\to \omega_0)| \leq d_{1\times \square\times \square}(\omega_0 \to \omega_1) d_{\square\times \square\times 1}(\omega_1\to \omega_0). 
    \end{equation*}
    As a consequence, for any admissible scale $u\times v\times w$,
    \begin{equation*}
        |d(\mb \omega_0 \to \mb \omega_1) - d(\mb \omega_1 \to \mb \omega_0)| \leq d(\mb \omega_0 \to \mb \omega_1)^2.
    \end{equation*}
\end{enumerate}
\end{lemma}
\begin{proof}\leavevmode
\begin{enumerate}[label=(\roman*)]
    \item Approximate transitivity follows from the identity 
    \begin{align*}
    [b_2 - (b_0 + c_0(a_2 -a_0))] - [b_2 - (b_1 + c_1(a_2 - a_1)) + (b_1 - (b_0 + c_0(a_1 - a_0)))] \\
    = (c_1-c_0)(a_2-a_1).
    \end{align*}

    It follows that 
    \begin{align*}
        v^{-1} d_{\square\times 1\times \square}(\omega_0\to \omega_2) &\leq v^{-1}d_{\square\times 1\times \square}(\omega_0\to \omega_1) + v^{-1} d_{\square\times 1 \times \square}(\omega_1\to \omega_2) + \\ 
        &\qquad v^{-1}d_{\square\times \square\times 1}(\omega_0\to \omega_1)d_{1\times \square\times \square}(\omega_1\to \omega_2) \\ 
        &\leq d_{u\times v\times w}(\omega_0 \to \omega_1) + d_{u\times v\times w}(\omega_1 \to \omega_2) + \frac{uw}{v} d_{u\times v\times w}(\omega_0 \to \omega_1)d_{u\times v\times w}(\omega_1 \to \omega_2).
    \end{align*}
    If $u\times v\times w$ is admissible then $uw/v \leq 1$ and the second property follows.

    \item Approximate symmetry follows from the identity
    \begin{align*}
        (b_1 - (b_0+c_0(a_1-a_0))) + (b_0 - (b_1 + c_1(a_0-a_1))) &= (c_1-c_0)(a_1-a_0).
    \end{align*}

    If follows that 
    \begin{align*}
        v^{-1}|d_{\square\times 1\times \square}(\omega_0\to \omega_1) - d_{\square\times 1\times \square}(\omega_1\to \omega_0)| \leq  \frac{uw}{v} d_{u\times v\times w}(\omega_0\to \omega_1)^2
    \end{align*}
    and if $u\times v\times w$ is admissible the second property follows. 
\end{enumerate}
\end{proof}
Notice that 
\begin{equation}\label{eq:dist_rel_euclidean}
    d_{u\times v\times w}(\omega_0\to \omega_1) \leq \max\{u^{-1},2v^{-1},w^{-1}\} d_{\R^3}(\omega_0, \omega_1)
\end{equation}
and 
\begin{align*}
    d_{1\times 1\times 1}(\omega_0\to \omega_1) \sim d_{\R^3}(\omega_0, \omega_1)
\end{align*}
where $d_{\R^3}$ is the Euclidean distance. 
We can evaluate the point and line metrics in terms of the space $\Omega$:
\begin{equation}\label{eq:metric_compare_Omega}
\begin{split}
    d(p_{\omega_0}, p_{\omega_1}) &\sim d_{1\times 1\times \square}(\omega_0 \to \omega_1), \\
    d(\ell_{\omega_0}, \ell_{\omega_1}) &\sim d_{\square\times 1\times 1}(\omega_0 \to \omega_1),
\end{split}
\end{equation}
It follows from \eqref{eq:dist_rel_euclidean} that for any $u,w  > 0$ we have 
\begin{equation}\label{eq:metrics_pt_line_Omega}
\begin{split}
    \frac{1}{2u}d(p_{\omega_0}, p_{\omega_1}) &\leq d_{u\times u\times 1}(\omega_0 \to \omega_1) \leq \frac{2}{u}d(p_{\omega_0}, p_{\omega_1}) + 2,\\ 
    \frac{1}{2w}d(\ell_{\omega_0}, \ell_{\omega_1}) &\leq d_{1\times w\times w}(\omega_0 \to \omega_1) \leq \frac{2}{w}d(\ell_{\omega_0}, \ell_{\omega_1}) + 2.
\end{split}
\end{equation}

\subsection{Incidences}
For $\mb X \subset \Omega$, the incidence count $I(w; \mb X)$ is defined in \cref{subsec:notation}.
For $p\in P$ and $\ell \in L$, $d(\ell, p)$ denotes the Euclidean distance. On the other hand, $d_{\square\times 1\times \square}(\omega_1 \to \omega_2)$ is the vertical distance from the line $\ell_{\omega_1}$ to the point $p_{\omega_2}$. So
\begin{equation}\label{eq:incidence_value_metric}
    d(\ell_{\omega_1}, p_{\omega_2}) = d_{\square\times 1\times \square}(\omega_1\to \omega_2)(1+\mathrm{slope}(\ell_{\omega_1})^2)^{-1/2}.
\end{equation}
Because we restrict to slopes in $[-1,1]$, the Euclidean distance is off from the vertical distance by at most a factor of $\sqrt{2}$.

\subsection{Covering numbers}
For $\mb \omega \in \R^3$ and $u\times v\times w$ admissible we let 
\begin{align*}
    \mb R_{u\times v\times w}(\omega) &= \{\omega'\, :\, d_{u\times v\times w}(\omega\to \omega') \leq 1\}.
\end{align*}
Because of homogeneity,
\begin{align*}
    \lambda \cdot \mb R_{u\times v\times w}(\omega) &= \{\omega'\, :\, d_{u\times v\times w}(\omega\to \omega') \leq \lambda\}
\end{align*}
where the left hand side is the dilation in $\R^3$ centered at $\omega$. 
We use the label $\mb R$ because these sets are skewed rectangles: 
\begin{align*}
    \mb R_{u\times v\times w}(\omega) = \{(a_0 + t, b_0 + c_0t + r, c_0+s)\, :\, (t,r,s)\in [-u,u]\times [-v,v]\times [-w,w]\}.
\end{align*}

If $\lambda \geq 1$ and $\lambda \cdot \mb R_{u\times v\times w}(\omega_0)\cap \lambda \cdot \mb R_{u\times v\times w}(\omega_1) \neq \emptyset$ then $d(\omega_0 \to \omega') \leq \lambda$ and $d(\omega_1 \to \omega') \leq \lambda$ for some $\omega'$, so by approximate transitivity and symmetry (Lemma \ref{lem:approx_sym_trans})
\begin{equation}\label{eq:intersecting_rects}
\begin{split}
    d(\omega_0\to \omega_1) &\leq d(\omega_0\to \omega') + d(\omega'\to \omega_1)  + d(\omega_0\to \omega')d(\omega'\to \omega_1) \\ 
    &\leq d(\omega_0\to \omega') + d(\omega_1\to \omega')+ d(\omega_1\to \omega')^2  + d(\omega_0\to \omega')d(\omega_1\to \omega') + d(\omega_0\to \omega')d(\omega_1\to \omega')^2 \\ 
    &\leq 5\lambda^3.
\end{split}
\end{equation}
For any admissible $u,v,w \in 2^{-\N}$ consider the collection of dyadic centers and dyadic rectangles 
\begin{align*}
    \mc C_{u\times v\times w} &= \{(n_1u, n_2 v, n_3 w),\, n_j \in 2^{-10}\Z\}, \\ 
    \mc R_{u\times v\times w} &= \{\mb R_{u\times v\times w}(\omega)\, :\, \omega \in \mc C_{u\times v\times w}\}.
\end{align*}
We let $\mc R$ be the union of $\mc R_{u\times v\times w}$ over admissible and dyadic $u\times v\times w$. 

\begin{lemma}
For any admissible $u\times v\times w$, the collection $\mc R_{u\times v\times w}$ covers $\R^3$ with bounded overlap. 
\end{lemma}
\begin{proof}
First we show that the collection covers $\R^3$. Let $(a,b,c) \in \R^3$. Choose $a_0 \in 2^{-10}u\Z$ and $c_0 \in2^{-10}w\Z $ so that $|a_0 - a| \leq 2^{-10}u$ and $|c_0 - c| \leq 2^{-10}w$. Then choose $b_0 \in 2^{-10}v\Z $ so that $|b_0 - (b+c(a_0-a))| \leq 2^{-10}v$. For this choice, 
\begin{align*}
    d((a,b,c)\to (a_0, b_0, c_0)) \leq 2^{-10},
\end{align*}
so by approximate symmetry $d((a_0, b_0, c_0)\to (a,b,c)) \leq 2^{-9}$ and $(a,b,c) \in \mb R(a_0, b_0, c_0)$. 

Also by approximate symmetry the collection has bounded overlap because 
\begin{align*}
    \#\{(a_0, b_0, c_0) \in \mc C_{u\times v\times w}\, :\, d((a_0,b_0,c_0) \to (a,b,c)) \leq 1\} &\leq \\ 
    \#\{(a_0, b_0, c_0) \in \mc C_{u\times v\times w}\, :\, d((a,b,c) \to (a_0,b_0,c_0)) \leq 2\} \leq 2^{50}
\end{align*}
where in the last line we think about choosing the $a_0, c_0$ coordinates first and the $b_0$ coordinate after. 
\end{proof}

For a subset $\mb X \subset \R^3$, let $\mc R_{u\times v\times w}(\mb X)$ be the set of $\mb R \in \mc R_{u\times v\times w}$ which intersect $\mb X$. For $\mb X \subset \Omega$ and $u\times v\times w$ an admissible dyadic scale define the covering number 
\begin{align*}
    |\mb X|_{u\times v\times w} = |\mc R_{u\times v\times w}(\mb X)|. 
\end{align*}
Similarly, we define the concentration number 
\begin{align*}
    \ms M_{u\times v\times w}(\mb X) = \max_{\mb R \in \mc R_{u\times v\times w}} |\mb X \cap\mb R|. 
\end{align*}
Recall that $|P|_w$ is the $w$-covering number in the point metric and $\ms M_{w\times w}(P)$ is the maximum number of points in a $w$-ball. Because the point metric on $P[\mb X]$ is comparable to the $d_{w\times w\times 1}$ distance on $\mb X$ and the line metric on $L[\mb X]$ is comparable to the $d_{1\times w\times w}$ distance (see \eqref{eq:metrics_pt_line_Omega}), we have 
\begin{equation}\label{eq:covering_num_comp_PL}
\begin{split}
    |P[\mb X]|_w \sim |\mb X|_{w\times w\times 1},\qquad \ms M_{w\times w}(P[\mb X]) \sim \ms M_{w\times w\times 1}(\mb X), \\ 
    |L[\mb X]|_w \sim |\mb X|_{1\times w\times w},\qquad \ms M_{1\times w}(L[\mb X]) \sim \ms M_{1\times w\times w}(\mb X).
\end{split}
\end{equation}

\subsection{Rescaling}
Given a rectangle $\mb R = \mb R_{u_0\times u_0w_0\times w_0}(\omega_0)$ we consider the affine rescaling map 
\begin{align*}
    \psi_{\mb R}(a,b,c) = \Bigl(\frac{a-a_0}{u_0}, \frac{b - (b_0 + c_0(a - a_0))}{u_0w_0}, \frac{c-c_0}{w_0}\Bigr),
\end{align*}
which gives a bijective map:
\begin{align*}
    \psi_{\mb R}:\, \mb R_{u_0\times u_0w_0\times w_0}(\omega_0) \mapsto \Omega.
\end{align*}

In the next lemma, we show that our rectangle formalism behaves well with respect to rescaling maps. In fact, this was one of the main motivations for our setup.

\begin{lemma}\label{lem:rescaling_isometry}
For any $u,v,w\in 2^{-\N}$ we have the isometry property
\begin{align*}
    d_{u\times v\times w}(\psi_{\mb R}(\omega_1) \to \psi_{\mb R}(\omega_2)) = d_{uu_0\times vu_0w_0\times ww_0}(\omega_1\to \omega_2). 
\end{align*}
As a consequence, the map $\psi_{\mb R}^{-1}$ induces a bijection 
\begin{align*}
    \psi_{\mb R}^{-1}\, :\, \mc R_{u\times v\times w} \to \mc R_{uu_0\times vu_0w_0\times ww_0}.
\end{align*}
\end{lemma}
\begin{proof} 
The most complicated term is $d_{\square\times 1\times \square}(\psi_{\mb R}(\omega_1), \psi_{\mb R}(\omega_2))$, which we compute as
\begin{align*}
    (u_0w_0)^{-1}[b_2 - (b_0 + c_0(a_2 - a_0)) - (b_1 - (b_0 + c_0(a_1 - a_0)))] - u_0^{-1}w_0^{-1}(c_1-c_0)(a_2-a_1) =\\ 
    (u_0w_0)^{-1}[b_2 - b_1 + c_0(a_1 - a_2) - (c_1-c_0)(a_2-a_1)] = \\ 
    (u_0w_0)^{-1}[b_2 - b_1 - c_1(a_2-a_1)].
\end{align*}
For the second part, it is easy to check that $\psi_{\mb R}^{-1}$ bijectively maps 
\begin{align*}
    \psi_{\mb R}^{-1}: \mc C_{u\times v\times w} \to \mc C_{uu_0\times vu_0w_0\times ww_0}
\end{align*}
giving the result. 
\end{proof}

\subsection{Basic covering properties}
We say a set $\mc P \subset \Omega$ is \textit{$\lambda$-separated} for the distance $d = d_{u\times v\times w}$ if $d(\omega \to \omega') \geq \lambda$ for all distinct $\omega , \omega' \in \mc P$.
We say $u_0\times v_0\times w_0\leq u_1\times v_1\times w_1$ if $u_0\leq u_1$, $v_0\leq v_1$, and $w_0 \leq w_1$. This relation corresponds to domination between the metrics:
\[
d_{u_1\times v_1\times w_1}(\omega \to \omega') \le d_{u_0\times v_0\times w_0}(\omega \to \omega').
\]

\begin{lemma}\label{lem:covering_props}\leavevmode
\begin{enumerate}[label=(\roman*)]
    \item Let $u_0\times v_0\times w_0\leq u_1\times v_1\times w_1$ be admissible scales. Let $\mb R = \mb R_{u_1\times v_1\times w_1}(\omega)$ with $\omega \in 10\Omega$. Then 
    \begin{align*}
        |\mb R|_{u_0\times v_0\times w_0} &\lesssim \frac{u_1v_1w_1}{u_0v_0w_0}. 
    \end{align*} 
    Moreover, for $\lambda \ge 1$
    \begin{align*}
        |\lambda\cdot \mb R|_{u_0\times v_0\times w_0} &\lesssim \lambda^3\frac{u_1v_1w_1}{u_0v_0w_0}. 
    \end{align*}
    \label{lemitem:covering_bd}

    \item Let $\mb X\subset \Omega$. Let $\lambda \geq 1$ and $\mc P_{\lambda}$ be a maximal $\lambda$-separated in the metric $d_{u\times v\times w}$ set of points in $\mb X$. Then 
    \begin{align*}
        |\mb X|_{u\times v\times w} &\sim_{\lambda^6} |\mc P_{\lambda}|. 
    \end{align*}
    \label{lemitem:covering_num_sep_set}

    \item Let $\mb R \in \mc R_{u\times v\times w}$ and let $\mb X \subset \mb R$. Let $u'\times v'\times w'$ be an admissible scale and set $u''\times v'' \times w'' = \min(u\times v\times w, u'\times v'\times w')$ where we take the $\min$ coordinate-wise. Then $u''\times v''\times w''$ is also an admissible scale and 
    \begin{align*}
        |\mb X|_{u''\times v''\times w''} \sim |\mb X|_{u'\times v'\times w'}.
    \end{align*}\label{lemitem:cover_subset_rect}
\end{enumerate}
\end{lemma}
\begin{proof}
\parag{Proof of \ref{lemitem:covering_bd}.}
Using approximate transitivity and symmetry,
\begin{align*}
    |\lambda \cdot \mb R|_{u_0\times v_0\times w_0} &= \#\{\omega_0 \in \mc C_{u_0\times v_0\times w_0}\, :\, d_{u_1\times v_1\times w_1}(\omega_0 \to \omega') \leq 1 \text{ and } d_{u_0\times v_0\times w_0}(\omega \to \omega') \leq \lambda\text{ for some $\omega'$}\} \\ 
    &\leq \#\{\omega_0 \in \mc C_{u_0\times v_0\times w_0}\, :\, d_{u_1\times v_1\times w_1}(\omega \to \omega_0) \leq 3\lambda+2\} \lesssim \lambda^3 \frac{u_1v_1w_1}{u_0v_0w_0}.
\end{align*}
In the last line we use the fact that there are $\lesssim \lambda\frac{u_1}{u_0}$-many choices for $a_0$ given $a$, $\lesssim \lambda\frac{w_1}{w_0}$-many choices for $c_0$ given $c$, and $\lesssim \lambda\frac{v_1}{v_0}$-many choices for $b_0$ given $a_0$, $c_0$ and $b$. 

\parag{Proof of \ref{lemitem:covering_num_sep_set}.} 
Let $\omega \in \mb X$. Then there is some $\omega' \in \mc P_{\lambda}$ so that either $d(\omega\to \omega') \leq \lambda$ or $d(\omega'\to \omega) \leq \lambda$. In the first case, approximate symmetry implies $d(\omega'\to \omega) \leq 2\lambda^2$. So 
\begin{align*}
    \mb X &\subset \bigcup_{\omega \in \mc P_{\lambda}} 2\lambda^2\cdot \mb R_{u\times v\times w}(\omega) \\ 
    |\mb X|_{u\times v\times w} &\lesssim \lambda^6 |\mc P_{\lambda}|. 
\end{align*}
In the other direction, if $\mb R \in \mc R_{u\times v\times w}$ then by a packing argument in the metric $d_{u\times v\times w}$ and approximate transitivity and symmetry we get $|\mb R \cap \mc P_{\lambda}| \lesssim 1$, so $|\mc P_{\lambda}|\lesssim |\mb X|_{u\times v\times w}$. 
\end{proof}

\parag{Proof of \ref{lemitem:cover_subset_rect}.}
First of all $|\mb X|_{u''\times v''\times w''} \gtrsim |\mb X|_{u'\times v'\times w'}$ because $u''\times v''\times w'' \leq u'\times v'\times w'$. We have 
\begin{align*}
    d_{u''\times v''\times w''}(\omega \to \omega') = \max\{d_{u\times v\times w}(\omega \to \omega'), d_{u'\times v'\times w'}(\omega \to \omega')\},
\end{align*}
so for $\omega, \omega' \in \mb R$, 
\begin{align*}
    d_{u''\times v''\times w''}(\omega \to \omega') \leq \max\{5, d_{u'\times v'\times w'}(\omega \to \omega')\}.
\end{align*}
If $\mc P \subset \mb X$ is a maximal $10$-separated set in $d_{u'\times v'\times w'}$ then it is also a maximal $10$-separated set in $d_{u''\times v''\times w''}$, and \ref{lemitem:covering_num_sep_set} gives the result.

\subsection{Uniform structure}\leavevmode

\begin{definition}
Let $\mc S = \{u_j\times v_j\times w_j\}_{j=1}^M$ be a set of admissible scales. We say $\mb X \subset \Omega$ is $K$-uniform on the scales $\mc S$ if for every $u\times v\times w \in \mc S$ and $\omega \in \mb X$,
\begin{equation}\label{eq:uniformity_rect_around_pt_defn}
    |\mb X\cap \mb R_{u\times v\times w}(\omega)| \geq \frac{1}{K} \ms M_{u\times v\times w}(\mb X).
\end{equation}
If we specify $\mc S = \{\Delta_j\}$ to be a set of numbers rather than scales, we mean uniformity on all the admissible scales $\Delta_i\times \Delta_j\times \Delta_k$. 
\end{definition}

If $\mb R \in \mc R_{u\times v\times w}(\mb X)$ then there is some $\omega \in \mb R\cap \mb X$ and $4 \cdot \mb R \supset \mb R_{u\times v\times w}(\omega)$ by approximate symmetry and transitivity, so
\begin{equation}\label{eq:uniformity_dilate_defn}
    |\mb X\cap 4\cdot \mb R| \geq \frac{1}{K} \ms M_{u\times v\times w}(\mb X). 
\end{equation}
We will prove that any finite configuration in $\Omega$ has a large uniform subset.
In the proof we use the following graph theory lemma. 
Let $A_1, \ldots, A_t$ be finite sets and $H \subset A_1\times \ldots \times A_t$ an arbitrary subset. 
For $a \in A_i$ we denote by $\deg_{H}a$ the number of $h \in H$ which have $a$ as their $i$-th coordinate.

\begin{lemma}\label{lem_regular_graph}
Let $N, t \in \N$ and denote $K = (Ct \log N)^t$.
Let $A_1, \ldots, A_t$ be sets of size at most $N$ and let $H \subset A_1 \times \ldots \times A_t$ be a subset. Then there exist subsets $A_i' \subset A_i$ such that the restriction $H' = H \cap (A_1' \times \ldots \times A_t')$ has size at least $K^{-1} |H|$ and there are numbers $d_i \ge 0$ such that $\deg_{H'} a \in [d_i/K, d_i]$ for every $a \in A_i'$.
\end{lemma}

\begin{proof}
The first step is to pass to subsets $A_i''\subset A_i$ so the maximum degree is bounded in terms of the average degree. Let $H_0 = H$ and for $i = 1, \ldots, t$ do the following. Given a subset $H_{i-1} \subset A_1\times \ldots \times A_t$ define 
$$
A_i'' = \{a \in A_i:~ \deg_{H_{i-1}} a \in [d_i/2, d_i]\},~~ H_i = H_{i-1} \cap (A_1 \times \ldots \times A_i'' \times \ldots \times A_t) 
$$
for some dyadic $d_i \ge 1$. Since the degree is bounded by $|H| \le N^t$, there is a choice of $d_i$ so that $|H_i| \gtrsim (t\log N)^{-1} |H_{i-1}|$. Fix this choice of $H_i$ and proceed to choosing $H_{i+1}$. 

After this step is complete we obtain subsets $A_i'' \subset A_i$ and a set $H'' = H_t = H \cap (A_1'' \times \ldots \times A_t'')$ so that $|H''| \gtrsim (C t \log N)^{-t} |H|$. Also note that for each $i$ we have $|A_i''| d_i \sim |H_i| \lesssim (Ct \log N)^{t} |H''|$ and the $H''$-degree of any vertex $a \in A''_i$ is bounded by $d_i$. 

The next step is to prune low degree vertices so the min degree is bounded in terms of the average degree. 
Let $A_i^{(0)} = A_i''$ for $i = 1, \ldots, t$ and let $H^{(0)} = H''$.
For $j=0, 1, \ldots $ do the following. If for some $i\in \{1, \ldots, t\}$ there is an $a \in A^{(j)}_i$ with $\deg_{H^{(j)}} a \le \frac{|H''|}{10 t |A_i''|}$ then define $A^{(j+1)}_i = A^{(j)}_i \setminus \{a\}$ and for $i'\neq i$ let $A^{(j+1)}_{i'} = A^{(j)}_{i'}$ and put $H^{(j+1)} = H\cap (A^{(j+1)}_1\times \ldots \times A^{(j+1)}_t)$. If for all $i \in \{1,\ldots, t\}$ and $a \in A_i^{(j)}$ there is a degree lower bound $\deg_{H^{(j)}} a \geq \frac{|H''|}{10t |A_i''|}$, stop the procedure. 

Let $A_i' = A^{(j)}_i$ denote the remaining subsets after terminating the algorithm, and let $H' = H^{(j)}$.
The degree of any $a \in A_i'$ is lower bounded by the termination condition,
\[
\deg_{H'}a = \deg_{H^{(j)}}a \ge \frac{|H''|}{10 t |A_i''|} \ge (C't\log N)^{-t} d_i
\]
and there is a corresponding degree upper bound $\deg_{H'}a \le \deg_{H''}a \le d_i$.
We estimate $|H'|$ by
$$
|H'| \ge |H''| - \sum_{i=1}^t |A_i'' \setminus A_i'| \frac{|H''|}{10 t |A_i''|} \ge |H''| /2 \geq (C't \log N)^{-t} |H|.
$$
\end{proof}

In the following Lemma we make use of dyadic cubes. Let $\delta = 2^{-k}$ be dyadic, and let $\mc D_{\delta}$ be the set of half-open dyadic cubes in $\R^3$ with side length $\delta$. For a set $\mb X \subset \R^3$, let $\mc D_{\delta}(\mb X)$ be the dyadic cubes intersecting $\mb X$. We will use the collection $\mc D_{\delta}$ instead of the similar collection $\mc R_{\delta\times\delta\times\delta}$ because the cubes in $\mc D_{\delta}$ are disjoint. 
\begin{lemma}\label{lem:exists_uniform_subset}
Let $\mc S = \{u_j\times v_j\times w_j\}_{j=1}^M$ be a set of admissible scales with $u_j,v_j,w_j \geq \delta$ for all $j$. Let $K = (CM\log(1/\delta))^{M+2}$. For every finite $\mb X \subset \Omega$ there exists a subset $\mb X'\subset \mb X$ such that $|\mb X'| \ge K^{-1} |\mb X|$ and $\mb X'$ is $K$-uniform. 
\end{lemma}
\begin{remark}
The proof guarantees $\mb X'$ has the following additional property which will be useful in the proof of \Cref{thm:exists_extra_incidence}. Let $\delta$ be the largest dyadic number less than $2^{-40} \min_j\{u_j,v_j,w_j\}$. 
\begin{itemize}
    \item There is some number $r > 0$ such that for every cube $\mb Q \in \mc D_{\delta}(\mb X')$, $|\mb X'\cap \mb Q| \in [r/2, r]$. 
    \item No two dyadic cubes in $\mc D_{\delta}(\mb X')$ are adjacent (including diagonal adjacencies). 
    \item For every $\omega \in \mb X'$ and $u\times v\times w\in \mc S$, not only is there a lower bound $|\mb X'\cap \mb R_{u\times v\times w}(\omega)| \geq \frac{1}{K} \ms M_{u\times v\times w}(\mb X')$, but this inequality remains true if we restrict the left hand side to the dyadic cubes completely contained inside $\mb R_{u\times v\times w}(\omega)$. To be precise, 
    \begin{align*}
        \sum_{\mb Q \in \mc D_{\delta}, \mb Q \subset \mb R_{u\times v\times w}(\omega)} |\mb X'\cap \mb Q| \geq \frac{1}{K} \ms M_{u\times v\times w}(\mb X'). 
    \end{align*}
\end{itemize}
\end{remark}
\begin{proof}
Without loss of generality we take $\delta$ to be the largest dyadic number less than $2^{-40}\min_j\{u_j,v_j,w_j\}$.
For $j \in \{1, \ldots, M\}$ let $A_j$ be the set of boxes $\mc R_{u_j\times v_j\times w_j}(\Omega)$. The cardinality of $A_j$ is bounded by 
\begin{align*}
    |A_j| \lesssim |\Omega|_{u_j\times v_j\times w_j} \lesssim \delta^{-3}.
\end{align*}

First we pigeonhole to a subset $\mb X'' \subset \mb X$ of size $\gtrsim |\log \delta|^{-1} |\mb X|$ such that $|\mb X''\cap \mb Q| \in [r/2, r]$ for some $r \ge 0$ and every  $\mb Q \in \mc D_{\delta}(\mb X'')$. Our final set will correspond to a union of sets $\mb X''\cap \mb Q$ over a subset of dyadic cubes. By passing to a constant fraction subset of $\mc D_{\delta}(\mb X'')$, we may assume no two cubes in $\mc D_{\delta}(\mb X'')$ are adjacent (including diagonal adjacencies). 

Now define a subset $H \subset \prod_{j=1}^M A_{j} \times \mc D_{\delta}(\mb X'')$ as follows. For each box $\mb Q \in \mc D_{\delta}(\mb X'')$ and each $j \in \{1,\ldots, M\}$, let $e_{\mb Q,j} \in \mc R_{u_j\times v_j\times w_j}$ be a rectangle such that $\mb Q \subset 2^{-10}\cdot e_{\mb Q,j}$. Such a rectangle exists because $\delta \leq 2^{-40}\min_j\{u_j,v_j,w_j\}$. There may be several possible choices for $e_{\mb Q,j}$, so we break ties arbitrarily. 
We let $e_{\mb Q}$ be the tuple $(e_{\mb Q,j})_{j=1}^M \in \prod_{j=1}^M A_j$, and define the edge set to be the union of these tuples with $\mb Q$ attached on the last coordinate (this ensures that all edges are distinct),
\begin{equation*}
    H = \{(e_{\mb Q}, {\mb Q})\, :\, \mb Q \in \mc D_{\delta}(\mb X'')\}. 
\end{equation*}
We apply \Cref{lem_regular_graph} to $H$ to obtain subsets $A'_j \subset A_j$ satisfying the following with $K' = (C M \log (1/\delta))^{M+1}$.
\begin{itemize}
    \item $H' = H \cap (A'_1 \times \ldots \times A'_M \times \mc D_{\delta})$ has size at least $(K')^{-1} |H|$. We associate to $H'$ the set of cubes 
    \begin{align*}
        \mc Q' = \{\mb Q \in \mc D_{\delta}(\mb X'')\, :\, (e_{\mb Q}, \mb Q) \in H'\}.
    \end{align*}
    \item For each $j \in \{1,\dots, M\}$ there is a number $d_j$ such that for every $a \in A_j'$, $\deg_{H'} a \in [d_j / K', d_j]$. 
\end{itemize}
We put
$$
\mb X' = \mb X'' \cap \bigcup_{\mb Q \in \mc Q'} \mb Q.
$$
By the lower bound on $|H'|$, 
\begin{align}\label{eq:lower_bd_size_Xprime}
    |\mb X'| \geq (r/2) |\mc Q'| \geq (CK')^{-1} |\mb X|. 
\end{align}
For each $\omega \in \mb X'$ we will prove a lower bound for $|\mb X' \cap \mb R_{u_j\times v_j\times w_j}(\omega)|$. Let $\mb Q$ be the dyadic $\delta$-cube containing $\omega$, so $e_{\mb Q} \in H'$ and $e_{\mb Q,j} \in A_j'$. By the degree lower bound on $H'$, 
\begin{align*}
    \#\{\mb Q' \in \mc Q'\, :\, e_{\mb Q',j} = e_{\mb Q,j}\} \geq d_j/K'. 
\end{align*}
All the cubes above are contained in $\mb R_{u_j\times v_j\times w_j}(\omega)$. 
Indeed, if $e_{\mb Q',j} = e_{\mb Q,j}$ then $\mb Q' \subset 2^{-10} \cdot e_{\mb Q,j}$, and because $\omega \in 2^{-10} e_{\mb Q,j}$, by approximate symmetry and transitivity $2^{-10} e_{\mb Q,j} \subset \mb R_{u_j\times v_j\times w_j}(\omega)$. Thus
\begin{align}\label{eq:rect_itnersect_lower_bd}
    |\mb X'\cap \mb R_{u_j\times v_j\times w_j}(\omega)| \geq \sum_{\{\mb Q' \in \mc Q'\, :\, e_{\mb Q',j} = e_{\mb Q,j}\}} |\mb Q'\cap \mb X'|  \geq \frac{1}{2K'} rd_{j}.
\end{align}
Moreover, as described in the remark before this proof, this sum only counts dyadic cubes that are entirely contained in the rectangle $\mb R_{u_j\times v_j\times w_j}(\omega)$. 
Now we prove a matching upper bound. We start by writing $\mb X'$ as
\begin{align*}
    \mb X' &=  \bigcup_{\mb R \in A_j'}\bigcup_{\{\mb Q\, :\, e_{\mb Q,j} = \mb R\}}\mb X'' \cap \mb Q \\ 
    &=: \bigcup_{\mb R \in A_j'} S_{\mb R}. 
\end{align*}
Each set $S_{\mb R}$ is contained in $\mb R$, and 
\begin{align*}
    |S_{\mb R}| \leq rd_j. 
\end{align*}
For any rectangle $\mb R = \mb R_{u_j\times v_j\times w_j}(\omega')$ with scale in $\mc S$, 
\begin{align*}
    \mb R \cap \mb X \subset \bigcup_{\mb R'\cap \mb R \neq \emptyset, \mb R' \in A_j'} S_{\mb R'}.
\end{align*}
There are $\leq C$ many rectangles $\mb R' \in A_j'$ intersecting $\mb R$, so 
\begin{align}\label{eq:rect_intersect_upper_bd}
    |\mb R\cap \mb X| \leq Crd_j.
\end{align}
Setting $K = CK'$ and combining \eqref{eq:rect_itnersect_lower_bd} and \eqref{eq:rect_intersect_upper_bd} gives that for any $\omega \in \mb X'$,
\begin{align*}
    |\mb X' \cap \mb R_{u_j\times v_j\times w_j}(\omega)| \geq \frac{1}{K} \ms M_{u_j\times v_j\times w_j}(\mb X'),
\end{align*}
so $\mb X'$ is $K$-uniform. 
\end{proof}

\subsection{Weak uniformity}
For any set $\mb X$, there is an inequality
\begin{align*}
    |\mb X| \leq \ms M_{u\times v\times w}(\mb X) |\mb X|_{u\times v\times w}.
\end{align*}
If $\mb X$ is $K$-uniform at scale $u\times v\times w$ then there is a partial converse.
Let $\mc P \subset \mb X$ be a maximal set of $10$-separated points in the $d_{u\times v\times w}$ distance. 
By \Cref{lem:covering_props}\ref{lemitem:covering_num_sep_set} we know $|\mc P| \sim |\mb X|_{u\times v\times w}$. 
The rectangles $\mb R_{u\times v\times w}(\omega)$ are disjoint for each $\omega \in \mc P$, so by $K$-uniformity, 
\begin{align*}
    \sum_{\omega \in \mc P} |\mb X\cap \mb R_{u\times v\times w}(\omega)| &\geq K^{-1} \ms M_{u\times v\times w}(\mb X)|\mc P|. 
\end{align*}
On the other hand 
\begin{align*}
    \sum_{\omega \in \mc P} |\mb X\cap \mb R_{u\times v\times w}(\omega)| \leq |\mb X|,
\end{align*}
and combining these gives 
\begin{equation}\label{eq:weak_uniform}
    |\mb X| \geq \frac{1}{CK} \ms M_{u\times v\times w}(\mb X) |\mb X|_{u\times v\times w}
\end{equation}
for some universal $C$. In general if \eqref{eq:weak_uniform} holds for an admissible $u\times v\times w$ with constant $K$ we say that $\mb X$ is \textit{weakly $K$-uniform} at that scale.  
Weak uniformity is a useful property because of the following Lemma. 

\begin{lemma}\label{lem:branching_struct_subrect}
Suppose that $\mb X$ is $K$-uniform on the scales $\mc S$. Let $u\times v\times w \in \mc S$ and $\mb R_0 \in \mc R_{u\times v\times w}(\mb X)$, set $\mbt X = \mb X\cap \mb R_0$. If $|\mbt X| \geq C^{-1} \ms M_{u\times v\times w}(\mb X)$ then $\mbt X$ is weakly $CK^2$-uniform on scales $\mc S' = \{u'\times v'\times w' \in \mc S\, :\, u'\times v'\times w' \leq u\times v\times w\}$, and
\begin{align*}
    |\mbt X|_{u'\times v'\times w'} &\sim_{CK^3} \frac{|\mb X|_{u'\times v'\times w'}}{|\mb X|_{u\times v\times w}} \qquad \text{for all $u'\times v'\times w' \in \mc S'$.}
\end{align*}
\end{lemma}
\begin{proof}
Let $u'\times v'\times w' \in \mc S'$, and let $\mc P\subset \mbt X$ be a maximal $10$-separated set for $d_{u'\times v'\times w'}$. As in the proof that uniformity implies weak uniformity, 
\begin{align*}
    K^{-1} |\mc P| \ms M_{u'\times v'\times w'}(\mb X) \leq |\mb X \cap 10\cdot \mb R_0| \lesssim \ms M_{u\times v\times w}(\mb X). 
\end{align*}
Rearranging and using \Cref{lem:covering_props}\ref{lemitem:covering_num_sep_set} and $K$-uniformity of $\mb X$ gives 
\begin{align*}
    |\mbt X|_{u'\times v'\times w'} &\lesssim K^2 \frac{|\mb X|_{u'\times v'\times w'}}{|\mb X|_{u\times v\times w}}. 
\end{align*}
In the other direction  
\begin{align*}
    |\mbt X|_{u'\times v'\times w'} &\geq \frac{|\mbt X|}{\ms M_{u'\times v'\times w'}(\mb X)}  
    \geq C^{-1} \frac{\ms M_{u\times v\times w}(\mb X)}{\ms M_{u'\times v'\times w'}(\mb X)} 
    \geq (CK)^{-1} \frac{|\mb X|_{u'\times v'\times w'}}{|\mb X|_{u\times v\times w}}.
\end{align*}
Using the lower bound on $|\mbt X|$ stated in the Lemma's hypothesis and weak uniformity of $\mb X$, in particular \eqref{eq:weak_uniform}, we find 
\begin{align*}
    |\mbt X|  \geq C^{-1} \ms M_{u\times v\times w}(\mb X) \geq C^{-1} \frac{|\mb X|}{|\mb X|_{u\times v\times w}} 
    \gtrsim C^{-1} K^{-2} \ms M_{u'\times v'\times w'}(\mbt X) |\mbt X|_{u'\times v'\times w'}
\end{align*}
giving weak uniformity of $\mbt X$. 
\end{proof}

\subsection{Branching function}\label{subsec:branching_func}
Let $m, T \geq 1$ be integers and set $\delta = 2^{-mT}$. We say $\mb X$ is $(m, T)$-uniform if $\mb X$ is $K$-uniform on the sequence $\{2^{-jT}\}_{j=0}^m$ with $K = (Cm^3\log(1/\delta))^{10m^3+2}$. We will take $m$ to be a large constant and let $T \to \infty$ with $m$ fixed. We will sometimes write $o(1)$ errors in what follows. These are errors that go to zero for fixed large $m$ as $T \to \infty$. 

We define the branching function $f(x,y,z) = f_{\mb X}(x,y,z)$ so that 
\begin{align*}
    |\mb X|_{\delta^x\times \delta^{z}\times \delta^y} = \delta^{-f(x,y,z)}.
\end{align*}

Notice the change in the order of indices from `x-z-y' to `x-y-z'. Up to this point in the paper the order `x-z-y' had the advantage that the middle variable is related to both the `point' and the `line' components of the point-line space. In the rest of the paper we will often omit the `z' variable and write $f(x,y) = f(x,y,x+y)$ so we changed the order to make this more natural. 

Let
\begin{align*}
    \ms D &= \{(x,y,z)\, :\, x,y\in [0,1],\, z \in [0, \min(1, x+y)]\} \\ 
    \ms D_m &= \ms D\cap \frac{1}{m}\Z^3.
\end{align*}
We set $\ms D_m$ to be the domain of $f_{\mb X}$. We restrict attention to $\ms D_m$ because those are the inputs along which $\mb X$ has branching structure. Note that the restriction $z \le x+y$ corresponds to admissible triples. We write 
\begin{equation*}
    f(x,y) = f(x,y,x+y).
\end{equation*}
Say $(x,y)\in \ms D$ if $x,y\geq 0$ and $x+y\leq 1$ and similarly for $(x,y)\in \ms D_m$. We also write
\begin{align*}
    f(x,y,z\sep x_0, y_0) = f(x_0+x,y_0+y,x_0+y_0+z) - f(x_0, y_0).
\end{align*}
By Lemma \ref{lem:branching_struct_subrect}, up to $o(1)$ errors this is the branching function of $\psi_{\mb R}(\mb X \cap \mb R)$ where $\mb R$ is a $\delta^{x_0}\times \delta^{x_0+y_0} \times \delta^{y_0}$ rectangle intersecting $\mb X$ significantly. 

\begin{lemma}\label{lem:f_is_lip}
Let $(x,y,z) \in \ms D_m$ and $(x+x',y+y',z+z')\in \ms D_m$ with $x',y',z'\geq 0$. Then 
\begin{align*}
    0 \leq f(x+x', y+y', z+z') - f(x,y,z) &\leq x'+y'+z' + o(1).
\end{align*}
In particular, $f$ is monotone with respect to each coordinate and Lipschitz. 
\end{lemma}
\begin{proof}
First of all, every element of $\mc R_{u\times v\times w}$ is contained in an element of $\mc R_{u'\times v'\times w'}$ if $u\leq u'$, $v\leq v'$, $w\leq w'$, giving monotonicity. 
For $a'\leq a$, $b'\leq b$, $c'\leq c$ and $b \leq a+c$, $b'\leq a'+c'$, \Cref{lem:covering_props}\ref{lemitem:covering_bd} gives
\begin{align*}
|\mb X|_{\delta^a\times \delta^b\times \delta^c} \lesssim \delta^{a'-a}\delta^{b'-b}\delta^{c'-c} |\mb X|_{\delta^{a'}\times \delta^{b'}\times \delta^{c'}},
\end{align*}
take logs to obtain the result. 
\end{proof}

In the next lemma we show that the branching functions have a submodularity property. Later we use this fact to
deduce the crucial direction set inequalities needed for our proof.
\begin{lemma}[Submodularity]\label{lem:relative_covering_ineq}
Let $\mb X$ be $(m, T)$-uniform and let $f$ be the branching function. Let $a, b \in \ms D_m$ and let $\max(a,b)$ and $\min(a,b)$ denote the coordinate-wise $\max$ and $\min$ of the tuples. It follows automatically that $\max(a,b) \in \ms D_m$, assume in addition that $\min(a,b) \in \ms D_m$. Then
\begin{align*}
    f(\max(a,b)) + f(\min(a,b)) \leq f(a) + f(b) + o(1). 
\end{align*}
\end{lemma}
\begin{proof}
Denote 
\begin{align*}
    a &= (x_0', y_0', z_0'), \\ 
    b &= (x_1, y_1, z_1), \\ 
    \max(a,b) &= (x_1', y_1', z_1'),\\
    \min(a, b) &= (x_0, y_0, z_0).
\end{align*}
Let $\mb \omega_* \in \mb X$ be fixed. Let $\mb R_j = \mb R_{\delta^{x_j}\times \delta^{z_j}\times \delta^{y_j}}(\omega_*)$ for $j = 0,1$. Uniformity implies
\begin{align*}
    |\mb X \cap \mb R_0| &\geq K^{-1}\ms M_{\delta^{x_0}\times \delta^{z_0}\times \delta^{y_0}}(\mb X), \\ 
    |\mb X \cap \mb R_1| &\geq K^{-1}\ms M_{\delta^{x_1}\times \delta^{z_1}\times \delta^{y_1}}(\mb X).
\end{align*}
By \Cref{lem:branching_struct_subrect},
\begin{align*}
    f(x_0',y_0',z_0') - f(x_0,y_0,z_0) &= \log_{\delta^{-1}} |\mb X\cap \mb R_0|_{\delta^{x_0'}\times \delta^{z_0'}\times \delta^{y_0'}} \pm o(1) \\ 
    f(x_1',y_1',z_1') - f(x_1,y_1,z_1) &= \log_{\delta^{-1}} |\mb X\cap \mb R_1|_{\delta^{x_1'}\times \delta^{z_1'}\times \delta^{y_1'}} \pm o(1).
\end{align*}
The claim has been reduced to proving $|\mb X\cap \mb R_0|_{\delta^{x_0'}\times \delta^{z_0'}\times \delta^{y_0'}} \gtrsim |\mb X\cap \mb R_1|_{\delta^{x_1'}\times \delta^{z_1'}\times \delta^{y_1'}}$. 
This estimate follows from applying \Cref{lem:covering_props}\ref{lemitem:cover_subset_rect} to $\mb X\cap \mb R_1\subset \mb R_1$.
\end{proof}

An important quantity in our proof is the covering number of the slope set. We let 
\begin{align*}
    \Dir(\eta) &= |\mb X|_{1\times 1\times \eta}, \\ 
    d(t) &= f(0, t, 0) \\ 
    d(t\sep x, y) &= f(0, t, 0\sep x, y).
\end{align*}
Up to constants $\Dir(\eta)$ is the $\eta$-covering number of the slope set of $L[\mb X]$, $d(t)$ is essentially $\log_{\delta^{-1}} \Dir(\delta^t)$ and $d(t;x, y)$ is the analogous quantity in the rescaled set $\psi_{\mb R}(\mb X \cap \mb R)$ for $\mb R$ a $\delta^x\times \delta^{x+y} \times \delta^y$ rectangle which intersects $\mb X$ significantly. 
Similarly, we define dual quantities
\begin{align*}
    \Dir^{\vee}(\eta) &= |\mb X|_{\eta\times 1\times 1} \\ 
    d^{\vee}(t) &= f(t, 0, 0) \\ 
    d^{\vee}(t\sep x, y) &= f(t, 0, 0\sep x, y).
\end{align*}
Up to constants $\Dir^{\vee}(\eta)$ is the $\eta$-covering number of the set of $x$-coordinates of $P[\mb X]$. 
\begin{lemma}\label{lem:dir_set_props}
Let $\mb X$ be $(m, T)$-uniform. Let $x,y,s,t \in m^{-1}\Z_{\geq 0}$ be such that the inputs below are defined. Then we have
\begin{enumerate}[label=(\roman*)]
    \item $d(t+s\sep x, y) \geq d(t\sep x, y)+d(s\sep x, y+t) - o(1)$\label{lemitem:dir_compare_scales}
    \item $d(t\sep x, y) \geq f(0, t\sep x,y) - t - o(1)$ \label{lemitem:dir_lower_bd}
    \item The function $d$ satisfies the following bounds for $t', x', y' \geq 0$,
    \begin{align*}
        0 \leq d(t+t'\sep x, y) - d(t\sep x, y) &\leq t' + o(1), \\ 
        -o(1) \leq d(t\sep x, y)  - d(t\sep x+x', y) &\leq 2x' + o(1), \\ 
        |d(t\sep x, y+y') - d(t\sep x, y)| &\leq 2y' + o(1).
    \end{align*}
    In particular, $d$ is monotone with respect to $t$ and $x$ and Lipschitz with respect to all three parameters. 
    \label{lemitem:dir_lip}
\end{enumerate}
Similarly for the dual quantities, 
\begin{enumerate}[label=(\roman*)]
    \item $d^{\vee}(t+s\sep x, y) \geq d(t\sep x, y) + d(s\sep x+t, y) - o(1)$
    \item $d^{\vee}(t\sep x, y) \geq f(t, 0\sep x, y) - t - o(1)$
    \item The function $d^{\vee}$ satisfies the following bounds for $t', x', y' \geq 0$,
    \begin{align*}
        0 \leq d^{\vee}(t+t'\sep x, y) - d^{\vee}(t\sep x, y) &\leq t' + o(1), \\ 
        0\leq |d^{\vee}(t\sep x, y)  - d^{\vee}(t\sep x+x', y)| &\leq 2x' + o(1), \\ 
        -o(1) \leq d^{\vee}(t\sep x, y) - d^{\vee}(t\sep x, y+y') &\leq 2y' + o(1).
    \end{align*}
    In particular, $d^{\vee}$ is monotone with respect to $t$ and $y$ and Lipschitz with respect to all three parameters. 
\end{enumerate}
\end{lemma}
\begin{proof}
We prove the properties for the primal direction set, the dual direction set is similar.
\parag{Proof of \ref{lemitem:dir_compare_scales}}
By submodularity (\Cref{lem:relative_covering_ineq}),
\begin{align*}
    f(0,t+s,t\sep x, y) + f(0,t,0\sep x, y) \leq f(0, t+s,0\sep x, y) + f(0, t, t\sep x, y) + o(1)
\end{align*}
and rearranging gives the inequality.

\parag{Proof of \ref{lemitem:dir_lower_bd}}
\Cref{lem:f_is_lip} implies $f(0, t, 0\sep x, y) \geq f(0, t, t\sep x, y) - t-o(1)$. 

\parag{Proof of \ref{lemitem:dir_lip}} The Lipschitz properties and monotonicity in the $t$ variable follows from Lipschitz and monotonicity of $f$ (\Cref{lem:f_is_lip}). Monotonicity in the $x$-variable follows from submodularity (\Cref{lem:relative_covering_ineq}),
\begin{align*}
    f(0,t,0\sep x, y) - f(0,0,0\sep x, y) \geq f(x',t,x'\sep x, y) - f(x',0,x'\sep x, y)-o(1).
\end{align*}
\end{proof}

\subsection{Limiting branching function}\label{subsec:limiting_branching_func}
It is helpful to consider limiting branching functions to avoid $o(1)$ errors. 

\begin{definition}
Let $\mc L$ be the class of functions $f: \ms D \to \R_{\geq 0}$ which are limits of branching functions. To be precise, let $\{\mb X_{j,k}\}_{j,k=0}^{\infty}$ be a sequence of sets in $\Omega$ which are $(m_j, T_{j,k})$-uniform where $m_j \to \infty$ and for each $j$, $T_{j,k}\to \infty$ as $k\to \infty$. Let $f_j$ be a subsequential limit of the branching functions $f_{\mb X_{j,k}}$ as $k\to \infty$. Extend $f_j$ from a Lipschitz function on $\ms D_{m_j}$ to a Lipschitz function on $\ms D$, and let $f$ be a subsequential limit of the functions $f_j$ as $j\to \infty$. The set $\mc L$ consists of all the functions $f$ arising like this. 
\end{definition}
\bigskip\\
Here are some properties of $\mc L$. 
\begin{enumerate}[label=(\arabic*)]
    \item The class $\mc L$ is closed under uniform limits. 

    \item Any $f \in \mc L$ is $1$-Lipschitz and monotonic,
    \begin{equation}\label{eq:f_lip}
        0 \leq f(x+x',y+y',z+z') - f(x,y,z)\leq x'+y'+z'.
    \end{equation} \label{lemitem:classL_lip_mon}
    
    \item  Any $f \in \mc L$ satisfies the submodularity inequality from \Cref{lem:relative_covering_ineq}. Let $a,b \in \ms D$ be given and let $\max(a,b)$ and $\min(a,b)$ denote the coordinate-wise $\max$ and $\min$ of the tuples. Assume $\min(a,b) \in \ms D$.
    Then
    \begin{align*}
        f(\max(a,b)) + f(\min(a,b)) \leq f(a) + f(b).
    \end{align*}  \label{lemitem:classL_rel_covering}
    \item Given an $f \in \mc L$ we can make sense of the quantities $d(t\sep x, y)$ and $d^\vee(t\sep x,y)$, namely
    \begin{align*}
        d(t\sep x,y) = f(0,t,0;x,y) = f(x, y+t, x+y) - f(x,y),\\
        d^\vee(t\sep x,y) = f(t,0,0;x,y) = f(x+t, y, x+y) - f(x,y),
    \end{align*}
    where we write $f(x, y) = f(x, y, x+y)$. 
    The direction set inequalities from \Cref{lem:dir_set_props} hold without $o(1)$ errors by combining Lipschitz, monotonicity, and submodularity.

    \item Let $(x_0, y_0) \in \ms D$ and $w \in (0,1-x_0-y_0]$. We denote 
    \begin{align*}
    Q_w(x,y) = \{x,y\, :\, x \geq x_0, y \geq y_0,\, (x-x_0)+(y-y_0) \leq w\}.
    \end{align*}
    Set 
    \begin{align*}
        f_Q(x,y,z) = \frac{1}{w}f(wx,wy,wz\sep x_0,y_0)\qquad \text{for $(x,y,z)\in \ms D$.}
    \end{align*}
    Just like $f(x, y, z\sep x_0, y_0)$ corresponds to blowing up $\mb X$ into a $\delta^{x_0}\times \delta^{x_0+y_0}$ rectangle, $f_Q(x,y,z)$ corresponds to blowing up into a $\delta^{x_0}\times \delta^{x_0+y_0}$ rectangle and then blurring below scale $\delta^w$. If $f \in \mc L$ then $f_Q \in \mc L$ as well.

    To prove this, we have to understand how covering numbers behave under rescaling. Let $\mb X \subset \Omega$ be an $(m, T)$-uniform set. Let $(x_0,y_0,z_0) \in \ms D_m$, and let $\mb R \in \mc R_{\delta^{x_0}\times \delta^{x_0+y_0}\times \delta^{y_0}}$ intersect $\mb X$ significantly, meaning $|\mb X\cap \mb R| \geq \frac{1}{C} \ms M_{\delta^{x_0}\times \delta^{x_0+y_0}\times \delta^{y_0}}(\mb X)$. Let $m' = m(1-x_0-y_0)$, set $\mb X' = \psi_{\mb R}(\mb X\cap \mb R)$, and extract an $(m', T)$-uniform subset $\mb X'' \subset \mb X'$ with $|\mb X''| \geq \frac{1}{K}|\mb X'|$. 
    \begin{lemma}\label{lem:covering_number_rescaling}
    Let $\mb X$, $\mb X'$, and $\mb X''$ be as above. 
    For any $(x,y,z) \in \ms D_{m'}$,
    \begin{align*}
        |\mb X''|_{\delta^{x}\times \delta^{z}\times \delta^{y}} \sim_{K^7} \frac{|\mb X|_{\delta^{x_0+x}\times \delta^{x_0+y_0+z}\times \delta^{y_0+y}}}{|\mb X|_{\delta^{x_0}\times \delta^{x_0+y_0}\times \delta^{y_0}}}.
    \end{align*}
    By taking logs with base $\delta' = \delta^{1-x_0-y_0}$, we find
    \begin{align*}
        f_{\mb X''}(x,y,z) = \frac{1}{1-x_0-y_0}f_{\mb X}((1-x_0-y_0)x,y(1-x_0-y_0),(1-x_0-y_0)z; x_0,y_0) + O\Bigl(\frac{\log K}{\log \delta'}\Bigr). 
    \end{align*}
    The big-$O$ term is bounded above by $\frac{1}{1-x_0-y_0}o_{\delta}(1)$.
    \end{lemma}
    Note that the branching function $f_{\mb X''}$ satisfies
    \begin{align*}
        |\mb X''|_{(\delta')^x\times (\delta')^z\times (\delta')^y} = (\delta')^{-f_{\mb X''}(x,y,z)},
    \end{align*}
    where $\delta' = \delta^{1-x_0-y_0}$.
    \begin{proof}
    By \Cref{lem:branching_struct_subrect}, $\mb X'$ is weakly $K^3$-uniform on scales $\{2^{-jT}\}_{j=0}^{m'}$, and for $(x,y,z) \in \ms D_{m'}$,
    \begin{equation*}
        |\mb X'|_{\delta^x\times \delta^z\times \delta^y} \sim_{K^3} \frac{|\mb X|_{\delta^{x_0+x}\times \delta^{x_0+y_0+z}\times \delta^{y_0+y}}}{|\mb X|_{\delta^{x}\times \delta^{x+y}\times \delta^y}}.
    \end{equation*}
    Obviously $|\mb X''|_{\delta^x\times \delta^z\times \delta^y} \leq |\mb X'|_{\delta^x\times \delta^z\times \delta^y}$. Using weak uniformity of $\mb X'$, we can prove a matching lower bound,
    \begin{align*}
        |\mb X''|_{\delta^x\times \delta^z\times \delta^y} &\geq \frac{|\mb X''|}{\ms M_{\delta^x\times \delta^z\times \delta^y}(\mb X'')}\gtrsim K^{-1} \frac{|\mb X'|}{\ms M_{\delta^x\times \delta^z\times \delta^y}(\mb X'')} \\ 
        &\gtrsim K^{-4} \frac{|\mb X'|_{\delta^x\times \delta^z\times \delta^y} \ms M_{\delta^x\times \delta^z\times \delta^y}(\mb X')}{\ms M_{\delta^x\times \delta^z\times \delta^y}(\mb X'')} \gtrsim K^{-4} |\mb X'|_{\ms M_{\delta^x\times \delta^z\times \delta^y}(\mb X'')}.
    \end{align*}
    This completes our estimate for the covering numbers of $\mb X''$, and the branching function bound follows by taking logs. 
    \end{proof}

    Using \Cref{lem:covering_number_rescaling}, we can prove that if $f \in \mc L$ then $f_Q\in \mc L$ as well. Let $\{\mb X_{j,k}\}_{j,k=0}^{\infty}$ be a sequence of sets having $f$ as their only limiting branching function in the sense described above. 
    Let $\mb R_{j,k} \in \mc R_{\delta^{x_0}\times \delta^{x_0+y_0}\times \delta^{y_0}}$ intersect $\mb X_{j,k}$ significantly, and set $\mb X_{j,k}' = \psi_{\mb R_{j,k}}(\mb X_{j,k}\cap \mb R_{j,k})$. 
    Let $\mb X_{j,k}'' \subset \mb X_{j,k}'$ be $(m_j', T_{j,k})$-uniform where we choose $m_j' < m_j$ to be the largest integer so that $2^{-m_j' T_{j,k}} \geq 2^{-wm_j T_{j,k}}$. 
    By \Cref{lem:covering_number_rescaling}, $f_Q$ is the limiting branching function of $\{\mb X_{j,k}''\}_{j,k=0}^{\infty}$.\label{item:limiting_branching_func_blowup}

    \item Ren and Wang's Furstenberg set estimate \cite[Theorem 4.1]{RenWang2023} can be phrased in terms of $\mc L$. Let $f \in \mc L$ and suppose 
    \begin{align*}
        d(y\sep 1,0) &\geq s y  \qquad \text{for $0 \leq y \leq 1$,} \\ 
        f(x, 0) &\geq t x  \qquad \text{for $0 \leq x \leq 1$.}
    \end{align*}
    Then 
    \begin{align*}
        f(0, 1) &\geq \min\{t,(s+t)/2,1\} +  d(1\sep 1,0).
    \end{align*}
    See \S\ref{sec:furstenberg_pf_phase_space} for a proof. \label{exampleitem:info_L_furstenberg}
\end{enumerate}

\section{A two step method for incidence lower bounds}\label{sec:two_step_inc_lower_bd}
\subsection{Initial estimate and inductive step for uniform sets}
In this section we prove bounds for the normalized incidence count
\begin{equation*}
    B(w; \mb X) = \frac{I(w; \mb X)}{w|\mb X|^2}
\end{equation*}
using the high-low inequality. 
Below we abbreviate $B(w) = B(w; \mb X)$. 

Recall that if $P$ and $L$ are a point and a line set, potentially counted with multiplicity, the high-low inequality (\Cref{thm:high_low_ineq}) states 
\begin{align*}
 |B(w; P, L) - B(w/2; P, L)| \lesssim \Bigl(\frac{\ms M_{w\times w}(P)}{|P|}  \frac{\ms M_{1\times w}(L)}{|L|}w^{-3}\Bigr)^{1/2}.
\end{align*}
Let $P[\mb X]$ and $L[\mb X]$ be the point and line sets of the phase space set $\mb X$, counted with multiplicity so that
\begin{align*}
    |P[\mb X]| = |L[\mb X]| = |\mb X|\qquad \text{and}\qquad I(w; \mb X) = I(w; P[\mb X], L[\mb X]).
\end{align*}
By applying the high-low inequality to $P[\mb X]$ and $L[\mb X]$ and using \eqref{eq:covering_num_comp_PL} to estimate the concentration numbers thereof,
\begin{align}
    |B(w; \mb X) - B(w/2; \mb X)| &\lesssim \Bigl(\frac{\ms M_{w\times w}(P[\mb X])}{|P[\mb X]|}  \frac{\ms M_{1\times w}(L[\mb X])}{|L[\mb X]|}w^{-3}\Bigr)^{1/2} \nonumber \\ 
    &\lesssim \Bigl(\frac{\ms M_{w\times w\times 1}(\mb X)}{|\mb X|}\frac{\ms M_{1\times w\times w}(\mb X)}{|\mb X|}w^{-3}\Bigr)^{1/2}. \label{eq:HighLowPhaseSpace}
\end{align}

Here is a version of the high-low estimate adapted uniform sets in configuration space. 
\begin{prop}\label{prop:phase_space_high_low}
Suppose $\mb X$ is $K$-uniform at scales $\{1,w\}$. Then for any dyadic $A = 2^i$, 
\begin{align*}
    |B(w) - B(w/A)| &\lesssim K A^{3/2} (|\mb X|_{w\times w\times 1}^{-1} |\mb X|_{1\times w\times w}^{-1} w^{-3})^{1/2}, \\ 
    |B(w) - B(Aw)| &\lesssim KA^{1/2} (|\mb X|_{w\times w\times 1}^{-1} |\mb X|_{1\times w\times w}^{-1} w^{-3})^{1/2}.
\end{align*}
\end{prop}
\begin{proof}
To obtain the first inequality, use \eqref{eq:HighLowPhaseSpace} and \eqref{eq:weak_uniform} to find
\begin{align*}
    |B(2^{-j}w) - B(2^{-j-1}w)| &\lesssim 2^{\frac{3}{2}(j+1)}\Bigl(\frac{\ms M_{w\times w\times 1}(\mb X)}{|\mb X|}\frac{\ms M_{1\times w\times w}(\mb X)}{|\mb X|}w^{-3}\Bigr)^{1/2} \\ 
    &\lesssim 2^{\frac{3}{2}(j+1)} K(|\mb X|_{w\times w\times 1}^{-1} |\mb X|_{1\times w\times w}^{-1}w^{-3})^{1/2}
\end{align*}
and sum the geometric series. To obtain the second inequality, we use $\ms M_{2^jw\times 2^jw\times 1}(\mb X) \lesssim 2^{2j} \ms M_{w\times w\times 1}(\mb X)$ and $\ms M_{1\times 2^jw\times 2^jw}(\mb X) \lesssim 2^{2j} \ms M_{1\times w\times w}(\mb X)$. Again using \eqref{eq:covering_num_comp_PL} and \eqref{eq:weak_uniform}, 
\begin{align*}
    |B(2^jw) - B(2^{j+1}w)| &\lesssim \Bigl(\frac{\ms M_{2^{j+1}w\times 2^{j+1}w}(P[\mb X])}{|P|}\frac{\ms M_{1\times 2^{j+1}w}(L[\mb X])}{|L|}(2^jw)^{-3}\Bigr)^{1/2}, \\  
    &\lesssim 2^{\frac{1}{2}j} \Bigl(\frac{\ms M_{w\times w}(P[\mb X])}{|P|}\frac{\ms M_{1\times w}(L[\mb X])}{|L|} w^{-3}\Bigr)^{1/2}
\end{align*}
and we sum the geometric series. 
\end{proof}

We use the following initial estimate.
\begin{prop}\label{prop:phase_space_initial_estimate}
Suppose that $\mb X$ is $K$-uniform at scales $\{1,\eta\}$. Then 
\begin{align*}
    B(\eta) \gtrsim K^{-2} \frac{|\mb X|_{\eta\times\eta\times\eta}}{\eta |\mb X|_{\eta\times\eta\times 1}|\mb X|_{1\times \eta\times \eta}}.
\end{align*}
\end{prop}

The right hand side can be written in terms of direction sets or dual direction sets,
\begin{align*}
    \frac{|\mb X|_{\eta\times\eta\times\eta}}{\eta |\mb X|_{\eta\times\eta\times 1}|\mb X|_{1\times \eta\times \eta}} = \frac{\Dir(\eta\sep \eta\times \eta \times 1)}{\eta |\mb X|_{1\times \eta\times \eta}} =  \frac{\Dir^{\vee}(\eta\sep 1\times \eta \times \eta)}{\eta |\mb X|_{\eta\times \eta\times 1}}
\end{align*}
where 
\begin{align*}
    \Dir(\eta\sep \eta\times \eta\times 1) = \frac{|\mb X|_{\eta\times\eta\times \eta}}{|\mb X|_{\eta\times \eta\times 1}} = \text{The $\eta$-covering number of the directions spanned inside an $\eta\times \eta$ square}
\end{align*}
and similarly $\Dir^{\vee}(\eta\sep 1\times \eta\times \eta)$ is the $\eta$-covering number of the $x$-coordinates spanned inside a $1\times \eta$ tube.
\begin{proof}
Let $\mc P$ be a maximal $10$-separated set of points in $\mb X$ under the $d_{\eta\times\eta\times\eta}$ metric. There are $\gtrsim |\mb X|_{\eta\times \eta\times \eta}$ points in $\mc P$. 
For each $\omega \in \mc P$, let 
\begin{align*}
    \mb A^{\omega} &= \mb X \cap 2^{-10}\cdot \mb R_{\eta\times \eta\times 1}(\omega), \\ 
    \mb B^{\omega} &= \mb X \cap 2^{-10}\cdot \mb R_{1\times \eta\times \eta}(\omega).  
\end{align*}
For any $\omega_1 \in\mb A^{\omega}$ and $\omega_2\in \mb B^{\omega}$, we have 
\begin{align*}
    d_{\square\times 1\times \square}(\omega_2\to \omega_1) &\leq 2^{-5}\eta
\end{align*}
by approximate transitivity and symmetry. 
Thus $d(p_{\omega_1}, \ell_{\omega_2}) \leq \eta/16$ and every $(\omega_1, \omega_2)\in \mb A^{\omega}\times \mb B^{\omega}$ contributes $1$ to the smoothed incidence count $I(\eta)$. These are all distinct incidences because for $\omega, \omega' \in \mc P$ distinct, either $\mb A^{\omega} \cap \mb A^{\omega'} = \emptyset$ or $\mb B^{\omega}\cap \mb B^{\omega'} = \emptyset$. Indeed, otherwise by symmetry and transitivity $d_{\eta\times \eta\times 1}(\omega \to \omega')\le 2^{-7}$ and $d_{1\times \eta\times \eta}(\omega \to \omega')\le 2^{-7}$ which together contradict $d_{\eta\times \eta\times \eta}(\omega\to \omega') \ge 10$.
We have 
\begin{align*}
    |\mb A^{\omega}| &\gtrsim K^{-1} \frac{|\mb X|}{|\mb X|_{\eta\times \eta\times 1}}, \\ 
    |\mb B^{\omega}| &\gtrsim K^{-1} \frac{|\mb X|}{|\mb X|_{1\times \eta\times \eta}},
\end{align*}
so overall
\begin{align*}
    I(\eta) \gtrsim K^{-2} |\mb X|^2 \frac{|\mb X|_{\eta\times\eta\times\eta}}{|\mb X|_{\eta\times\eta\times 1}|\mb X|_{1\times\eta\times\eta}}.
\end{align*}
\end{proof}

\subsection{Lower bounds for incidences criterion}
Combining the above propositions we obtain the following. Let $\mb X$ be $K$-uniform at scales $\{2^{-jT}\}_{j=0}^m$ and let $f$ be the branching function. Define 
\begin{align}
    b(t\sep x, y) &= f(t, t, t\sep x, y) - (f(t, 0\sep x, y) + f(0, t\sep x, y) - t), \\ 
    e(s\sep x, y) &= \frac{1}{2}(f(s, 0\sep x, y) + f(0, s\sep x, y) - 3s). \label{eq:error_func_defn}
\end{align}
We have 
\begin{equation}\label{eq:b_from_dir}
\begin{split}
    b(t\sep x, y) &= d(t\sep x+t,y) - (f(0, t\sep x, y) - t)  \geq d(t\sep x+t,y) - d(t\sep x,y) - o(1), \\ 
    b(t\sep x, y) &= d^{\vee}(t\sep x,y+t) - (f(t, 0\sep x, y) - t) \geq d^{\vee}(t\sep x,y+t) - d^{\vee}(t\sep x,y) - o(1).
\end{split}
\end{equation}
\Cref{prop:phase_space_initial_estimate} says $B(\delta^t) \geq \delta^{-b(t) + o(1)}$ (note that $b(t)$ is typically a negative number). \Cref{prop:phase_space_high_low}, the high-low inequality, essentially says that if $t' > t$ and $b(t) + e(s) > 0$ for $t \leq s \leq t'$, then $B(\delta^{t'}) \geq \delta^{-b(t)+o(1)}$ as well. 

Motivated by this discussion we say a triple $(t\sep x, y)$ is \textit{$(c_1,c_2)$-effective} for the branching function $f$ if $c_{2} \leq 1/3$, $\max\{t,x,y\} \leq c_2$, and 
\begin{equation}\label{eq:effectiveness_ineq}
    b(t\sep x, y) + e(s\sep x, y) \geq c_1 \qquad \text{for all $t \leq s \leq 1-(x+y)$}. 
\end{equation}
We also assume $t,x,y \in m^{-1}\Z_{\geq 0}$ because the domain of $f$ is $\ms D_m$. 
\begin{prop}\label{prop:effective_incidence_lower_bd}
Let $c_1,c_2 \in (0,1/5)$. There is some large enough integer $m_0 = m_0(c_1) > 0$ so that for every $m > m_0$ and $T > T_0(m, c_1, c_2)$ the following holds. Suppose that $\mb X$ is $(m, T)$-uniform and let $f = f_{\mb X}$ be the branching function. Suppose there is some triple $(t\sep x, y)$ which is $(c_1, c_2)$-effective for $f_{\mb X}$. Then 
\begin{align*}
    B(\delta; \mb X) \geq \delta^{20c_2},\qquad \delta = 2^{-mT}. 
\end{align*}
\end{prop}
\begin{proof}
Let $\mb R \in \mc R_{\delta^{x}\times \delta^{x+y}\times \delta^{y}}$ have $|\mb X \cap \mb R| \gtrsim \ms M_{\delta^{x}\times \delta^{x+y}\times \delta^y}(\mb X)$. Let $\mbt X' = \psi_{\mb R}(\mb X\cap \mb R)$. Using \Cref{lem:exists_uniform_subset}, let $\mbt X\subset \mbt X'$ be a subset which is $K$-uniform on scales $\{2^{-jT}\}_{j=0}^{m(1-x-y)}$ and such that $|\mbt X| \gtrsim K^{-1} |\mbt X'|$. By \Cref{lem:covering_number_rescaling}, for $(x',y',z')\in \ms D_{m(1-x-y)}$,
\begin{align*}
    |\mbt X|_{\delta^{x'}\times \delta^{z'}\times \delta^{y'}} &\sim_{K^7} \frac{|\mb X|_{\delta^{x+x'}\times \delta^{x+y+z'}\times\delta^{y+y'}}}{|\mb X|_{\delta^x\times\delta^{x+y}\times\delta^y}}.
\end{align*}
By \Cref{prop:phase_space_initial_estimate}, 
\begin{align*}
    B(\delta^t; \mbt X) &\gtrsim K^{-2} \frac{|\mbt X|_{\delta^t\times\delta^t\times \delta^t}}{\delta^t |\mbt X|_{\delta^t\times\delta^t\times1}|\mbt X|_{1\times\delta^t\times\delta^t}} \\ 
    &\gtrsim K^{-20} \delta^{-b(t\sep x, y)}. 
\end{align*}
By \Cref{prop:phase_space_high_low}, for any fixed $c>0$,
\begin{align*}
    |B(\delta^t; \mbt X) - B(c\delta^{1-x-y}; \mbt X)| &\leq \sum_{j=mt}^{m(1-x-y)} |B(2^{-jT}; \mbt X) - B(2^{-(j+1)T}; \mbt X)|+ |B(\delta^{1-x-y}; \mbt X) - B(c\delta^{1-x-y}; \mbt X)|  \\ 
    &\lesssim 2^{\frac{3}{2}T} K^{20} \sum_{j=mt}^{m(1-x-y)} \delta^{e(j/m\sep x,y)}.
\end{align*}
Using the hypothesis that $(t\sep x, y)$ is effective, 
\begin{align*}
    B(c\delta^{1-x-y}; \mbt X) &\geq C^{-1}K^{-20} \delta^{-b(t\sep x, y)} - 2^{\frac{3}{2}T}CK^{20} m\sup_{t \leq s \leq 1-x-y} \delta^{e(s; x, y)} \\ 
    &\ge \delta^{-b(t; x, y)}(C^{-1}K^{-20} - 2^{\frac{3}{2}T}CK^{20}m\delta^{c_1}). 
\end{align*}
Notice that $2^T = \delta^{-1/m}$, so if $m > 3c_1^{-1}$, then as $\delta \to 0$ with $m$ fixed the term in parentheses is $\geq\delta^{o(1)}$. Because $t,x,y \leq c_2$, we have $b(t; x, y) \geq -10c_2$ by \Cref{lem:f_is_lip}, so
\[
B(c\delta^{1-x-y}\sep \mbt X) \ge \delta^{11c_2} \qquad \text{if $\delta$ is sufficiently small.}
\]
Finally, we lower bound $I(\delta\sep \mb X)$ as follows. The smoothed incidence count is given by
\begin{align*}
    I(c\delta^{1-x-y}\sep \mbt X) = \sum_{\omega, \omega' \in \mbt X} \eta(d(p_{\omega}, \ell_{\omega'})/(c\delta^{1-x-y})).
\end{align*}
Using the rescaling isometry,
\begin{align*}
    d(p_{\omega}, \ell_{\omega'}) &\geq \frac{1}{2}d_{\square\times 1\times \square}(\omega' \to \omega)  \\ 
    &= \frac{1}{2}\delta^{-(x+y)}d_{\square\times 1\times \square}(\psi_{\mb R}^{-1}(\omega') \to \psi_{\mb R}^{-1}(\omega)) \\ 
    &\geq \frac{1}{2} \delta^{-(x+y)}d(p_{\psi_{\mb R}^{-1}(\omega')}, \ell_{\psi_{\mb R}^{-1}(\omega)}). 
\end{align*}
Thus for $c < 1/2$,
\begin{align*}
    I(c\delta^{1-x-y}\sep \mbt X) &\leq \sum_{\omega, \omega' \in \mbt X} \eta((1/2c)d(p_{\psi_{\mb R}^{-1}(\omega')}, \ell_{\psi_{\mb R}^{-1}(\omega)})/\delta) \\ 
    &\leq I(\delta; \mb X),
\end{align*}
and we obtain the desired lower bound on $I(\delta; \mb X)$. 
\end{proof}

\subsection{Limiting lower bounds criterion}\label{subsec:limiting_lower_bounds}
For any limiting branching function $f \in \mc L$ we can make sense of the quantities $b(t\sep x, y)$, $e(s\sep x, y)$, and of the triple $(t\sep x,y)$ being effective. Recall that $(t\sep x, y)$ is $(c_1,c_2)$-effective if $c_2 \leq 1/3$, $\max\{t,x,y\} \leq c_2$, and 
\begin{equation*}
    b(t\sep x, y) + e(s\sep x, y) \geq c_1 \qquad \text{for all $t \leq s \leq 1-(x+y)$}.
\end{equation*}
We denote by $\mc L^{good(c_2)}\subset \mc L$ the set of $f$ that have a $(c_1, c_2)$-effective triple for some $c_1 > 0$, and $\mc L^{good} = \bigcap_{c_2 > 0} \mc L^{good(c_2)}$. We define $\mc L^{bad(c_2)}$ to be the complement of $\mc L^{good(c_2)}$ in $\mc L$, and $\mc L^{bad}$ to be the complement of $\mc L^{good}$ in $\mc L$. 

The set $\mc L^{good(c_2)}$ is the union over all choices $c_1 > 0$ and $t,x,y \in [0,c_2]$ of the functions $f \in \mc L$ satisfying
\begin{align}\label{eq:strict_effective}
    b(t\sep x, y) + e(s\sep x, y) > c_1 \qquad \text{for all $t \leq s \leq 1-(x+y)$}.
\end{align}
This is an open condition, so $\mc L^{good(c_2)}$ is an open set. The complement $\mc L^{bad(c_2)}$ is closed for every $c_2 > 0$. 

Let $\{(\mb X_k, \delta_k)\}_{k=1}^{\infty}$ be a sequence of pairs where $\mb X_k \subset \Omega$ and $\delta_k \to 0$. For each $m\geq 1$ let $T_{m,k}\geq 1$ be the smallest integer so $2^{-mT_{m,k}} \leq \delta_k$. Let $\mb X_{m,k}\subset \mb X_k$ be an $(m, T_{m,k})$-uniform subset with $|\mb X_{m,k}| \ge K^{-1} |\mb X_k|$, where $K = K(m, T_{m, k})$ as in \Cref{lem:exists_uniform_subset}. Let $f_{m}$ be a subsequential limit of $f_{m,k}$ as $k\to \infty$ and let $f\in \mc L$ be a subsequential limit of $f_m$ as $m\to \infty$. We say $f\in \mc L$ is a limiting branching function of the sequence $\{(\mb X_k, \delta_k)\}_{k=1}^{\infty}$. By compactness in the space of Lipschitz functions any sequence of pairs has a limiting branching function. 

Here is another way to think about limiting branching functions. A function $f \in \mc L$ is a limiting branching function for the sequence if for all $\varepsilon > 0$ and $m > m_0(\varepsilon)$, there is a subsequence $\{(\mb X_j, \delta_j)\}_{j=1}^{\infty}$ satisfying the following. Letting $T_j$ be the smallest integer so $2^{-mT_j} \leq \delta_j$, for each $j \geq 1$ there exists an $(m, T_j)$-uniform subset $\mb X_j' \subset \mb X_j$ such that $\|f_{\mb X_j'} - f\|_{L^{\infty}(\ms D_m)} < \varepsilon$. 
\begin{prop}\label{prop:limiting_inc_lower_bd}
Let $\{(\mb X_k, \delta_k)\}_{k=1}^{\infty}$ be a sequence so that $\delta_k \to 0$ and any limiting branching function for this sequence is in $\mc L^{good}$. For any $\varepsilon > 0$, if $k > k_0(\varepsilon)$ then
\begin{align*}
    B(\delta_k\sep \mb X_k) \geq \delta_k^{\varepsilon}.
\end{align*}
\end{prop}
\begin{proof}
Suppose the conclusion does not hold. After passing to a subsequence we may assume that for some $\varepsilon > 0$, $B(\delta_k\sep \mb X_k) \leq \delta_k^{\varepsilon}$ for all $k\geq 1$. 

Let $f \in \mc L$ be a limiting branching function of the sequence $\{(\mb X_k, \delta_k)\}_{k=1}^{\infty}$. By hypothesis $f \in \mc L^{good}$, so for any $c_2 \in (0,1/5)$ there is some $c_1 > 0$ and a $(c_1, c_2)$-effective triple $(t\sep x, y)$ for $f$. 

Fix a $(c_1,c_2)$-effective triple $(t\sep x, y)$ for $f$. Let $\zeta = c_1/1000$. For any $m > m_0(\zeta)$, there is a subsequence of pairs $\{(\mb X_{j}, \delta_{j})\}_{j=1}^{\infty}$ satisfying the following. Letting $T_j$ be the smallest integer so $2^{-mT_j} < \delta_{j}$, there is an $(m, T_j)$-uniform subset $\mb X_j' \subset \mb X_j$ so that $\| f_{\mb X_j'} - f \|_{L^{\infty}(\ms D_m)} \leq \zeta$. Let $t' \in m^{-1}\Z_{\geq 0}$ be the largest number $\leq t$, and similarly for $x'$ and $y'$. For $s' \in [t', 1-(x'+y')]$, let $s \in [t, (1-x-y)]$ be the closest number to $s'$ in $m^{-1}\Z$. Let $b(t\sep x, y)$ and $e(s\sep x, y)$ be the initial estimate and error functions corresponding to $f$. Then if we pick $m > 1/\zeta$,
\begin{align*}
    |(b(t'\sep x', y') + e(s'\sep x',y')) - (b(t\sep x, y) + e(s\sep x, y))| \leq 20/m < c_1 / 50
\end{align*}
using the Lipschitz property. We have 
\begin{align*}
    |(b_{f_{\mb X_j'}}(t'\sep x', y') + e_{f_{\mb X_j'}}(s'\sep x',y')) - (b(t'\sep x', y') + e(s'\sep x',y'))| \leq 20\| f_{\mb X'} - f \|_{L^{\infty}(\ms D_m)} < c_1/50. 
\end{align*}
Combining these estimates, the triple $(t'\sep x',y')$ is $(c_1/2, c_2)$-effective for $f_{\mb X_j'}$ for all $j \geq 1$. 

Choose $m > m_0$ from \Cref{prop:effective_incidence_lower_bd}. By \Cref{prop:effective_incidence_lower_bd}, for $j$ large enough, 
\begin{align*}
    B(\delta_j'; \mb X_j') \geq (\delta_j')^{20c_2},\qquad \delta_j' = 2^{-mT_j}. 
\end{align*}
Because $2^{-T_j}\delta_j \leq \delta_j' \leq \delta_j$ and $|\mb X_j'| \geq K^{-1} |\mb X_j|$ where $K = K(m_j, T)$ is the uniformity constant, 
\begin{align*}
    B(\delta\sep \mb X_j) \geq \delta_j^{20c_2+o_j(1)}.
\end{align*}
Choosing $c_2 < \varepsilon / 100$ and $j$ large enough gives a contradiction. 
\end{proof}

\section{Proof of main results}\label{sec:pf_main_results}

\subsection{Proof of \Cref{thm:incidence_lower_bd}}\label{subsec:pf_thm_inc_lower_bd}

First we recall the definition of a $(\delta, \alpha, \beta, C)$-set, which we defined in \cref{subsec:intro_statement_main_thm} before building up the notation of \cref{sec:space_pt_line_pairs}.

\begin{definition}
A set $\mb X\subset \Omega$ is a $(\delta, \alpha, \beta, C)$-set if it is $\delta$-separated in the $d_{1\times 1\times 1}$ distance (this is equivalent to the $\R^3$ metric) and satisfies the following Frostman condition. For any $u, w \in (0,1]$ with $uw \geq \delta$ and any rectangle $\mb R = \mb R_{u\times uw\times w}(\omega)$, we have
\begin{align*}
    |\mb X\cap \mb R| \leq Cu^{\alpha} w^{\beta}\, |\mb X|. 
\end{align*}     
\end{definition}

If $\mb X$ is a $(\delta, \alpha, \beta, \delta^{-\eta})$-set which is $(m, T)$ uniform, then $f_{\mb X}(x,y) \geq \alpha x + \beta y - \eta - o(1)$. We denote by $\mc L_{\alpha, \beta}$ the functions in $\mc L$ which additionally satisfy 
\begin{align*}
    f(x,y) \geq \alpha x + \beta y.
\end{align*}
The class of Lipschitz functions $\mc L_{\alpha, \beta}$ captures the branching structure of $(\alpha, \beta)$-Frostman sets. The following is our main theorem about branching functions. Recall from \cref{subsec:limiting_lower_bounds} that $\mc L^{good}$ is the limiting branching functions for which we can prove an incindence lower bound.
\begin{theorem}\label{thm:lip_func_theorem}
If $\alpha + \beta > 3$ then $\mc L_{\alpha,\beta} \subset \mc L^{good}$. 
\end{theorem}
We combine \Cref{thm:lip_func_theorem} with \Cref{prop:limiting_inc_lower_bd} to prove our main incidence lower bound result \Cref{thm:incidence_lower_bd}, which we restate below. 
\begin{theorem*}
Let $\alpha, \beta \in [1,2]$ satisfy $\alpha+\beta > 3$, and let $\varepsilon > 0$. There exists $\eta = \eta(\alpha, \beta, \varepsilon) > 0$ such that the following holds for all  $\delta < \delta_0(\alpha, \beta, \varepsilon)$. Suppose $\mb X \subset \Omega$ is an $(\delta, \alpha, \beta, \delta^{-\eta})$-set. Then 
\begin{align*}
    I(\delta; \mb X) \geq \delta^{1+\varepsilon}|\mb X|^2.
\end{align*}
\end{theorem*}
\begin{proof}
Suppose the conclusion fails. Then for some $\varepsilon > 0$ there is a sequence $\eta_k \to 0$ and $\delta_k \to 0$ and a sequence of $(\delta_k, \alpha, \beta, \delta^{-\eta_k})$ sets $\mb X_k \subset \Omega$ such that $B(\delta_k\sep \mb X_k) \leq \delta^{\varepsilon}$. Any limiting branching function of this sequence lives in $\mc L_{\alpha, \beta}$, so \Cref{prop:limiting_inc_lower_bd} gives a contradiction. 
\end{proof}

\subsection{Proof of \Cref{thm:exists_extra_incidence}}\label{subsec:pf_exists_extra_inc}
For the reader's convenience we restate \Cref{thm:exists_extra_incidence} below. 
\begin{theorem*}
Fix $t \in [1, 2]$ and $s \in [0,1]$ such that $2t+s>3$. There exists $\eta(t, s) > 0$ such that the following holds for all $\delta < \delta_0(t, s)$. 
Let $P \subset [0,1]^2$ be a set of $\delta^{-t}$ many points. For each $p \in P$ let $\T_p$ be a $(\delta, s, \delta^{-\eta})$-set of $\delta$-tubes through $p$ and let $\T = \bigsqcup_{p} \T_p$. Then there is some nontrivial incidence between $P$ and $\T$, meaning there is a point $p \in P$ and a tube $T \in \T \setminus \T_{p}$ so that $p \in T$.  
\end{theorem*}

We proceed by contradiction. 
Fix $t\in [1,2]$ and $s\in [0,1]$ such that $2t+s > 3$, and let $\eta(s,t) \in (0,1/10)$ be a small constant to be chosen later.  
Let $P \subset [-1,1]^2$ be a set of $\tilde \delta^{-t}$-many points (we use the notation $\tilde \delta$ because we will modify $\tilde \delta\to \delta$ for the bulk of the proof). 
For each $p \in P$, let $\T_p$ be a $(\tilde \delta, s, \tilde \delta^{-\eta})$ set of $\tilde \delta$-tubes through $p$, and suppose there is no nontrivial incidence between $P$ and $\T = \bigsqcup_p \T_p$.  For $\eta$ and $\tilde \delta$ small enough in terms of $s$ and $t$ we will obtain a contradiction. 

After passing to a constant size subset we may assume all the tubes have slope in $[-1,1]$ relative to the $x$-axis.
By \Cref{lem:katz_tao_subset}, for each $p\in P$ there exists a subset $\T_p' \subset \T_p$ such that $|\T_p'| = c\tilde \delta^{-s+\eta}$ and 
\begin{align}
    |\slope(\T_p') \cap I| \lesssim (w/\tilde \delta)^s \qquad \text{for all intervals $I\subset [-1,1]$ with length $w\in [\tilde \delta, 1]$.}
\end{align}

Let $m = \lceil \eta^{-1} \rceil$, and let $T \geq 0$ be the smallest integer such that $2^{-mT} \leq \tilde \delta$. Set $\delta = 2^{-mT}$. 
Using our new choice $\delta \in [2^{-T}\tilde \delta, \tilde \delta]$, we find
\begin{align*}
    |P| &\geq (2^T \delta)^{-t} \geq \delta^{-t+t/m}  \geq \delta^{-t+2\eta}, \\ 
    |\T_p'| &\geq c(2^T \delta)^{-s+\eta} \geq c \delta^{-s+2\eta},
\end{align*}
and 
\begin{align}
    |\slope(\T_p') \cap I| \lesssim (w/\delta)^s \qquad \text{for all intervals $I\subset [-1,1]$ with length $w\in [\delta, 1]$.}\label{eq:KatzTaoCondTp}
\end{align}

We may construct a set in phase space from $P$ and $\T$. 
Let $p \in P$ be a point and let $T$ be a tube through $p$, meaning that the central line $\ell_T$ passes through $p$. We associate to this pair the point $\omega \in \Omega$ where $p_{\omega} = p$ and $\ell_{\omega} = \ell_T$. For each $p\in P$, let $\mb X_p \subset \Omega$ be the subset associated to all the point-tube pairs $\{p \in T\}_{T\in \T_p'}$. We set $\mb X = \bigsqcup_p \mb X_p$. Notice that 
\begin{equation*}
    c\delta^{-t-s+4\eta} \leq |\mb X| \leq \delta^{-t-s}.
\end{equation*}

Let $\mb X' \subset \mb X$ be an $(m, T)$-uniform subset with $|\mb X'| \geq K^{-1} |\mb X|$, where $K = K(m, T)$ as in \Cref{lem:exists_uniform_subset}, and let $f: \ms D_m \to \R_{\geq 0}$ be the branching function of $\mb X'$. In what follows, $o(1)$ denotes a quantity that goes to zero as $\delta \to 0$ with $s,t,\eta$ fixed. 

\noindent \textbf{Establishing separation and regularity of $\mb X$.}
Because we assumed there are no nontrivial incidences, for any $\omega \in \mb X_p$ and $\omega' \in \mb X_{p'}$ with $p\neq p'$, 
\begin{align}\label{eq:incidence_sep_X}
    d_{\square\times 1\times \square}(\omega \to \omega') \geq \delta / 20
\end{align}
where we add a factor of $20$ to convert from Euclidean point-line distance to vertical point-line distance, and use $\delta \leq \tilde \delta$. Because each tube set $\T_p$ is $\delta$-separated, for $\omega, \omega' \in \mb X_p$ we have 
\begin{align}\label{eq:Xp_slope_sep}
    d_{\square\times \square\times 1}(\omega \to \omega') \geq \delta / 20
\end{align}
where we add the factor of 20 to convert from the line metric \eqref{eq:line_metric_all_slope} in the introduction to the $d_{\square\times \square\times 1}$ distance, which measures the difference in slopes. Combining \eqref{eq:incidence_sep_X} and \eqref{eq:Xp_slope_sep}, for any distinct $\omega, \omega' \in \mb X$
\begin{equation}\label{eq:sep_X_box11}
    d_{\square\times 1\times 1}(\omega \to \omega') \geq \delta / 20.
\end{equation}
Because of the Katz-Tao condition \eqref{eq:KatzTaoCondTp}, for any $w\in [\delta, 1]$ and $\mb R \in \mc R_{1\times 1\times w}$,
\begin{equation}\label{eq:sreg_Xp}
    |\mb X_p\cap \mb R| \lesssim (w/\delta)^s.
\end{equation}
Now let $u, w \in [\delta, 1]$ satisfy $uw < \delta$, and let $\mb R \in \mc R_{u\times \delta\times w}$ be a rectangle. Because of \eqref{eq:incidence_sep_X} there are $\lesssim 1$ many $p\in P$ where $\mb X_p$ intersects $\mb R$, so \eqref{eq:sreg_Xp} implies
\begin{equation}\label{eq:X_intersect_R}
    |\mb X\cap \mb R| \lesssim (w/\delta)^s. 
\end{equation}

The concentration inequality \eqref{eq:X_intersect_R} implies that 
\begin{equation}\label{eq:concentration_Xp}
    \ms M_{\delta^{1-y}\times \delta \times \delta^{y}}(\mb X')\leq  \delta^{s(y-1)-o(1)}.
\end{equation}
The concentration numbers of $\mb X'$ are related to the branching function $f$ by
\begin{align*}
    \ms M_{\delta^x\times \delta^z\times \delta^y}(\mb X') \geq \frac{|\mb X'|}{|\mb X'|_{\delta^x\times \delta^z\times \delta^y}} = |\mb X'| \delta^{f(x,y,z)}\qquad \text{for $(x,y,z) \in \ms D_m$.}
\end{align*}
We have $|\mb X| \geq c\delta^{-t-s+4\eta}$ and $|\mb X'|\geq \frac{1}{K} |\mb X|$, so 
\begin{equation*}
    \ms M_{\delta^x\times \delta^z\times \delta^y}(\mb X') \geq \delta^{-t-s+f(x,y,z)+4\eta+o(1)},
\end{equation*}
and combining the above inequality with \eqref{eq:concentration_Xp} yields 
\begin{equation}\label{eq:f_at_critical}
    f(1-y,y) \geq t+sy - 4\eta - o(1)\qquad \text{for $y \in \frac{1}{m}\Z \cap [0,1]$.}
\end{equation}
We split into two cases depending on whether $s+t>2$ or $s+t\leq 2$. 

\medskip
\noindent \textbf{Case 1: $s+t>2$.}
Equation \eqref{eq:f_at_critical} implies that $f(0,1) \geq t+s-4\eta-o(1)$. However, the Lipschitz property of $f$ (\Cref{lem:f_is_lip}) implies $f(0,1) \leq 2+o(1)$. If we choose $\eta < (t+s-2)/4$ and $\delta$ sufficiently small, these inequalities form a contradiction. 

\medskip
\noindent \textbf{Case 2: $s+t\leq 2$.}

\noindent \textbf{Blowing up into a regular rectangle.}
For any $(x,y) \in \ms D_m$, combine \eqref{eq:f_at_critical} with the Lipschitz property (\Cref{lem:f_is_lip}) to find
\begin{equation}\label{eq:f_lower_bd_1}
    f(x,y) \geq f(x,1-x) - 2(1-x-y) - o(1) \geq (t+s-2) + (2-s)x + 2y - 4\eta - o(1).
\end{equation}
Fix
\begin{align*}
    \rho = \frac{1}{2}\min\{(2t+s-3)/2, t-1\} > 0.
\end{align*}
This is where we use the hypothesis that $2t+s>3$. It follows from $2t+s>3$ that $t>1$. Let $\alpha = t-\rho$ and $\beta = t+s-\rho$. Because we chose $\rho < (2t+s-3)/2$, we find $\alpha +\beta > 3$. We selected this choice of $\alpha$ and $\beta$ so that when $x+y=1$, \eqref{eq:f_at_critical} yields
\begin{equation*}
    f(x,y) \geq \alpha x + \beta y + \rho(x+y) - 4\eta - o(1).
\end{equation*}
Choose $(x_0,y_0) \in \ms D_m$ to minimize the quantity 
\begin{align}\label{eq:functional_minimize_f}
    f(x_0, y_0) - \alpha x_0 - \beta y_0.
\end{align}
This minimum value is less than $o(1)$ because $f(0,0) \leq o(1)$. By applying \eqref{eq:f_lower_bd_1} and plugging in the values of $\alpha, \beta$ we find
\begin{align*}
    f(x_0, y_0) - \alpha x_0 - \beta y_0 \geq -(2-t-s)(1-x_0-y_0) + \rho(x_0+y_0) - 4\eta - o(1).
\end{align*}
Using that the left hand side is $\leq o(1)$ and rearranging, we find
\begin{align*}
    1-x_0-y_0 \geq \frac{\rho - 4\eta -o(1)}{\rho + 2-t-s}.
\end{align*}
The denominator is always $\leq 2$, and if $\eta$ and $\delta$ are chosen small enough the numerator is $\geq \rho/2$, so 
\begin{align*}
    x_0+y_0 \leq 1 - \rho / 4 \qquad \text{if $\delta, \eta$ are small enough.}
\end{align*}
From now on we assume $\delta, \eta$ are small enough that the above holds. 
Let $\mb R \in \mc R_{\delta^{x_0}\times \delta^{x_0+y_0}\times \delta^{y_0}}$ be a rectangle containing nearly the maximum number of points, 
\begin{align*}
    |\mb X'\cap \mb R| \geq \frac{1}{CK} \delta^{f(x_0,y_0)}|\mb X'|. 
\end{align*}
Set $\mb X_1 = \psi_{\mb R}(\mb X'\cap \mb R)$. Let 
\begin{align*}
    m_1 &= m(1-x_0-y_0) \in \Z_{\geq 0},  \\  
    \delta_1 &= 2^{-m_1T} = \delta^{1-x_0-y_0} \leq \delta^{\rho/4}.
\end{align*}
Produce an $(m_1, T)$-uniform subset $\mb X_1' \subset \mb X_1$ with $|\mb X_1'| \geq \frac{1}{K} |\mb X_1|$. As we remarked after \Cref{lem:exists_uniform_subset}, we may assume 
\begin{itemize}
    \item There is some number $r > 0$ such that for every cube $\mb Q \in \mc D_{2^{-40}\delta_1}(\mb X_1')$, $|\mb X_1'\cap \mb Q| \in [r/2, r]$. 
    \item No two dyadic cubes in $\mc D_{2^{-40}\delta_1}(\mb X_1')$ are adjacent (including diagonal adjacencies). 
    \item For every $\omega \in \mb X_1'$ and admissible scale $u\times v\times w$ in $\{2^{-jT}\}_{j=0}^{m_1}$,
    \begin{align*}
        \sum_{\mb Q \in \mc D_{2^{-40}\delta_1}, \mb Q \subset \mb R_{u\times v\times w}(\omega)} |\mb X_1'\cap \mb Q| \geq \frac{1}{K} \ms M_{u\times v\times w}(\mb X_1'). 
    \end{align*}
\end{itemize}
Let $\mb X_1''$ be obtained by selecting one point of $\mb X_1'$ in every dyadic cube $\mc D_{2^{-40}\delta_1}(\mb X_1')$. By the first property above, $|\mb X_1''| \in [|\mb X_1'|/r, 2|\mb X_1'|/r]$.
By the second property, $\mb X_1''$ is $2^{-40}\delta_1$-separated in the $\R^3$ metric. By the third property, for every $\omega \in \mb X_1'$ and admissible scale $u\times v\times w$ in $\{2^{-jT}\}_{j=0}^{m_1}$,
\begin{align*}
    |\mb X_1''\cap \mb R_{u\times v\times w}(\mb \omega)| \geq \frac{1}{Kr} \ms M_{u\times v\times w}(\mb X_1'). 
\end{align*}
Finally, for any $\omega \in \R^3$
\begin{align*}
    |\mb X_1''\cap \mb R_{u\times v\times w}(\mb \omega)| &\leq \#\{\mb Q \in \mc D_{2^{-40}\delta_1}\, :\, \mb Q \cap \mb R_{u\times v\times w}(\mb \omega) \neq \emptyset\}  \\ 
    &\leq \frac{1}{r/2}\sum_{\mb Q \in \mc D_{2^{-40}\delta_1}, \mb Q \cap \mb R_{u\times v\times w}(\omega) \neq \emptyset} |\mb X_1'\cap \mb Q| \\ 
    &\lesssim \frac{1}{r} \ms M_{u\times v\times w}(\mb X_1'). 
\end{align*}
These upper and lower bounds together show that $\mb X_1''$ is $CK$-uniform on the same set of scales as $\mb X_1'$, and the covering numbers are the same up to factors of $CK^2$ because for any $(x,y,z) \in \ms D_{m_1}$,
\begin{align*}
    |\mb X_1''|_{\delta^{x}\times \delta^z\times \delta^y} \sim_K \frac{|\mb X_1''|}{\ms M_{\delta^{x}\times \delta^z\times \delta^y}(\mb X_1'')} \sim_{K} \frac{|\mb X_1'|}{\ms M_{\delta^{x}\times \delta^z\times \delta^y}(\mb X_1')} \sim_{K^2} |\mb X_1'|_{\delta^{x}\times \delta^z\times \delta^y},
\end{align*}
thus
\begin{equation*}
    f_{\mb X_1''}(x,y,z) = f_{\mb X_1'}(x,y,z) \pm o(1). 
\end{equation*}
By applying \Cref{lem:covering_number_rescaling} to $\mb X_1'$, $f_{\mb X_1''}$ can be written in terms of $f = f_{\mb X'}$ as follows,
\begin{equation}\label{eq:branching_func_X1pp}
    f_{\mb X_1''}(x,y,z) = \frac{1}{1-x_0-y_0}f((1-x_0-y_0)x, (1-x_0-y_0)y, (1-x_0-y_0)z; x_0,y_0) \pm o(1). 
\end{equation}
Because $(x_0, y_0)$ was chosen to minimize \eqref{eq:functional_minimize_f}, $f(x, y; x_0, y_0) \geq \alpha x + \beta y$ whenever the inputs are defined. Using the above formula for $f_{\mb X_1''}(x,y,z)$, 
\begin{align}\label{eq:branching_func_lower_bd_applied}
    f_{\mb X_1''}(x,y) \geq \alpha x + \beta y - o(1)\qquad \text{for all $(x,y) \in \ms D_{m_1}.$}
\end{align}
It follows that 
\begin{align*}
    \ms M_{\delta_1^x\times \delta_1^{x+y}\times \delta_1^z}(\mb X_1'') \lesssim \frac{1}{K} \delta_1^{f_{\mb X_1'}(x,y,z)} |\mb X_{1}''| \leq \delta_1^{\alpha x + \beta y - o(1)}|\mb X_1''|. 
\end{align*}
For any $u, w \leq 1$ with $uw \geq \delta_1$, let $(x,y) \in \ms D_m$ be maximal such that $\delta_1^x \geq u$ and $\delta_1^y \geq w$. Then $\delta_1^x\leq \delta_1^{-1/m} u \leq \delta_1^{\eta} u$ and $\delta_1^y \leq \delta_1^{-1/m} w \leq \delta_1^{\eta}w$, so $\mb X_1''$ satisfies the Frostman condition
\begin{align*}
    \ms M_{u\times uw\times w}(\mb X_1'') &\leq \delta_1^{\alpha x + \beta y-o(1)}|\mb X_1''| \\ 
    &\leq u^{\alpha} w^{\beta} \delta_1^{-o(1)-(\alpha+\beta)\eta} |\mb X_1''|. 
\end{align*}
If we only assume $ uw \geq \delta_1 / A $ for some constant $ A $, the inequality remains valid because the constant factor loss introduced by  $ A $ is absorbed into the $ \delta_1^{-o(1)} $ term. Consequently, for sufficiently large constants $ A $, $ \mathbf{X}_1'' $ is a $(\delta_1 / A, \alpha, \beta, \delta^{-\eta_1})$-set, where $ \eta_1 = (\alpha + \beta)\eta + o(1) $. The factor of $ A $ is necessary to ensure that $ \mathbf{X}_1'' $ is $ 1 $-separated in the $ d_{\delta_1 / A \times \delta_1 / A \times \delta_1 / A} $ distance, rather than being $ 2^{-40} \delta_1 $-separated in the Euclidean distance.

\medskip
\noindent \textbf{Applying \Cref{thm:incidence_lower_bd}.}
Fix $\varepsilon = (t-1-\rho)/2$ (this is where we use $\rho < t-1$) and a large constant $A$. For $\eta < \eta_0(s,t)$ and $\delta < \delta_0(s,t,\eta)$, we are guaranteed that 
\begin{itemize}
    \item $\alpha+\beta>3$ because $\rho < (2t+s-3)/2$,
    \item $\delta_1 < \delta_0(\alpha, \beta, \varepsilon)$ from \Cref{thm:incidence_lower_bd},
    \item $\mb X_1''$ is a $(\delta_1/A, \alpha, \beta, \delta^{-\eta_1})$-set where $\eta_1 < \eta(\alpha, \beta, \varepsilon)$ from \Cref{thm:incidence_lower_bd}.
\end{itemize}
By \Cref{thm:incidence_lower_bd}, 
\begin{align}\label{eq:applied_inc_lower_bd}
    I(\delta_1/A; \mb X_1'') &\geq \delta_1^{1+\varepsilon} |\mb X_1''|^2.
\end{align}
We can compare this lower bound to an upper bound coming from the no extra incidences hypothesis. Suppose $\omega, \omega' \in \mb X_1''$ come from distinct base points $p , p' \in P$.  By \eqref{eq:incidence_sep_X},
\begin{align*}
    d_{\square\times 1\times \square}(\psi_{\mb R}^{-1}(\omega)\to \psi_{\mb R}^{-1}(\omega')) \geq \delta / 20. 
\end{align*}
By the rescaling isometry,
\begin{align*}
    d_{\square\times 1\times \square}(\psi_{\mb R}^{-1}(\omega)\to \psi_{\mb R}^{-1}(\omega')) = \delta^{x_0+y_0}d_{\square\times 1\times \square}(\omega \to \omega'),
\end{align*}
so in fact
\begin{align*}
    d_{\square\times 1\times \square}(\omega\to \omega') \gtrsim \delta^{1-x_0-y_0} \gtrsim \delta_1,
\end{align*}
and if $A$ is chosen large enough, the pair $(\omega, \omega')$ contributes nothing to the incidence count $I(\delta_1/A; \mb X_1'')$. For each $\omega \in \mb X_1''$, the number of points in $\mb X_1''$ coming from the same base point is at most $\ms M_{\delta_1\times \delta_1\times 1}(\mb X_1'')$. Thus 
\begin{align}
    I(\delta_1/A; \mb X_1'') &\lesssim |\mb X_1''|\, \ms M_{\delta_1\times \delta_1\times 1}(\mb X_1'') \nonumber \\ 
    &\lesssim|\mb X_1''|^2 \delta_1^{f_{\mb X_1''}(1, 0)-o(1)}. \label{eq:upper_bd_inc_in_rect}
\end{align}
By \eqref{eq:branching_func_lower_bd_applied}, $f_{\mb X_1''}(1,0) \geq \alpha - o(1) = t-\rho - o(1)$.
Comparing the upper bound \eqref{eq:upper_bd_inc_in_rect} with the lower bound \eqref{eq:applied_inc_lower_bd}, we find 
\begin{align*}
    \frac{\text{Upper bound \eqref{eq:upper_bd_inc_in_rect} for incidences}}{\text{Lower bound \eqref{eq:applied_inc_lower_bd} for incidences}} \lesssim \delta_1^{f_{\mb X_1''}(1, 0) - 1 - \varepsilon- o(1)} \leq \delta^{(\rho/4)(t - 1 -  \rho - \varepsilon) - o(1)}
\end{align*}
Because we chose $\varepsilon < t-1-\rho$, as long as the $o(1)$ term is small enough and $\delta$ is small enough the above inequality yields a contradiction. 

\subsection{Proof of \Cref{thm:inc_lower_bd_t_set}}
Fix $t \in [1,2]$ and $s \in [0,1]$ such that $t+s > 2$. Let $\varepsilon > 0$ be arbitrary and $\eta = \eta(\varepsilon) > 0$ small. Consider a $(\delta, t, \delta^{-\eta})$-set of points $P \subset [0,1]^2$ along with a $(\delta, s, \delta^{-\eta})$-set of tubes $\T_p$ through each $p \in P$, and set $\T = \bigsqcup_{p\in P} \T_p$. We would like to show $I(P, \T) \gtrsim \delta^{1+\varepsilon}|P|\, |\T|$. As in the proof of \Cref{thm:exists_extra_incidence}, let $\mb X \subset \Omega$ be the set in configuration space corresponding to $P, \T$ (see \cref{subsec:pf_exists_extra_inc}).

As discussed in \cref{sec:pf_sketch}, if we assumed an Ahlfors-David regularity condition on $P$ rather than just a Frostman condition, it would follow that $\mb X$ is a $(\delta, t,s+1, C)$-set, and the incidence lower bound would follow from \Cref{thm:incidence_lower_bd}. Moreover, we would be in the special case of \cref{sec:horizontal_good_line} where the entire horizontal line is $(t,s+1)$-good. 
Morally we can reduce to the case where $P$ is AD-regular. To execute this reduction we use limiting Lipschitz functions.

Let $\mb X_k\subset \Omega$ be a sequence of configurations as described above with parameters $\delta_k \to 0$ and $\eta_k \to 0$. Let $f\in \mc L$ be a subsequential limit of branching functions of uniform subsets of $\mb X_k$. We must have 
\begin{align}
    f(x, 0) &\geq tx \qquad \text{for all $0 \leq x \leq 1$}, \label{eq:f_points_t_frost} \\  
    d(y\sep 1, 0) &\geq sy \qquad \text{for all $0 \leq y \leq 1$}. \label{eq:f_lines_s_frost}
\end{align}
Let $\mc L_{t;s}\subset \mc L$ be the limiting branching functions that satisfy \eqref{eq:f_points_t_frost} and \eqref{eq:f_lines_s_frost}. The combination of \Cref{prop:limiting_inc_lower_bd} and the following proposition proves \Cref{thm:inc_lower_bd_t_set}.
\begin{prop}
If $s+t > 2$ then $\mc L_{t;s}\subset \mc L^{good}$.
\end{prop}
\begin{proof}
Suppose by way of contradiction that for some $c_2 > 0$, $\tilde f \in \mc L_{t; s} \cap \mc L^{bad(c_2)}$. We will locate an $x_* \in [0,c_2/100]$, take blowups of $\tilde f$ near $(x_*, 0)$, and use these blowups to produce an $f \in \mc L_{t; s}\cap \mc L^{bad}$ which is easier to analyze. 

\medskip 
\noindent \textbf{Locating a point $(x_*, 0)$ to blowup near.}
Choose some $t' \in (1,t)$ such that $t'+s>2$. Let
\begin{align*}
    E = \{x\in [0,1]\, :\, \tilde f(y, 0) \geq \tilde f(x,0) + t'(y-x) \text{ for all } y \in [x,1]\}. 
\end{align*}
We consider the following non-decreasing sequence $\{x_j\}_{j=0}^{\infty} \subset E$,
\begin{align*}
    x_j = \sup\, \{x\in [0,1]\, :\, \tilde f(x,0) - t' x \leq j\delta\}.
\end{align*}
Let $a < 1$. By \eqref{eq:f_points_t_frost}, if $j \leq a(t-t')/\delta$ then $x_j \leq a$. Also, $\tilde f(1, 0) - t' > j\delta$, so using continuity of $\tilde f$, $$\tilde f(x_j, 0) - t' x_j = j\delta.$$ Using the fact that $x \to \tilde f(x,0)$ is $2$-Lipschitz, for $j \leq a(t-t')/\delta - 1$
\begin{align*}
    \delta = |(\tilde f(x_j, 0) - t' x_j) - (\tilde f(x_{j+1}, 0) - t' x_{j+1})| \leq (2+t')|x_j-x_{j+1}| \leq 4 |x_j-x_{j+1}| 
\end{align*}
Thus $\{x_j\}_{j=0}^{\lfloor a(t-t')/\delta \rfloor}$ is a $\delta/4$-separated set contained in $E\cap [0,a]$, so the $\delta$-covering number $|E\cap [0,a]|_{\delta}$ is at least $\frac{1}{C}(a(t-t')/\delta - 1)$. 
Because $E$ is a closed set, the Lebesgue measure is lower bounded by 
\begin{align*}
    |E\cap [0,a]| \gtrsim \limsup_{\delta \to 0} \delta a|E\cap [0,a]|_{\delta} \gtrsim a(t-t'). 
\end{align*}
By Rademacher's theorem \cite[Theorem 7.20]{rudin}, for almost every $x \in [0,1]$, $\tilde f(x,0)$ is differentiable at $x$. Choose some $x_* \in [0,c_2/100]\cap E$ that is differentiable. Because $x_* \in E$, 
\begin{equation*}
    \frac{d}{dx}\tilde f(x_*, 0) = t_* \geq t'.
\end{equation*} 

\medskip 
\noindent \textbf{Lower bounds on $\tilde f$ near $(x_*, 0)$.}
By the definition of $E$, 
\begin{equation}\label{eq:tilde_f_lower_bd_x}
    \tilde f(x_*+u, 0) \geq \tilde f(x_*, 0) + t' u. 
\end{equation}
Next we use the Furstenberg set estimate \cref{subsec:limiting_branching_func} \cref{exampleitem:info_L_furstenberg} to bound $f(x_*, u)$ from below. Actually, we use the regime where $s+t>2$ and $t>1$, which is due to Fu and Ren \cite{fu2021incidence}. For $w \in (0, 1-x_*]$, let 
\begin{align*}
    f_{w}(x,y,z) &= \frac{1}{w}\tilde f(w x,  w y, w z\sep x_*, 0).
\end{align*}
By \cref{subsec:limiting_branching_func} \cref{item:limiting_branching_func_blowup}, $f_w$ lies in $\mc L$. By \eqref{eq:tilde_f_lower_bd_x}, 
\begin{align*}
    f_w(x, 0) \geq t' x\qquad \text{for $x \in [0,1]$.}
\end{align*}
The direction numbers are bounded below by  
\begin{align*}
    d_{f_w}(y; 1, 0) &= f_w(1,y,1) - f_w(1,0,1) \\ 
    &= \frac{1}{w}(\tilde f(w, wy, w \sep x_*, 0) - \tilde f(w, 0, w\sep x_*, 0) \\  
    &= \frac{1}{w} d_{\tilde f}(wy\sep x_*+w, 0) \\ 
    &\geq d_{\tilde f}(wy; 1, 0) \geq sy. 
\end{align*}
where in the last line we use monotonicity of $d_{\tilde f}(y\sep x, 0)$ in the $x$-variable, see \Cref{lem:dir_set_props}\ref{lemitem:dir_lip}. Because $t'+s>2$ and $t' > 1$, the Furstenberg set estimate (\cref{subsec:limiting_branching_func} \cref{exampleitem:info_L_furstenberg}) implies 
\begin{align*}
    f_w(0,1) \geq 1+s, 
\end{align*}
which implies
\begin{align}\label{eq:tilde_f_lower_bd_y}
    \tilde f(x_*, u) \geq \tilde f(x_*, 0) + u(1+s)\qquad \text{for $u \in [0, 1-x_*]$.}
\end{align}
Combining \eqref{eq:tilde_f_lower_bd_x} and \eqref{eq:tilde_f_lower_bd_y}, we estimate the high-low error function by
\begin{align}\label{eq:error_func_f_tilde}
    e_{\tilde f}(u\sep x_*, 0) = \frac{1}{2}(\tilde f(x_*+u) + \tilde f(x_*, u) - 2f(x_*, 0) - 3u) \geq \frac{t'+s-2}{2}u.
\end{align}

\medskip 
\noindent \textbf{Microscopic blowup.} 
The high-low error function estimate \eqref{eq:error_func_f_tilde} allows us to apply \Cref{lem:zooming_in} with the choice $r = (t'+s-2)/4$ to show $f_w \in \mc L^{bad}$.
Suppose $f_w \in \mc L^{good(c_2')}$ for $c_2' = (t'+s-2)/50$. By \Cref{lem:zooming_in}, this implies that $\tilde f \in \mc L^{good(x_* + c_2' w)}$. Because $x_* \leq c_2/100$, if $w \leq c_2 / 100$ then $\tilde f \in \mc L^{good(c_2)}$, contradicting our hypothesis that $\tilde f \in \mc L^{bad(c_2)}$. Overall, we find 
\begin{align*}
    f_w \in \mc L^{bad(c_2')} \qquad \text{for all $w \in (0,c_2/100]$}.
\end{align*}
Let $w_j \to 0$ be a sequence along which $f_{w_j}$ converges uniformly to some branching function $f \in \mc L$. Because $\mc L^{bad(c_2')}$ is a closed set (see the discussion at the beginning of \cref{subsec:limiting_lower_bounds}) and $f_w \in \mc L^{bad(c_2')}$ for sufficiently small $w$, we find $f \in \mc L^{bad(c_2)}$. 

We will prove that $f \in \mc L_{t_*, s+1}$. 
First of all, because $x \to \tilde f(x,0)$ is differentiable at $x = x_*$ with derivative $t_*$, 
\begin{equation}\label{eq:f_AD_reg}
    f(x, 0) = t_* x \qquad \text{for $x \in [0,1].$}
\end{equation}
Let $x_0 \in [0,1]$, and let $w \in (0, 1-x_0]$. Consider the blowup 
\begin{align*}
    f_{Q_{w}(x_0, 0)}(x,y,z) = \frac{1}{w} f(wx,wy,wz\sep x_0, 0). 
\end{align*}
By \eqref{eq:f_AD_reg}, 
\begin{align*}
    f_{Q_w(x_0, 0)}(x, 0) = t_* x.
\end{align*}
It is important for this estimate that $f(x, 0) = t_* x$, rather than $f(x, 0) \geq t_* x$. 
By writing $f$ as a limit of the branching functions $f_{w_j}$, we bound the direction numbers of $f_{Q_w(x_0, 0)}$ from below using the direction numbers of $\tilde f$,
\begin{align*}
    d_{f_{Q_w(x_0, 0)}}(y; 1, 0)  &= \frac{1}{w}d_f(wy; x_0+w,0) \\ 
    &= \lim_{j\to \infty} \frac{1}{w} d_{f_{w_j}}(wy; x_0+w, 0) \\ 
    &= \lim_{j\to \infty} \frac{1}{ww_j} d_{\tilde f}(w_j wy; x_*+w_j(x_0+w), 0). 
\end{align*}
By monotonicity of the direction function in the $x$-variable (see \Cref{lem:dir_set_props}\ref{lemitem:dir_lip}),
\begin{align*}
    d_{\tilde f}(w w_j y; x_*+w_j(x_0+w), 0) \geq d_{\tilde f}(ww_j y\sep 1, 0) \geq s w w_j y,
\end{align*}
so $d_{f_{Q_w(x_0, 0)}}(y; 1, 0) \geq sy$. Because $t'+s>2$ and $t' > 1$, the Furstenberg set estimate (\cref{subsec:limiting_branching_func} \cref{exampleitem:info_L_furstenberg}) implies 
\begin{align*}
    f_{Q_w(x_0, 0)}(0, 1) \geq 1+s.
\end{align*}
For any $(x_0, y_0) \in \ms D$, take $w = y_0$ and apply the above estimate to obtain 
\begin{align*}
    f(x_0,y_0) &= f(x_0,0) + y_0 f_{Q_{y_0}(x_0,0)}(0,1) \\ 
     &\geq t_* x_0 + (s+1)y_0. 
\end{align*}
Thus $f \in \mc L_{t_*,s+1}$, and by \Cref{thm:incidence_lower_bd}, $f \in \mc L^{good}$, yielding a contradiction. 
\end{proof}

\begin{remark}
In the above proof we land in the special case \cref{sec:horizontal_good_line} of \Cref{thm:lip_func_theorem} because $f(x, y\sep x_0, 0) \geq t_* x + (s+1)y$.
\end{remark}

\section{Proof of Lipschitz function theorem}\label{sec:pf_lip_func_thm}

\subsection{Reductions of the Lipschitz function theorem}\label{sec:reduce_lip_func_good_line}
Using uniform subsets, the basic initial estimate, and the high-low method, we reduced the incidence lower bound problem \Cref{thm:incidence_lower_bd} to the Lipschitz function problem from \Cref{thm:lip_func_theorem}.
Before proving \Cref{thm:lip_func_theorem}, we make a few more reductions to simplify the problem. 

The next Lemma relates effective triples for a branching function $f$ to effective triples for a blowup of $f$. See \cref{subsec:limiting_branching_func} \cref{item:limiting_branching_func_blowup} for the definition of the blowup and \cref{subsec:limiting_lower_bounds} for the definition of $\mc L^{good}$, $\mc L^{bad}$. 

\begin{lemma}[Zooming in]\label{lem:zooming_in}
Let $f \in \mc L$, $x_0,y_0 \in [0, 1/5]$, and $w \in (0, 1-x_0-y_0]$ satisfy $e(s\sep x_0, y_0) \geq r w$ for $s \in [w/2, 1-x_0-y_0]$ and some $r>0$. 
Let 
\begin{align*}
    f_{Q_w(x_0, y_0)}(x,y,z) = \frac{1}{w} f(wx,wy,wz\sep x_0, y_0)
\end{align*}
be the blowup of $f$ near $(x_0, y_0)$ with scale $w$. 
Suppose $f_{Q_w(x_0,y_0)} \in \mc L^{good(c_2)}$ for some $c_2 < r/10$. Then $f \in \mc L^{good(\max\{x_0,y_0\}+c_2w)}$.
\end{lemma}
\begin{proof}
Suppose $(t\sep x, y)$ is $(c_1,c_2)$-effective for $f_{Q_w(x_0,y_0)}$ (see (\ref{eq:effectiveness_ineq}) for the definition of effective triples). Then 
\begin{align*}
    b(wt\sep x_0+wx,y_0+wy) + e(s\sep x_0+wx,y_0+wy) \geq wc_1
\end{align*}
for $s \in [wt, w(1-x-y)]$. By (\ref{eq:b_from_dir}) and \Cref{lem:dir_set_props}\ref{lemitem:dir_lip} we have the trivial lower bound $b(wt\sep x_0+wx,y_0+wy) \ge -wt$.
For $s \in [w(1-x-y), 1-(x_0+wx)-(y_0+wy)]$ by the assumption $e(s\sep x_0, y_0) \ge rw$ and \eqref{eq:f_lip} we have
\begin{align*}
    b(wt\sep x_0+wx,y_0+wy) + e(s\sep x_0+wx,y_0+wy) &\geq -wt + r w - 4wc_2 \geq (r - 5c_2)w. 
\end{align*}
Thus if $c_2 < r/10$ the triple $(wt\sep x_0+wx,y_0+wy)$ is $(c_1',c_2')$-effective for $f$ where $c_1' = \min\{wc_1, wr/2\}$ and $c_2' = \max\{x_0,y_0\} + c_2w$. 
\end{proof}

\begin{definition}
Let $f(x,y)$ be a Lipschitz function of two variables. We say $(x,y)$ is $(\alpha, \beta)$-good for $f$ if 
\begin{align*}
    f(x+x',y+y') - f(x,y) \geq \alpha x' + \beta y' 
\end{align*}
for all $x',y' \geq 0$ such that the left hand side is defined. 
\end{definition}

\begin{definition}
We say $f \in \mc L$ has an $(\alpha,\beta)$-good line if there is some $q \in [-1,1]$ and a line segment $\tau = \{(t+qt, t-qt)\, :\, 0 \leq t \leq 1/10\}$ so that every point in $\tau$ is $(\alpha, \beta)$-good for $f$.
\end{definition}

\begin{definition}
Let $f \in \mc L$ and $(x_0, y_0) \in \ms D$. Consider the sequence of blowups $f_{Q_w(x_0, y_0)}$ as $w \to 0$. If $\tilde f$ is a subsequential limit of these blowups, we say $\tilde f$ is a \textit{microscopic blowup} of $f$ at $(x_0,y_0)$.
\end{definition}
\mbox{}\bigskip\\
The plan is to use microscopic blowups to prove that if $\mc L_{\alpha, \beta}\cap \mc L^{bad}$ is nonempty, then for any $\alpha' < \alpha$ and $\beta' < \beta$ satisfying $\alpha'+\beta' > 3$, there is some $f \in \mc L^{bad}$ with an $(\alpha', \beta')$-good line. 
The next lemma locates lots of good pairs. We state the lemma in terms of a function of two variables to emphasize that it does not have to do with direction sets.

\begin{lemma}\label{lem:finding_good_points}
Let $\alpha,\beta > 0$.
Let $f(x,y)$ be a Lipschitz function with constant $C_{Lip}$ on a closed planar domain that includes $[0,1/5]\times [0,1/5]$. Suppose that $f$ satisfies the lower bound 
\begin{equation}\label{eq:f_lower_bd_good_pts_lem}
    f(x,y) \geq \alpha x + \beta y. 
\end{equation}
Then there exists a $1$-Lipschitz function $p: [0,1/10] \to [-1,1]$ with $p(0) = 0$ and a closed set $E \subset [0,1/10]$ such that 
\begin{align*}
    \text{$(t+p(t), t-p(t))$ is $0$-good for every $t \in E$.}
\end{align*}
Moreover, 
\begin{equation}\label{eq:setE_msr_lower_bd}
    |E\cap [0,s]| \gtrsim \frac{\min\{\alpha, \beta\}}{C_{Lip}}s\qquad \text{for all $s \leq 1/10$.}
\end{equation}
Here $0$-good means $f(x+x',y+y') \geq f(x,y)$ for all $x',y'\geq 0$ such that the inputs are defined.
\end{lemma}
When we say $f$ is $C_{Lip}$-Lipschitz, we mean 
\begin{align*}
    |f(x,y) - f(x', y')| \leq C_{Lip} (|x-x'| + |y-y'|). 
\end{align*}
\begin{proof}
Say $(x', y') \preceq (x,y)$ if $x' \leq x$ and $y' \leq y$. Say $(x',y') \prec (x,y)$ if $(x',y') \preceq(x,y)$ and $(x',y') \neq (x,y)$. 

Let $v = \min\{\alpha,\beta\}/10$. 
For every dyadic integer $m = 2^k$ and $\delta = v/m$, we will construct a discrete approximation $p_{\delta}$ to $p$ and a discrete approximation $E_{\delta}$ to $E$. Then we will take a limit as $\delta \to 0$ to obtain the result.

\medskip 
\noindent\textbf{Constructing a discrete approximation.}
Let $m = 2^k$ be a dyadic integer and let $\delta = v/m$. Construct a sequence of points $(x_j, y_j)_{j=0}^{m-1}$ as follows. Set $(x_0, y_0) = (0,0)$ and for $j \in \{1,\ldots, m-1\}$, set $(x_j,y _j)$ to be a maximal point under the $\preceq$ partial order such that 
\begin{itemize}
    \item $(x_j, y_j) \succeq (x_{j-1}, y_{j-1})$
    \item $(x_j,y_j)\in \mathrm{Domain}(f)$
    \item $f(x_j, y_j) \leq j \delta$. 
\end{itemize}
Because $(x_j, y_j)$ is chosen to be a maximal point in a sublevel set, $(x_j, y_j)$ is $0$-good. 
By the lower bound \eqref{eq:f_lower_bd_good_pts_lem}, $j\delta \geq f(x_j, y_j)\geq \max\{x_j,y_j\}\cdot \min\{\alpha, \beta\}$, so for all $0 \leq j \leq m-1$
\begin{align*}
    \max\{x_j, y_j\} \leq \frac{j\delta}{\min\{\alpha, \beta\}} \leq \frac{1}{10}.
\end{align*}
As a consequence, $(x_j, y_j) \prec (1/5, 1/5)$. The lower bound on $f$ gives 
\begin{align*}
    f(1/5, 1/5) \geq (\alpha+\beta)/5\geq 2m\delta > j\delta.
\end{align*}
If $f(x_j, y_j) < j\delta$ then by continuity of $f$, there would be some point $(x,y)$ on the straight line connecting $(x_j, y_j)$ to $(1/5,1/5)$ where $f(x,y) = j\delta$, contradicting maximality of the point $(x_j, y_j)$ in the partial order `$\preceq$'. Thus
\begin{align}\label{eq:fxj_eq}
    f(x_j, y_j) = j\delta.
\end{align}
Let 
\begin{align*}
    h_j^{\delta} &= (x_j+y_j)/2, \\ 
    E^{\delta} &= \{h_0^{\delta}, \ldots, h_{m-1}^{\delta}\} \subset [0,1/10]. 
\end{align*}
We prove separation between the points $h_j^{\delta}$ using the fact that $f$ is Lipschitz. Because $f$ is Lipschitz and $(x_j,y_j)_{j=0}^{m-1}$ is monotonic under $\preceq$,
\begin{align*}
    \delta = |f(x_{j+1}, y_{j+1}) - f(x_j, y_j)| \leq C_{Lip} ((x_{j+1} - x_j) + (y_{j+1} - y_j)) = 2C_{Lip}(h^{\delta}_{j+1} - h^{\delta}_j).
\end{align*}
This separation implies that the $\delta$-covering number of $E^{\delta}$ is lower bounded by 
\begin{align}\label{eq:Edelta_covering_num}
    |E^{\delta}|_{\delta} \gtrsim \frac{m}{C_{Lip}} \gtrsim \frac{\min\{\alpha,\beta\}}{C_{Lip}} \delta^{-1}. 
\end{align}
Now let $p^{\delta}: E^{\delta} \to [-1,1]$ be defined by 
\begin{align*}
    p^{\delta}(h_j^{\delta}) = (x_j-y_j)/2. 
\end{align*}
We can write $(x_j, y_j)$ as 
\begin{align}\label{eq:write_x_as_p}
    (x_j, y_j) = (h_j^{\delta}+p^{\delta}(h_j^{\delta}), h_j^{\delta}-p^{\delta}(h_j^{\delta})).
\end{align}
We can verify that $p^{\delta}$ is Lipschitz with Lipschitz constant $1$, 
\begin{align*}
    |p^{\delta}(h_{j+k}^{\delta}) - p^{\delta}(h_j^{\delta})| &= \frac{1}{2} |(x_{j+k} - x_j) - (y_{j+k} - y_j)| \\ 
    &\leq \frac{1}{2} (x_{j+k} - x_j) + \frac{1}{2} (y_{j+k} - y_j)\\ 
    &= h_{j+k}^{\delta} - h_j^{\delta}. 
\end{align*}
Extend $p^{\delta}$ to a piecewise linear function $p^{\delta,\circ}: [0,1/10]\to [-1,1]$ (extend by a constant for inputs $t \geq h_{m-1}$). For this choice of extension, $p^{\delta,\circ}$ is also Lipschitz with Lipschitz constant $1$.

\medskip 
\noindent \textbf{Taking a limit as $\delta \to 0$.}
Let $\delta_{k}$ be a subsequence of values of $\delta$ along which $p^{\delta_{k},\circ}$ converges uniformly to a $1$-Lipschitz function $p: [0,1/10]\to [-1,1]$. Let $E \subset [0,1/10]$ be the set of limit points of $E^{\delta_{k}}$ as $k \to \infty$,
\begin{align*}
    E = \bigcap_{\ell=1}^{\infty} \overline{\bigcup_{k = \ell}^{\infty} E^{\delta_{k}}}.
\end{align*}
We claim that $p$, $E$ satisfy the conclusion of the Lemma. 

First we show $(t+p(t), t-p(t))$ is $0$-good for every $t \in E$. Let $t \in E$, and write 
\begin{align*}
    t = \lim_{\ell\to \infty} t^{\delta_{k_{\ell}}},\qquad \text{$\delta_{k_{\ell}}$ is a subsequence and $t^{\delta_{k_{\ell}}} \in E^{\delta_{k_{\ell}}}$.}
\end{align*}
For each $\ell$, the point
\begin{align*}
    (t^{\delta_{k_{\ell}}} + p^{\delta_{k_{\ell}}}(t^{\delta_{k_{\ell}}}), t^{\delta_{k_{\ell}}} - p^{\delta_{k_{\ell}}}(t^{\delta_{k_{\ell}}}))
\end{align*}
is $0$-good by \eqref{eq:write_x_as_p}. Taking a limit as $\ell\to \infty$,
\begin{align*}
    (t+p(t), t-p(t)) = \lim_{\ell\to \infty} (t^{\delta_{k_{\ell}}} + p^{\delta_{k_{\ell}}}(t^{\delta_{k_{\ell}}}), t^{\delta_{k_{\ell}}} - p^{\delta_{k_{\ell}}}(t^{\delta_{k_{\ell}}})). 
\end{align*}
The set of $0$-good points is closed, so $(t+p(t), t-p(t))$ is $0$-good as well. 

Next we show a lower bound on the Lebesgue measure of $E$. In order to do so, we lower bound the Lebesgue measure of the $w$-neighborhoods $N_w(E)$ for all $w > 0$. By the lower bound on the $\delta$-covering number of $E^{\delta}$ (Eq. \eqref{eq:Edelta_covering_num}), if $\delta < w$ then 
\begin{align*}
    |E^{\delta}|_w \gtrsim \frac{\delta}{w}|E^{\delta}|_\delta  \gtrsim w^{-1} \frac{\min\{\alpha,\beta\}}{C_{Lip}}.
\end{align*}
The Lebesgue measure of the $w$-neighborhood is at least
\begin{align*}
    |N_w(E^{\delta})| \gtrsim w|E^{\delta}|_w \gtrsim \frac{\min\{\alpha,\beta\}}{C_{Lip}}. 
\end{align*}
For any $w> 0$ and any $\ell > 0$, 
\begin{align*}
    \Bigl|N_w\Bigl(\bigcup_{k = \ell}^{\infty} E^{\delta_{k}}\Bigr)\Bigr| \gtrsim \frac{\min\{\alpha,\beta\}}{C_{Lip}},
\end{align*}
so taking an intersection over all $w > 0$ gives a lower bound on the measure of the closure,
\begin{align*}
    \Bigl|\overline{\bigcup_{k = \ell}^{\infty} E^{\delta_{k}}} \Bigr| \gtrsim \frac{\min\{\alpha,\beta\}}{C_{Lip}}. 
\end{align*}
Taking an intersection over all $\ell \geq 1$  gives a lower bound on the measure of $E$, 
\begin{equation*}
    |E| \gtrsim \frac{\min\{\alpha,\beta\}}{C_{Lip}}.
\end{equation*}
\end{proof}

\begin{lemma}\label{lem:finding_microscopic_blowup}
Let $f \in \mc L_{\alpha,\beta}$ and choose $\alpha' < \alpha $, $\beta' < \beta$, $c_2 \in (0, 1/5)$. Then there is some $(x_0, y_0) \in [0,c_2]^2$ that is an $(\alpha',\beta')$-good point for $f$ and such that any microscopic blowup $\tilde f$ of $f$ at $(x_0, y_0)$ has an $(\alpha',\beta')$-good line.
\end{lemma}
\begin{proof}
Start by applying \Cref{lem:finding_good_points} to the function $f'(x, y) = f(x, y) - \alpha' x-\beta' y$ with parameters $\alpha-\alpha', \beta-\beta'>0$. We pick $(x_0, y_0)$ by analyzing the $1$-Lipschitz function $p$. By Rademacher's theorem in one dimension \cite[Theorem 7.20]{rudin}, $p$ is differentiable at almost every $t \in [0,1/10]$. Also, by Lebesgue's differentiation theorem almost every $t \in E$ is a Lebesgue point for $E$, meaning 
\begin{align*}
    \lim_{h\to 0} \frac{|E\cap [t, t+h]|}{h} = 1. 
\end{align*}
Pick some $t_0\in E$ which is a Lebesgue point and for which $q = p'(t_0)$ exists. Let $(x_0, y_0) = (t_0+p(t_0), t_0-p(t_0))$. The set of points 
\begin{align*}
    \{(t+(p(t_0+wt)-p(t_0))/w, t-(p(t_0+wt)-p(t_0))/w)\, :\, t_0+wt \in E\}
\end{align*}
are all $(\alpha', \beta')$-good for $f_{Q_w(x_0,y_0)}$. Let $\tilde f$ be a microscopic limit at $(x_0, y_0)$ along a sequence of scales $w_j \to 0$. We claim that $\{(t+qt,t-qt)\, :\, 0\leq t \leq 1/10\}$ consists of $(\alpha', \beta')$ good points. Because $t_0$ is a Lebesgue point for $E$, for any $t \in [0,1/10]$, there is a sequence $t_j \to t$ such that $t_0+w_jt_j \in E$ for all $j\geq 1$. The points $(t+(p(t_0+w_jt)-p(t_0))/w_j, t-(p(t_0+w_jt)-p(t_0))/w_j)$ converge to $(t+qt, t-qt)$, so that is a good point for $\tilde f$.
\end{proof}

\begin{lemma}\label{lem:reduce_to_good_line_case}
Let $\alpha + \beta > 3$. If $\mc L_{\alpha, \beta}\cap \mc L^{bad}$ is nonempty then for every choice of $\alpha' < \alpha$ and $\beta' < \beta$ such that $\alpha'+\beta' > 3$, there is some $f \in \mc L^{bad}$ with an $(\alpha',\beta')$-good line.
\end{lemma}
\begin{proof}
Suppose $f \in \mc L_{\alpha, \beta} \cap \mc L^{bad(c_2)}$ for some $c_2 > 0$. Let $\alpha' < \alpha$ and $\beta' < \beta$ satisfy $\alpha'+\beta'>3$, and let $c_0 = 3-\alpha'-\beta'$. By \Cref{lem:finding_microscopic_blowup}, there is some $(x_0, y_0) \in [0,c_2/20]^2$ which is $(\alpha', \beta')$-good for $f$ and such that every microscopic blowup of $f$ at $(x_0, y_0)$ has an $(\alpha', \beta')$-good line. 

We will show that any microscopic blowup at $(x_0, y_0)$ lies in $\mc L^{bad}$. For each $w \in (0, 1-x_0-y_0]$ let 
\begin{align*}
    f_{Q_w(x_0, y_0)}(x,y,z) = \frac{1}{w}f(wx, wy,wz; x_0, y_0) 
\end{align*}
be the blowup at $(x_0, y_0)$ with scale $w$. We would like to apply \Cref{lem:zooming_in} to prove $f_{Q_w(x_0, y_0)} \in \mc L^{bad(c_0/40)}$. In order to apply \Cref{lem:zooming_in}, we estimate the error function of $f$ near $(x_0, y_0)$ using the hypothesis that $(x_0, y_0)$ is $(\alpha', \beta')$-good,
\begin{align*}
    e(s; x_0, y_0) = \frac{1}{2}(f(s, 0; x_0, y_0) + f(0, s; x_0, y_0) - 3s) \geq \frac{1}{2}c_0 s\qquad \text{for all $s \in [0,1-x_0-y_0]$.}
\end{align*}
Suppose by way of contradiction that $f_{Q_w(x_0, y_0)} \in \mc L^{good(c_0/40)}$. By \Cref{lem:zooming_in}, $f \in \mc L^{good(\max\{x_0,y_0\}+wc_0/40)}$, so $f \in \mc L^{good(c_2/20+wc_0/20)}$. For small enough $w$, this contradicts the hypothesis that $f \in \mc L^{bad(c_2)}$, so we learn $f_{Q_w(x_0, y_0)} \in \mc L^{bad(c_0/40)}$ for small enough $w$. 

Let $w_j \to 0$ be a subsequence along which $f_{Q_{w_j}(x_0, y_0)}$ converges uniformly to some branching function $\tilde f \in \mc L$. Because $\mc L^{bad(c_0/40)}$ is a closed set (see the discussion at the beginning of \cref{subsec:limiting_lower_bounds}), $\tilde f \in \mc L^{bad(c_0/40)}$ as well. Because we chose $(x_0, y_0)$ using \Cref{lem:finding_microscopic_blowup}, $\tilde f$ has an $(\alpha', \beta')$-good line. 
\end{proof}

Thanks to \Cref{lem:reduce_to_good_line_case}, in order to prove \Cref{thm:lip_func_theorem} it remains to show that every $f \in \mc L_{\alpha,\beta}$ with an $(\alpha, \beta)$-good line 
\begin{align*}
    \{(t+qt,t-qt)\, :\, 0 \leq t \leq 1/10\},\qquad q \in [-1,1]
\end{align*}
lies in $\mc L^{good}$. Recall this means proving that for every $c_2 \in (0,1/3)$, there is some $c_1 > 0$ so that $f$ has a $(c_1,c_2)$-effective triple. 
We split into cases depending on $q = 1$, $q = -1$, $q\in [0,1)$, $q \in (-1, 0]$. See \cref{subsec:pf_sketch_analysis_lip} for a proof sketch. We define $c_0 := \alpha+\beta-3$.

\subsection{Horizontal good line, i.e. $q = 1$}\label{sec:horizontal_good_line}
Suppose that $(x, 0)$ is $(\alpha,\beta)$-good for $f$ for all $0 \leq x \leq 1/10$. Fix a small $t > 0$ to be chosen later. Our candidate triples will be of the form $(t\sep x, 0)$ for $x \in [0,1/10]$. 

Say $x \in [0,1/20]$ is \textit{$\eta$-directionally stable} if 
\begin{align*}
    d(t\sep x, 0) - d(t\sep x+t, 0) \leq \eta t. 
\end{align*}
Notice that $d(t\sep x, 0) - d(t\sep x+t,0) \geq 0$ for all $x$. Let $E_{ds} \subset [0,1/10]$ be the directionally stable set of $x$-values. We have 
\begin{align*}
    \int_0^{c_2} t^{-1}(d(t\sep x, 0) - d(t\sep x+t,0))\, dx \geq \eta |E_{ds}^c\cap [0,c_2]|. 
\end{align*}
On the other hand, as long as $t < c_2$,
\begin{align*}
     \int_0^{c_2} t^{-1}(d(t\sep x, 0) - d(t\sep x+t,0))\, dx = t^{-1} \int_0^t d(t\sep x, 0) dx- t^{-1}\int_{c_2}^{c_2+t} d(t\sep x, 0)dx \leq t
\end{align*}
using $d(t\sep x, 0) \leq t$. 
Thus
\begin{align*}
    |E_{ds}^c\cap [0,c_2]| \leq t/\eta. 
\end{align*}
As long as $t < \eta c_2$ there is some directionally stable $x \in [0,c_2]$. Because $x$ is $\eta$-directionally stable, (\ref{eq:b_from_dir}) implies $b(t\sep x, 0) \geq -\eta t$. Because $(x, 0)$ is $(\alpha, \beta)$-good, $e(s\sep x, 0) \geq \frac{1}{2}c_0 s$. 
It follows that 
\begin{align*}
    b(t\sep x, 0) + e(s\sep x, 0) \geq (c_0/2 - \eta)t \qquad \text{for all $s \geq t$ where the inputs are defined}.
\end{align*}
If we choose $\eta = c_0/4$ and $t = c_0c_2/8$ then there is some $x \in [0,c_2]$ such that $(t\sep x, 0)$ is $(c_0^2c_2/32, c_2)$-effective. 

\subsection{Vertical good line, i.e. $q=-1$}
This case is identical to the horizontal line case, except we look for triples of the form $(t\sep 0, y)$ instead of $(t\sep x, 0)$ and we use dual direction stability
\begin{equation*}
    d^{\vee}(t\sep 0, y) - d^{\vee}(t\sep 0, y+t) \leq \eta t
\end{equation*}
instead of direction stability. 

\subsection{Diagonal good line with slope $\leq 1$, i.e. $q \in [0, 1)$}
Let $m \in (0,1]$ and suppose that $(x, mx)$ is $(\alpha, \beta)$-good for $f$ for all $0 \leq x \leq 1/10$. Our candidate effective triples are of the form $(t\sep x, mx+mt)$. We parameterize these candidates by the set 
\begin{equation*}
    \ms B = \{(x,t)\, :\, x, t \in [0,c_2']\},\qquad c_2' = c_2/20. 
\end{equation*}
Let 
\begin{itemize}
    \item The directionally stable set $E_{ds}^{\rho}\subset \ms B$ be the pairs $(x,t)$ such that 
    \begin{align*}
        d(t\sep x,mx+mt) - d(t\sep x+t,mx+mt) \leq \rho t. 
    \end{align*}
    \item The high-low regular set $E_{hlr}^{\rho}\subset \ms B$ be the pairs $(x,t)$ such that 
    \begin{align*}
        e(s\sep x,mx+mt) \geq 2\rho t \qquad \text{for all $t \leq s \leq 1-(x+mx+mt)$}.
    \end{align*}
\end{itemize}
If $(x,t) \in E_{ds}^{\rho}\cap E_{hlr}^{\rho}$ then $(t\sep x,mx+mt)$ is $(\rho t, c_2)$-effective. 
We equip $\ms B$ with the measure 
\begin{align*}
    \mu = \frac{dxdt}{t},
\end{align*}
and let 
\begin{align*}
 \mu_{\varepsilon} = \frac{1_{t\geq \varepsilon}}{\mu\{(x,t)\, :\, t\geq \varepsilon\}}\mu
\end{align*}
be the normalized probability measure obtained by cutting off $\mu$ to $t\geq \varepsilon$. The normalization constant is equal to 
\begin{align*}
    \mu\{(x,t)\, :\, t\geq \varepsilon\} = c_2' \cdot \log(c_2'/\varepsilon). 
\end{align*}

\begin{lemma}\label{lem:HLR_density}
Choose $\rho = c_0/8$. There are constants $c(c_0,c_2), c'(c_0, c_2) > 0$ such that for all $\varepsilon < c_2'/10$, 
\begin{align*}
    \mu_{\varepsilon}(E_{hlr}^{c_0/8}) \geq c(c_0,c_2) - \frac{c'(c_0,c_2)}{\log \varepsilon^{-1}}.
\end{align*}
\end{lemma}

\begin{lemma}\label{lem:DS_density}
There exists $c(c_{2},m)>0$ such that for all $\rho > 0$ and $\varepsilon < c_2'/10$,
\begin{align*}
    \mu_{\varepsilon}(E_{ds}^{\rho}) \geq 1 - \rho^{-1} \frac{c(c_2, m)}{\log \varepsilon^{-1}}. 
\end{align*}
\end{lemma}
By combining these lemmas we can find an effective triple. 
Choose $\rho = c_0/8$. \Cref{lem:HLR_density} and \Cref{lem:DS_density} imply that if $\varepsilon < \varepsilon_0(c_0, c_2, m)$, then $\mu_{\varepsilon}(E_{hlr}^{c_0/8}) + \mu_{\varepsilon}(E_{ds}^{c_0/8}) > 1$. It follows that there is some $(x,t) \in E_{hlr}^{c_0/8} \cap E_{ds}^{c_0/8}$. The triple $(t\sep x, mx+mt)$ is $(c_0t/8, c_2)$-effective, so $f \in \mc L^{good(c_2)}$. It remains to prove \Cref{lem:HLR_density} and \Cref{lem:DS_density}. 

Towards proving \Cref{lem:HLR_density}, we define the slope function on $\ms B$ 
\begin{equation*}
    s_x(t) = t^{-1}(f(x,mx+t) - f(x,mx)). 
\end{equation*}
Using the hypothesis that $(x, mx)$ is $(\alpha, \beta)$-good along with the fact that $t\to f(x,t)$ is $2$-Lipschitz,
\begin{equation*}
    \beta \leq s_x(t) \leq 2 \qquad \text{for all $(x,t) \in \ms B$.}
\end{equation*}  
\begin{definition}
Let $\gamma \leq 1$ and $H \geq 1$ be parameters. 
A pair $(x,t) \in \ms B$ is $(\gamma, H)$-\textit{slope minimal} if $t \leq c_2'/(H+1)$ and 
\begin{equation*}
    s_x(mt+t') \geq s_x(mt) - \gamma \qquad \text{for all $t \leq t' \leq Ht$.}
\end{equation*}
we require $t \leq c_2'/(H+1)$ so all inputs lie in the domain $\ms B$. 
\end{definition}
We will take $\gamma = c_0/4$ and $H = 10/c_0$. 
\begin{claim}\label{claim:sm_implies_hlr}
Let $\gamma \leq c_0/4$ and $H \geq 10/c_0$. If a pair $(x,t)$ is $(\gamma, H)$-slope minimal, then $(x,t) \in E_{hlr}^{c_0/8}$. 
\end{claim}
\begin{proof}
In order to prove $(x,t) \in E_{hlr}^{c_0/8}$, we have to show a lower bound on the high-low error function. Recall that the error function is defined by 
\begin{align}\label{eq:error_func_sm}
    e(t'\sep x, mx+mt) = \frac{1}{2}[f(0, t'; x, mx+mt) + f(t', 0; x, mx+mt)  - 3t']. 
\end{align}
We start by estimating $e$ for scales $t \leq t' \leq Ht$.  We estimate the first term in the high-low error function using slope minimality,
\begin{align}
    f(0, t'; x, mx+mt) &= (mt+t') s_x(mt+t') - (mt)s_x(mt) \nonumber \\ 
    &\geq t' s_x(mt+t') + (mt)(s_x(mt+t') - s_x(mt)) \nonumber \\ 
    &\geq t' s_x(mt) - t' \gamma - mt \gamma \nonumber \\ 
    &\geq t' s_x(mt) - 2t' \gamma. \label{eq:sm_hlr_first}
\end{align}
The third line uses slope-minimality, and the last line uses the fact that $m \leq 1$ and $t' \geq t$. We estimate the second term in the high-low error function using the fact that $(x, mx)$ is $(\alpha, \beta)$-good,
\begin{align}
    f(t', 0; x,mx+mt) &= [f(x+t', mx+mt) - f(x, mx)] - [(f(x, mx+mt) - f(x, mx)]\nonumber \\ 
    &\geq (\alpha t' + \beta mt) - (mt) s_x(mt). \label{eq:sm_hlr_second}
\end{align}
Combining \eqref{eq:sm_hlr_first} and \eqref{eq:sm_hlr_second}, we estimate the error function by 
\begin{align*}
    e(t'; x, mx+mt) \geq \frac{1}{2}[(t'-mt)s_x(mt) + t'\alpha + \beta mt - 3t'] - t'\gamma. 
\end{align*}
Because $m\leq 1$, $(t'-mt)\geq 0$, so $(t'-mt)s_x(mt)\geq (t'-mt)\beta$. This estimate is the crucial place where we use the hypothesis $m \leq 1$. For $t \leq t' \leq Ht$, 
\begin{align*}
    e(t'; x, mx+mt) &\geq \frac{1}{2}[(t'-mt)\beta + t'\alpha + \beta mt - 3t'] - t'\gamma  \nonumber \\ 
    &\geq \frac{1}{2}c_0t' - t' \gamma && \text{(because $\alpha+\beta=3+c_0$)}\nonumber \\ 
    &\geq t(c_0/2-\gamma) 
\end{align*}
If $\gamma \leq c_0/4$, then 
\begin{equation}\label{eq:small_tp_sm_hlr}
    e(t'; x, mx+mt) \geq t c_0/4 \qquad \text{for $t \leq t' \leq Ht$.}
\end{equation}
Now we estimate the error function for $t' \geq Ht$. In this region we just use the fact that $(x, mx)$ is an $(\alpha, \beta)$-good scale and $f(x,y)$ is $2$-Lipschitz. We have 
\begin{align*}
    f(0,t'; x, mx+mt) &= [f(x,mx+mt+t') - f(x,mx)]- [f(x, mx+mt) - f(x, mx)] \\ 
    &\geq \beta t' - 2mt
\end{align*}
and 
\begin{align*}
    f(t', 0; x, mx+mt) &= [f(x+t', mx+mt) - f(x, mx)] - [f(x, mx+mt) - f(x, mx)] \\  
    &\geq \alpha t' - 2mt. 
\end{align*}
Combining these two estimates, the error function is bounded by 
\begin{align*}
    e(t'; x, mx+mt) &\geq \frac{1}{2}[\alpha t' + \beta t' - 3t'] - 2mt \\ 
    &\geq \frac{1}{2}c_0 t' - 2mt. 
\end{align*}
If $t' \geq Ht$ then 
\begin{align*}
    e(t'; x, mx+mt) &\geq t( c_0 H/2 - 2m) \geq t(c_0 H/2 - 2). 
\end{align*}
If $H \geq 10/c_0$ then $c_0H/2 - 2 \geq c_0/4$ (using that $c_0 \leq 1$), so 
\begin{align}\label{eq:large_tp_sm_hlr}
    e(t'; x, mx+mt) \geq t c_0/4 && \text{for $t' \geq Ht$.}
\end{align}
Combining \eqref{eq:small_tp_sm_hlr} and \eqref{eq:large_tp_sm_hlr} proves that $(x,t) \in E_{hlr}^{c_0/8}$. 
\end{proof}

We establish some regularity of $t \to s_x(t)$. By the definition of the slope function, $ts_x(t) = f(0, t; x, mx)$. Because $f$ is Lipschitz and monotonic \eqref{eq:f_lip},
\begin{align*}
    f(0, t; x, mx) \leq f(0, t+h; x, mx) \leq f(0, t; x, mx) + 2h,
\end{align*}
so 
\begin{align*}
    ts_x(t) \leq (t+h) s_x(t+h) \leq ts_x(t) + 2h,
\end{align*}
and by setting $h = (r-1)t$ for some $r > 1$ we get 
\begin{align}\label{eq:slope_func_regularity}
    \frac{1}{r} s_x(t) \leq s_x(rt) \leq \frac{1}{r} s_x(t) + \frac{2(r-1)}{r}. 
\end{align}
We use regularity of the function $t\to s_x(t)$ to find whole intervals of slope minimal pairs. 

\begin{claim}\label{claim:interval_slope_minimal}
Let $\gamma \leq 1$ and $H\geq 1$. Suppose $(x,t)$ is $(\gamma, H)$-slope minimal and $t \leq c_2'/(2H+1)$. Then for all $\tilde t \in [t, (1+\gamma/100)t]$, the pair $(x, \tilde t)$ is $(2\gamma,H)$-slope minimal. 
\end{claim}
\begin{proof}
Let $\tilde t = r t$ where $r \in [1, 1+\gamma/100]$. First of all, $\tilde t \leq c_2'/(H+1)$. 
The hypothesis that $(x, t)$ is $(\gamma, H)$-slope minimal implies that
\begin{align*}
    s_x((m+u)t) - s_x(mt) \geq -\gamma \qquad \text{for all $u\in [1, H]$.}
\end{align*}
In order to prove $(x, \tilde t)$ is $(2\gamma, H)$-slope minimal we must show 
\begin{align*}
    s_x((m+u)rt ) - s_x(mrt) &\geq -2\gamma \qquad \text{for all $u\in [1, H]$.}
\end{align*}
By the lower bound in \eqref{eq:slope_func_regularity}, 
\begin{align*}
    s_x((m+u)rt) &\geq \frac{1}{r}s_x((m+u)t) \geq (1-\gamma/100) s_x((m+u)t) \geq s_x((m+u)t) - \gamma/50. 
\end{align*}
By the upper bound in \eqref{eq:slope_func_regularity},
\begin{align*}
    s_x(mrt) \leq s_x(mt) + \frac{2(r-1)}{r} \leq s_x(mt) + \gamma/50. 
\end{align*}
Combining these two estimates yields $(2\gamma, H)$-slope minimality,
\begin{align*}
     s_x((m+u)rt ) - s_x(mrt) &\geq s_x((m+u)t) - s_x(mt) - \gamma / 25 \geq -2\gamma. 
\end{align*}
\end{proof}

The following claim shows that for each fixed value of $x$, the values of $t$ such that $(x,t)$ is slope minimal are abundant on a logarithmic scale. 
\begin{claim}\label{claim:many_sm_pairs}
Let $\gamma \leq 1$ and $H \geq 1$, and set $R = (H+1)^{\lceil 2/\gamma \rceil}$. For any $x \in [0,c_2']$ and $t_0 \in (0,c_2'/(2HR)]$, there is some $t \in [t_0, Rt_0]$ such that $(x, t)$ is $(\gamma, H)$-slope minimal. 
\end{claim}
\begin{proof}
Let $k = \lceil2/\gamma\rceil$ and $R = (H+1)^k$.
Suppose by way of contradiction that for all $t \in [t_0, Rt_0]$, $(x, t)$ is not $(\gamma, H)$-slope minimal. 

Since $(x, t_0)$ is not $(\gamma, H)$-slope minimal, there is some $t_1 \in [mt_0+t_0,mt_0+Ht_0]$ such that $s_x(t_1) \leq s_x(mt_0) - \gamma$. In particular, $t_1 \in [t_0, (H+1)t_0]$. Since $(x, t_1)$ is not slope minimal, there is some $t_2 \in [t_0, (H+1)^2 t_0]$ such that  $s_x(t_2) \leq s_x(mt_0) - 2\gamma$. Continuing in this way, for every $0 \leq j\leq k$ there is some $t_j \in [t_0, (H+1)^j t_0]$ such that $s_x(t_j) \leq s_x(mt_0) - j\gamma$. However, $s_x(mt_0) \leq 2$ and $k\gamma > 2$, yielding a contradiction when $j = k$. 
\end{proof}

We combine the previous three claims to prove \Cref{lem:HLR_density}. 
\begin{proof}[Proof of \Cref{lem:HLR_density}]
For any $\gamma\leq 1$ and $H\geq 1$, let 
\begin{align*}
    E_{sm}^{(\gamma,H)} = \{(x,t) \in \ms B\, :\, (x,t) \text{ is $(\gamma, H)$-slope minimal}\}.
\end{align*}
We will show that for any such $(\gamma, H)$ and $\varepsilon < c_2'/10$,
\begin{align}\label{eq:sm_density_lower_bd}
    \mu_{\varepsilon}(E_{sm}^{(\gamma, H)}) \geq c(c_2, \gamma, H) - \frac{c'(c_2, \gamma, H)}{\log \varepsilon^{-1}}
\end{align}
for some constants $c(c_2, \gamma, H), c'(c_2, \gamma, H) > 0$.
By \Cref{claim:sm_implies_hlr}, $E_{sm}^{(c_0/4, 10/c_0)} \subset E_{hlr}^{c_0/8}$, so choosing $\gamma = c_0/4$ and $H = 10/c_0$ proves \Cref{lem:HLR_density}.

We work one $x$-coordinate at a time. Let $R = R(\gamma/2, H)$ from \Cref{claim:many_sm_pairs} (we assume without loss of generality $R \geq 10)$. For each $j\geq 0$, use \Cref{claim:many_sm_pairs} to find a point $t_j^{(x)} \in \Bigl[R^{-2j-1} \frac{c_2'}{3H}, R^{-2j} \frac{c_2'}{3H}\Bigr]$ such that $(x, t_j^{(x)})$ is $(\gamma/2, H)$-slope minimal.
By \Cref{claim:interval_slope_minimal}, for each $j\geq 0$
\begin{align}\label{eq:intervals_in_sm}
    \{x\}\times [t_j^{(x)}, (1+\gamma/100) t_j^{(x)}] \subset E_{sm}^{(\gamma, H)},
\end{align}
and the above sets are disjoint for all $j\geq 0$. We use Fubini's theorem to disintegrate the measure of $E_{sm}^{(\gamma, H)}$ as 
\begin{align}\label{eq:lower_bd_msr_disintegrated}
    \int 1_{E_{sm}^{(\gamma, H)}}(x,t)\, d\mu_{\varepsilon}(x,t) = \frac{1}{c_2'\cdot \log(c_2'/\varepsilon)}\int_0^{c_2'} \Bigl(\int_{\varepsilon}^{c_2'} 1_{E_{sm,x}^{(\gamma, H)}}(x,t)\, \frac{dt}{t}\Bigr)\, dx. 
\end{align}
For each $x \in [0,c_2']$ we estimate the inner integral using \eqref{eq:intervals_in_sm},
\begin{align}
    \int_{\varepsilon}^{c_2'} 1_{E_{sm}^{(\gamma, H)}}(x,t)\, \frac{dt}{t} &\geq \sum_{j\geq 0 \text{ s.t } t_j^{(x)} \geq \varepsilon}  \int_{t_j^{(x)}}^{(1+\gamma/100)t_j^{(x)}}\, \frac{dt}{t} \nonumber \\ 
    &\geq \log(1+\gamma/100) \#\{j\geq 0\, :\, t_j^{(x)} \geq \varepsilon\}. \label{eq:lower_bound_disintigrate}
\end{align}
By the lower bound $t_j^{(x)} \geq R^{-2j-1} c_2'/3H$, 
\begin{equation}
    \#\{j\geq 0\, :\, t_j^{(x)} \geq \varepsilon\} \geq \Bigl\lfloor \frac{\log c_2'/(3H\varepsilon)}{2\log R}-1\Bigr\rfloor \geq c(c_2, \gamma, H)(\log \varepsilon^{-1} - c'(c_2, \gamma, H)). \label{eq:lower_bd_num_j}
\end{equation}
Combining \eqref{eq:lower_bd_msr_disintegrated}, \eqref{eq:lower_bound_disintigrate}, and \eqref{eq:lower_bd_num_j} yields 
\begin{align*}
    \mu_{\varepsilon}(E_{sm}^{(\gamma, H)}) \geq c(c_2, \gamma, H) - \frac{c'(c_2, \gamma, H)}{\log \varepsilon^{-1}}
\end{align*}
proving \eqref{eq:sm_density_lower_bd}.
\end{proof}

Now we prove \Cref{lem:DS_density}, the lower bound on $\mu_{\varepsilon}(E_{ds}^{\rho})$, which completes the proof that there exists a $(c_1, c_2)$-effective triple for some $c_1 > 0$.

\begin{proof}[Proof of \Cref{lem:DS_density}]
We will evaluate direction numbers on the good line, so we use the short hand
\begin{align*}
    d(t\sep x) = d(t\sep x, mx). 
\end{align*}
Define the $x$-averaged direction number 
\begin{align*}
    \bar d(t) &= \frac{1}{c_2'}\int_0^{c_2'} d(t\sep x)\, dx. 
\end{align*}
The pair $(x,t)$ is $\rho$-directionally stable if 
\begin{align*}
    d(t\sep x, mx+mt) - d(t\sep x+t,mx+mt) \leq \rho t. 
\end{align*}
The direction set inequality \Cref{lem:dir_set_props}\ref{lemitem:dir_compare_scales} gives
\begin{align*}
    d((m+1)t\sep x) \geq d(mt\sep x) + d(t\sep x, mx+mt),
\end{align*}
and rearranging gives
\begin{align}\label{eq:Fnonnegative}
    0 \leq d(t\sep x,mx+mt) - d(t\sep x+t) \leq d((m+1)t\sep x) - d(mt\sep x) - d(t\sep x+t). 
\end{align}
The right hand side has the advantage that it only involves directions on the good line. Let 
\begin{align*}
    F(x,t) = \frac{1}{t}[d((m+1)t\sep x) - d(mt\sep x) - d(t\sep x+t)].
\end{align*}
By (\ref{eq:Fnonnegative}), $F(x, t)\ge 0$, and if $(x,t)\in (E_{ds}^{\rho})^c$ then $F(x,t) \geq \rho$. We evaluate $\int F(x,t) d\mu_{\varepsilon}$ in two steps. First, for each fixed $t \in [0, c_2']$,
\begin{align}
    \frac{1}{c_2'} \int_0^{c_2'} F(x, t)\, dx &= \frac{1}{t}[\bar d((m+1)t) - \bar d(mt) - \bar d(t)] + \frac{1}{tc_2'}\Bigl(\int_0^{t} d(t\sep x)dx - \int_{c_2'}^{c_2'+t} d(t\sep x)dx\Bigr)\nonumber \\ 
    &= \frac{1}{t}[\bar d((m+1)t) - \bar d(mt) - \bar d(t)] + O(t/c_2)\label{eq:approx_id_each_x}
\end{align}
using $d(t\sep x) \leq t$ to estimate the error term. 
Now we average over $t$ and rescale the integral to get for any fixed $\alpha > 0$ that 
\begin{align}
    \Bigl(\log \frac{c_2'}{\varepsilon}\Bigr)^{-1}\int_{\varepsilon}^{c_2'} \bar d(\alpha t)\, \frac{dt}{t^2} &= \alpha\, \Bigl(\log \frac{c_2'}{\varepsilon}\Bigr)^{-1}\int_{\alpha \varepsilon}^{\alpha c_2'} \bar d(t)\, \frac{dt}{t^2}\nonumber \\ 
    &= \alpha\, \Bigl(\log \frac{c_2'}{\varepsilon}\Bigr)^{-1}\int_{\varepsilon}^{c_2'} \bar d(t)\, \frac{dt}{t^2} + \mathrm{Err(\alpha)}. \label{eq:approx_id_alpha}
\end{align}
We will choose $\alpha \in \{m,m+1\}$ in the sequel.
If $\alpha \leq 1$ then
\begin{align*}
    \mathrm{Err}(\alpha) &= \alpha\, \Bigl(\log \frac{c_2'}{\varepsilon}\Bigr)^{-1} \int_{\alpha \varepsilon}^{\varepsilon} \bar d(t)\, \frac{dt}{t^2} - \alpha\, \Bigl(\log \frac{c_2'}{\varepsilon}\Bigr)^{-1} \int_{\alpha c_2'}^{c_2'} \frac{dt}{t^2}
\end{align*}
so
\begin{align*}
    |\mathrm{Err}(\alpha)| &\leq \alpha\, \Bigl(\log \frac{c_2'}{\varepsilon}\Bigr)^{-1} \int_{\alpha \varepsilon}^{\varepsilon}  \frac{dt}{t}  \\ 
    &\leq  \alpha\log \alpha^{-1}\, \Bigl(\log \frac{c_2'}{\varepsilon}\Bigr)^{-1} . 
\end{align*}
If $\alpha \geq 1$ a similar calculation gives the same result but with $\log \alpha^{-1}$ replaced with $\log \alpha$, so in either case 
\begin{align*}
    |\mathrm{Err}(\alpha)| \leq \alpha|\log \alpha|\, \Bigl(\log \frac{c_2'}{\varepsilon}\Bigr)^{-1}. 
\end{align*}
Combining the approximate identity \eqref{eq:approx_id_each_x} with the approximate identity \eqref{eq:approx_id_alpha} for $\alpha \in \{m,m+1\}$, we find
\begin{align*}
    \int F(x,t)\, d\mu_{\varepsilon} &= \Bigl(\log \frac{c_2'}{\varepsilon}\Bigr)^{-1} \int_0^{c_2'} \Bigl( \frac{1}{c_2'} \int_0^{c_2'} F(x,t)\, dx \Bigr)\, \frac{dt}{t} \\ 
    &= \Bigl(\log \frac{c_2'}{\varepsilon}\Bigr)^{-1} \int_0^{c_2'} \Bigl( \frac{1}{t}[\bar d((m+1)t) - \bar d(mt) - \bar d(t)] +O(t/c_2)\Bigr)\, \frac{dt}{t}  \\ 
    &\leq  \Err(m) + \Err(m+1) + \Bigl(\log \frac{c_2'}{\varepsilon}\Bigr)^{-1}O(1/c_2),
\end{align*}
the main terms cancel and we are left with only error terms. We estimate the error terms by
\begin{align*}
     \Err(m) + \Err(m+1) + \Bigl(\log \frac{c_2'}{\varepsilon}\Bigr)^{-1}O(1/c_2) &\leq C \Bigl(\log \frac{c_2'}{\varepsilon}\Bigr)^{-1} (\log m^{-1}+1/c_2) \\ 
    &\leq \frac{c(c_2, m)}{\log \varepsilon^{-1}}.
\end{align*}
On the other hand,
\begin{align*}
    \int F(x,t)\, d\mu_{\varepsilon} \geq \rho\, \mu_{\varepsilon}((E_{ds}^{\rho})^c), 
\end{align*}
so 
\begin{align*}
   \mu_{\varepsilon}(E_{ds}^{\rho}) \geq 1-\rho^{-1} \frac{c(c_2, m)}{\log \varepsilon^{-1}}. 
\end{align*}
\end{proof}

\subsection{Diagonal good line with slope $\geq 1$, i.e. $q \in (-1, 0]$}
This case is the same as the slope $\leq 1$ case except we switch the roles of $x$ and $y$ and the roles of $d$ and $d^{\vee}$.

\appendix 

\section{Miscellaneous proofs}\label{sec:misc_proofs}

\subsection{Proof of the high-low inequality}\label{sec:high_low_proof}

For the reader's convenience we prove the high-low estimate \Cref{thm:high_low_ineq}. The proof is the same as in \cite[\S 3.1]{CPZ}. Notice that the proof allows for $P$ and $L$ to be sets with multiplicity. 

Let $P = \{x_1, \ldots, x_n\}$ be a set of points in $[-1,1]^2$ and $L = \{\ell_1, \ldots, \ell_m\}$ a set of lines. 
Let
\begin{align*}
    f_w = \sum_{\ell \in L} \frac{1}{w} 1_{\T_{\ell}(w)},\quad g = \sum_{p \in P} \delta_p. 
\end{align*}
Then 
\begin{equation*}
    \langle f_w, g\rangle = w^{-1}\, \#\{(p, \ell) \in P \times L\, :\, d(p, \ell) \leq w/2\}. 
\end{equation*}
We are ready to give the smoothed definition of incidences.
Let $\chi: \R^2 \to \R_{\geq 0}$ be a fixed radially symmetric bump function with $\int \chi = 1$ and $\supp \chi \subset B_{1/50}$. Let $\chi_w(x) = w^{-2}\chi(w^{-1}x)$. We define the smoothed version of incidences by convolving $f_w$ with $\chi_w * \chi_{w/2}$, 
\begin{equation}\label{eq:defn_smoothed_incidences_inner_prod}
    I(w; P, L) := w\, \langle \chi_w * \chi_{w/2} * f_w, g\rangle. 
\end{equation}
Let 
\begin{equation*}
    \eta(t) = \int_{-1/2}^{1/2} \int_{-\infty}^{\infty}(\chi * \chi_{1/2})(t+r, s)\, ds\, dr. 
\end{equation*}
Notice that $\eta|_{[-1/2+1/10, 1/2-1/10]} = 1$ and $\supp \eta \subset [-1/2-1/10, 1/2+1/10]$. 
Because $\chi$ was chosen radially symmetric, we see that 
\begin{equation*}
    I(w; P, L) = \sum_{(p, \ell) \in P \times L} \eta(w^{-1}d(p, \ell))
\end{equation*}
as we defined in \eqref{eq:smoothed_incidence_count}. 
We have 
\begin{equation}
    B(w) = \frac{I(w)}{w\, |P|\, |L|} = \frac{1}{|P|\, |L|}\langle \chi_w * \chi_{w/2} * f_w, g\rangle,
\end{equation}
hence
\begin{align*}
    B(w) - B(w/2) &= \frac{1}{|P|\, |L|}\, \langle \chi_{w} * \chi_{w/2} *  f_w  - \chi_{w/2} * \chi_{w/4} * f_{w/2} , g\rangle \\ 
    &= \frac{1}{|P|\, |L|}\, \langle \chi_{w/2} * (\chi_{w} * f_w  - \chi_{w/4} * f_{w/2}) , g\rangle.
\end{align*}
We convolved $f_w$ with $\chi_w * \chi_{w/2}$ in (\ref{eq:defn_smoothed_incidences_inner_prod}) rather than just $\chi_w$ in order to factor out $\chi_{w/2}$ in this equation. Putting the convolution on the other side and using Cauchy-Schwarz,
\begin{equation}
    |B(w) - B(w/2)| \leq \frac{1}{|P|\, |L|}\, \| \chi_w * f_w - \chi_{w/4} * f_{w/2} \|_{L^2([-2,2]^2)}\, \| \chi_{w/2} * g \|_{L^2([-2,2]^2)}. 
\end{equation}
We estimate the term involving $g$ using the quantity $\ms M_{w\times w}(P)$, 
\begin{align*}
    \| \chi_{w/10} * g \|_2^2 \leq \| g \|_1 \| \chi_{w/10} * g \|_{\infty} \lesssim w^{-2} |P|\,\ms M_{w\times w}(P).
\end{align*}
We estimate the term involving $f$ using orthogonality and the quantity $\ms M_{1\times w}(L)$. 
Let
\begin{align*}
    \Phi_{\ell} &= \chi_w * w^{-1}1_{\T_{\ell}(w)} - \chi_{w/4} *  (w/2)^{-1}1_{\T_{\ell}(w/2)}\\ 
    \sum_{\ell \in L} \Phi_{\ell} &= \chi_w * f_w - \chi_{w/4} * f_{w/2}.
\end{align*}
We expand the $L^2$ norm as a sum of truncated inner products. Let $\psi: \R^2 \to [0,1]$ be a fixed smooth bump function with $\psi|_{[-2,2]^2} = 1$ and $\supp \psi \subset [-3,3]^2$. We have
\begin{equation*}
    \Bigl \| \sum_{\ell \in L} \Phi_{\ell} \Bigr \|_{L^2([-2,2]^2)}^2 \leq \sum_{\ell_1, \ell_2} \left|\int \Phi_{\ell_1}(x)\Phi_{\ell_2}(x)\, \psi(x) dx\right|.
\end{equation*}
For any $\ell_1, \ell_2$ we have the estimate 
\begin{align*}
    \left|\int \Phi_{\ell_1}(x)\Phi_{\ell_2}(x)\, \psi(x) dx\right| \leq \| \Phi_{\ell_1} \|_{L^2([-3,3]^2)}\| \Phi_{\ell_2} \|_{L^2([-3,3]^2)} \lesssim w^{-1}. 
\end{align*}
Now suppose $d(\ell_1, \ell_2)  \geq 1000w$ in the line metric \eqref{eq:line_metric_all_slope}. Either $\T_{2w}(\ell_1) \cap \T_{2w}(\ell_2) \cap [-3,3]^2 = \emptyset$, or $|\theta(\ell_1) - \theta(\ell_2)| \gtrsim d(\ell_1, \ell_2)$. In the first case 
\begin{align*}
    \int \Phi_{\ell_1}(x)\Phi_{\ell_2}(x)\, \psi(x) dx  = 0. 
\end{align*}
In the second case let $R = \T_{2w}(\ell_1) \cap \T_{2w}(\ell_2)$ and let $x_R = \ell_1 \cap \ell_2$ be the center of $R$. Let $\theta$ be the angle between $\ell_1$ and $\ell_2$. Notice that $R$ is a parallelogram which is contained in a rectangle with width $w$ and length $2w/\sin \theta$. Using the fact that $\Phi_{\ell_1}$ and $\Phi_{\ell_2}$ are orthogonal in $\R^2$ and re-centering to $R$, 
\begin{align*}
    \int \Phi_{\ell_1}(x)\Phi_{\ell_2}(x)\, \psi(x) dx = \int (\Phi_{\ell_1}\Phi_{\ell_2})(x_R+y)(\psi(x_R+y) - \psi(x_R))\, dy. 
\end{align*}
Because $\psi$ is a fixed smooth function we have the linear approximation
\begin{align*}
    \psi(x_R+y) - \psi(x_R) = \nabla \psi(x_R)\cdot y + O(|y|^2).
\end{align*}
The integral against the linear part vanishes due to the symmetry
\begin{align*}
    \Phi_{\ell_1}(x_R+y) = \Phi_{\ell_1}(x_R-y)\quad \text{and}\quad \Phi_{\ell_2}(x_R+y) = \Phi_{\ell_2}(x_R-y).
\end{align*}
The $O(|y|^2)$ term is bounded by $C \cdot \mathrm{Diameter}(R)^2$, so
\begin{equation*}
    \left|\int \Phi_{\ell_1}(x) \Phi_{\ell_2}(x)\, \psi(x) dx\right| \lesssim \Area(R)\, \mathrm{Diameter}(R)^2\, w^{-2} \lesssim (w^2\theta^{-1})\, (w\theta^{-1})^2\, w^{-2} = w^2\theta^{-3}.
\end{equation*}
Overall we find 
\begin{align*}
    \left|\int \Phi_{\ell_1}(x) \Phi_{\ell_2}(x)\, \psi(x) dx\right| \lesssim \min\{w^{-1}, w^2/d(\ell_1, \ell_2)^3\}. 
\end{align*}
We find 
\begin{align*}
    \sum_{\ell_2\in L} \Bigl|\int \Phi_{\ell_1}(x)\Phi_{\ell_2}(x)\psi(x)\, dx\Bigr| &\lesssim w^{-1} \#\{\ell_2\, :\, d(\ell_1, \ell_2) \leq 1000w\} + \sum_{w \leq 2^{-j} \leq 100}  w^22^{3j} \#\{\ell_2\, :\, d(\ell_1, \ell_2) \leq 2^{-j}\}  \\ 
    &\lesssim w^{-1} \ms M_{1\times w}(L) + w^2 \sum_{w \leq 2^{j} \leq 100}2^{-3j} \ms M_{1\times 2^{j}}(L) \\ 
    &\lesssim w^{-1} \ms M_{1\times w}(L) + w^2 \sum_{w \leq 2^{j} \leq 100}2^{-3j} \ms M_{1\times w}(L)(2^{j}/w)^2 \\ 
    &\lesssim w^{-1} \ms M_{1\times w}(L).
\end{align*}
Putting everything together, 
\begin{align*}
    |B(w) - B(w/2)| \lesssim \Bigl(\frac{\ms M_{w\times w}(P)}{|P|} \frac{\ms M_{1\times w}(L)}{|L|}w^{-3}\Bigr)^{1/2}.
\end{align*}

\subsection{Furstenberg in phase space}\label{sec:furstenberg_pf_phase_space}
We start by stating Ren and Wang's discretized Furstenberg set estimate \cite[Theorem 4.1]{RenWang2023} in their language.

\begin{definition}
\begin{itemize}
    \item For $\delta \in 2^{\Z}$, let $\mc D_{\delta}(\R^2) = \{[a\delta, (a+1)\delta]\times [b\delta, (b+1)\delta]\, :\, a,b \in \Z\}$ be the dyadic cubes in $\R^2$ of side length $\delta$. For $\mc P \subset \R^2$ denote 
    \begin{align*}
        \mc D_{\delta}(\mc P) = \{Q\in \mc D_{\delta}(\R^2)\, :\, Q \cap \mc P \neq \emptyset\}. 
    \end{align*}
    We denote $\mc D_{\delta} = \mc D_{\delta}([0,1)^2)$. 

    \item For a bounded set $P \subset \R^2$ and $\delta \in 2^{-\N}$, define the dyadic $\delta$-covering number 
    \begin{align*}
        |P|_{\delta} := |\mc D_{\delta}(P)|. 
    \end{align*}

    \item Let $\mc P \subset \mc D_{\delta}$ be a set of dyadic cubes and $P = \cup \mc P$. We say $\mc P$ is an $(\delta, s, C)$-set if 
    \begin{align*}
        |P\cap B(x, r)|_{\delta} \leq C r^s |P|_{\delta}.
    \end{align*}
\end{itemize}
\end{definition}

\begin{definition}
    \begin{itemize}
        \item Let $\delta \in 2^{-\N}$. A dyadic $\delta$-tube is a set of the form $T = \bigcup_{x\in p} \mb D(x)$ where $p \in \mc D_{\delta}([-1,1]\times \R)$ and $\mb D$ is the point-line duality map 
        \begin{align*}
            \mb D(a,b) = \{(x,y)\in \R^2\, :\, y = ax+b\} \subset \R^2.
        \end{align*}
        We abbreviate $\mb D(p) = \bigcup_{x\in p}\mb D(x)$. The collection of all dyadic $\delta$-tubes is denoted $\mc T^{\delta}$. 

        \item For $L$ a set of lines in $\R^2$, let 
        \begin{align*}
            \mc T^{\delta}(L) = \{T \in \mc T^{\delta}\, :\, T\cap L \neq \emptyset\}. 
        \end{align*}
        and define the covering number 
        \begin{align*}
            |L|_{\delta} = |\mc T^{\delta}(L)|.
        \end{align*}

        \item A finite collection of dyadic $\delta$-tubes $\{\mb D(p)\}_{p\in \mc P}$ is called a $(\delta, s, C)$-set if $\mc P$ is a $(\delta, s, C)$-set. 
    \end{itemize}
\end{definition}

\begin{definition}
Fix $\delta \in 2^{-\N}$, $s\in [0,1]$, $C > 0$, $M \in \N$. We say a pair $\mc P \subset \mc D_{\delta}$, $\mc T\subset \mc T^{\delta}$ is a $(\delta, s, C, M)$-nice configuration if for every $p \in \mc P$ there exists a $(\delta, s, C)$-set $\mc T(p) \subset \mc T$ with $|\mc T(p)| \sim M$ and such that $T\cap p\neq \emptyset$ for all $T \in \mc T(p)$. 
\end{definition}

The following is \cite[Theorem 4.1]{RenWang2023}.
\begin{theorem}\label{thm:RenWangFurstenberg}
For every $\varepsilon > 0$ there exists $\eta = \eta(\varepsilon, s, t) > 0$ such that the following holds for any $\delta < \delta_0(s, t, \varepsilon)$. Let $(\mc P, \mc T)$ be a $(\delta, s, \delta^{-\eta}, M)$-nice configuration with $s \in (0, 1]$, $\mc P$ is a $(\delta, t, \delta^{-\eta})$-set, and $t \in (0, 2]$. Then 
\begin{align*}
    |\mc T|_{\delta} \gtrsim_{\varepsilon} \delta^{-\min\{t,(s+t)/2,1\}+\varepsilon} M. 
\end{align*}
\end{theorem}
We are ready to state a phase-space version of the Furstenberg set estimate.
\begin{prop}\label{prop:furstenberg_quantitative}
For any $\varepsilon > 0$ there exists $\eta = \eta(\varepsilon, s, t) > 0$ such that the following holds for any $m > m_0(\varepsilon, s, t)$ and $T > T_0(m, \varepsilon, s, t)$. Let $\mb X \subset \Omega$ be $(m, T)$-uniform with branching function $f$, and suppose
\begin{align*}
    d(y\sep 1,0) &\geq s y - \eta \qquad \text{for $0 \leq y \leq 1$,} \\ 
    f(x, 0) &\geq t x - \eta \qquad \text{for $0 \leq x \leq 1$.}
\end{align*}
Then 
\begin{align*}
    f(0, 1) &\geq \min\{t,(s+t)/2,1\} +  d(1\sep 1,0) - \varepsilon.
\end{align*}
\end{prop}
\begin{proof}
Let $\delta_0 = 2^{-mT}$ and let $\delta = 2^{50}\cdot 2^{-mT}$. Let 
\begin{align*}
    \mc P &= \mc P_{\delta}(P[\mb X]), \\ 
    \mc T &= \mc T^{\delta}(L[\mb X]). 
\end{align*}
We have 
\begin{align}
    |\mc P|_{\delta} &\sim |\mb X|_{\delta\times \delta\times 1} \\ 
    |\mc T|_{\delta} &\sim |\mb X|_{1\times \delta\times \delta}\label{eq:delta_covering_num_T_set}
\end{align}
First of all, $\mc P$ is a $(\delta, t, \delta^{-\eta-o(1)})$-set because using uniformity and the lower bound on $f(a,0)$,
\begin{align*}
    \frac{|\mc P\cap B(x, \delta^a)|_{\delta}}{|\mc P|_{\delta}} \lesssim_K \frac{|\mb X|_{\delta^a\times \delta^a\times 1}}{|\mb X|_{\delta\times \delta\times 1}} \lesssim \delta^{ta-\eta-o(1)}.
\end{align*}
Next, consider a $p \in \mc P$ which contains $p_{\omega}$ for some $\omega \in \mb X$. Let $\mb R = \mb R_{\delta_0\times \delta_0\times 1}(\omega)$ and $\mb X_{\omega} = \mb X \cap \mb R$. 
For $T$ a dyadic $\delta_0$-tube, let $T^{\delta}$ be the unique dyadic $\delta$-parent. Let 
\begin{align*}
    \mc T(p) = \{T^{\delta}\, :\, T \in \mc T^{\delta_0}(L[\mb X_{\omega}])\}. 
\end{align*}
Every $T \in \mc T(p)$ intersects $p$ because $\delta_0$ is small enough compared to $\delta$. We have 
\begin{align*}
    |\mc T(p)|_{\delta} &\sim |\mb X_{\omega}|_{1\times 1\times \delta} \sim \delta^{-d(1\sep 1, 0)\pm o(1)}.
\end{align*}
Using uniformity and \Cref{lem:branching_struct_subrect},
\begin{align*}
    \frac{|\mc T(p)\cap B(\ell, \delta^t)|_{\delta}}{|\mc T(p)|_{\delta}} \lesssim \frac{|\mb X_{\omega}|_{1\times 1\times r}}{|\mb X_{\omega}|_{1\times \delta\times \delta}} \sim \delta^{d(t\sep 1, 0) \pm o(1) }\lesssim \delta^{tx-\eta-o(1)}.
\end{align*}
Thus $(\mc P, \mc T)$ is a $(\delta, s, \delta^{-\eta-o(1)}, \delta^{-d(1\sep 1, 0) + o(1)})$-nice configuration. The result then follows from applying \Cref{thm:RenWangFurstenberg} and \eqref{eq:delta_covering_num_T_set}.
\end{proof}

\subsection{Extracting a Katz-Tao subset of a Frostman set}
Let $P \subset [-1,1]$ be a $(\delta, s, C)$-set, meaning $P$ is $\delta$-separated and satisfies 
\begin{align*}
    |P\cap B_w(x)| \leq Cw^s |P|\qquad \text{for all balls $B_w(x)$ with $w > \delta$.}
\end{align*}
The following Lemma is well known, we include a proof for the reader's convenience. 
\begin{lemma}\label{lem:katz_tao_subset}
If $P \subset [-1,1]$ is a $(\delta, s, C)$-set, then there is a subset $P' \subset P$ with $|P'| \geq  \frac{1}{6}C^{-1}\delta^{-s}$ satisfying 
\begin{align*}
    |P'\cap B(x,w)| \leq 4(w/\delta)^s \qquad \text{for all intervals $B(x,w)$ with $w \in [\delta, 1]$.}
\end{align*}
We say $P'$ is a $(\delta, s, C)$ Katz-Tao set. 
\end{lemma}
\begin{proof}
First of all, if $p \in P$ then because $P$ is $\delta$-separated 
\begin{align*}
    1 = |P\cap B(p, \delta)| \leq C\delta^{s} |P|
\end{align*}
implying $|P| \geq C^{-1}\delta^{-s}$. 

Set $P_1 = P$ and let $B_{w_1}(p_1)$ be the smallest interval centered at some $p_1 \in P_1$ so that $|P_1\cap B_{w_1}(p_1)| \geq C\delta^{s} |P|$. Let $P_2 = P \setminus B_{w_1}(p_1)$. Then choose $B_{w_2}(p_2)$ to be the smallest interval so that $|P_2\cap B_{w_2}(p_2)| \geq C\delta^{s}|P|$. Because $|P_j\cap B_{w_j}(p_j)| \leq C \delta^s|P|+2 \leq 3C\delta^s|P|$, we have $|P_{j+1}| \geq |P| - 3C\delta^s k |P|$ and we can continue this process for $m = \lceil \frac{1}{6}C^{-1}\delta^{-s}\rceil$ steps. 
Set $P' = \{p_1,\ldots, p_m\}$. Let $B_w(x)$ be an interval with $w \geq \delta$. Then 
\begin{align*}
    |P\cap B_{3w}(x)| \geq \sum_{p_j \in B_w(x), w_j \leq 2w, } |P_j\cap B_{w_j}(p_j)| \geq C\delta^s|P| \#\{p_j \in B_w(x)\text{ with } w_j \leq 2w\}
\end{align*}
So rearranging
\begin{align*}
    \#\{p_j \in B_w(x)\text{ with } w_j \leq 2w\} \leq 3 (w/\delta)^s. 
\end{align*}
Also, 
\begin{align*}
    \#\{p_j\in B_w(x)\text{ with } w_j \geq 2w\} \leq 1
\end{align*}
so overall we find 
\begin{align*}
    |P'\cap B_w(x)| \leq 3(w/\delta)^s + 1 \leq 4(w/\delta)^s. 
\end{align*}
\end{proof}

\section{A Connection to Heilbronn's problem for \texorpdfstring{$k$}{k}-gons}\label{sec:heilbronn_kgon}

In this section we discuss potential applications of our point-line setup to a generalization of Heilbronn's triangle problem. Let $k \ge 3$ be a fixed integer and define a function $\Delta_k(n)$ to be the smallest $\Delta >0$ such that any set of $n$ points in the unit square contains $k$ points whose convex hull has area at most $\Delta$. If $k=3$ then we get the original Heilbronn's triangle problem.

The best known lower and upper bounds for $k\ge 4$ can be summarized as follows:
\begin{equation}\label{eq:deltak}
n^{-1 - \frac{1}{k-2}} \log^{\frac{1}{k-2}} n\lesssim \Delta_k(n) \lesssim_k n^{-1}.    
\end{equation}

The upper bound follows from a simple pigeonhole argument and any asymptotic improvement of this bound would be interesting. The only non-trivial result we are aware of is a slight improvement of the implied constant in front of $n^{-1}$ \cite{kgons-constant}. The lower bound appears in \cite{lefmann2008distributions} and is based on the semi-random method introduced by Koml\'os, Pintz and Szemer\'edi \cite{komlos1982lower}.

Our machinery does not appear to be sufficient to make progress on this problem. However, if one can obtain sufficiently larger exponents in place of 2/3 in \Cref{cor:dist_P_L}, then it would lead to polynomial improvements over the trivial upper bound (\ref{eq:deltak}). The goal of this section is to explain the connection between our point-line incidence problem and the function $\Delta_k(n)$.

\definition{ For $\gamma \in [1/2, 1]$, let $\operatorname{PL}(\gamma)$ denote the following assertion. For any $\varepsilon >0$ and $n > n_0(\varepsilon)$ and for any collection $\{p_j \in \ell_j\}_{j=1}^n$ of point-line pairs with $p_j\in [0,1]^2$, there is $j\neq k$ for which 
\[
d(p_j, \ell_k) \le n^{-\gamma+\varepsilon}.
\]
}

We show that assuming $\operatorname{PL}(\gamma)$ for $\gamma$ sufficiently close to $1$ leads to bounds of the form $\Delta_k(n) \lesssim n^{-1-\beta}$. Here is our main inductive claim:

\begin{prop}\label{prop:kgon}
    Let $\gamma \in [1/2, 1]$, $k\ge 3$, $\beta \in [0,1]$ and assume that $\operatorname{PL}(\gamma)$ holds and we have the bound $\Delta_k(n) \lesssim n^{-1-\beta+o(1)}$. If $\gamma \ge \frac{1+\beta}{1+2\beta}$ then we have $\Delta_{k+1}(n) \lesssim n^{-1-\beta' +o(1)}$ with $\beta' = \gamma \frac{1+2\beta}{2+2\beta} - \frac{1}{2}$.
\end{prop}

\begin{proof}
    Let $P \subset [0,1]^2$ be a set of size $n$. Let $w = n^{-\frac{1}{2} +\frac{\gamma}{2(1+\beta)}}$ be a parameter and cover $[0,1]^2$ by disjoint $w$-squares $Q$ and denote $P_Q = P \cap Q$. By dyadic pigeonholing, we can find some $m \in [w^2 n, n]$ and a collection of $w$-squares $\mc Q$ so that $|\mc Q| \gtrsim \frac{n}{m \log n}$ and $|P_Q| \sim m$ for every $Q \in \mc Q$.
    
    Fix $Q \in \mc Q$. Apply the the definition of $\Delta_k$ to every $|P_Q|/2$-element subset of the set $P_Q$, rescaled back to the unit square $[0,1]^2$. After doing the inverse rescaling, it follows that for at least $|P_Q|/2$ points $p \in P_Q$ there exists some rectangle $R_p$ centered at $p$ such that $|R_p \cap P_Q|\ge k$ and $R_p$ has sides $a_p, b_p$ satisfying:
    \begin{equation}\label{eq:propab}
    a_p b_p \sim w^2 m^{-1-\beta+o(1)}, ~ a_p \le b_p \le w.    
    \end{equation}
    So by passing to an appropriate subset $P'_Q \subset P_Q$ and applying dyadic pigeonholing, we may assume that such a rectangle $R_p$ exists for every $p \in P'_Q$ and, moreover there are some $a, b$ such that $a_p = a$ and $b_p = b$ for all $p \in P'_Q$ and $Q \in \mc Q$. Furthermore, we may also assume that for every $p \in P'_Q$ there exists a set $S_p \subset P_Q \setminus P'_Q$ of size $k-1$ such that $S_p \subset R_p$ and the sets $S_p$, $p \in P' = \bigcup P'_Q$ are pairwise disjoint.

    For each point $p \in P'$ let $\ell_p$ be the line passing through $p$ in the direction of the longer side of $R_p$.
    Let $u = n^{-1/2} (b/a)^{\frac{1}{2\gamma}}$ be a parameter and for any $u$-square $Q'$, consider the rescaling map $\psi: Q' \rightarrow [0,1]^2$ and apply the property $\operatorname{PL}(\gamma)$ to the point-line configuration 
    \[
    \{ \psi(p) \in \psi(\ell_p), ~p \in P' \cap Q' \}.
    \]
    This produces a pair of points $p' \neq p \in P' \cap Q'$ such that
    \[
    d( p', \ell_p )\lesssim u |P' \cap Q'|^{-\gamma+o(1)},
    \]
    provided that $|P' \cap Q'| > 1$. Let $R'$ be the smallest rectangle containing $p'$ and $R_p$. By construction we have $\{p, p'\} \cup S_p \subset R'$ and since $S_p$ was chosen to be disjoint from $P'$ we get $|R' \cap P|\ge k+1$. It remains to bound the area of $R'$.
    
    Change the coordinate system so that $p = (0, 0)$, $R_p = [-a/2,a/2] \times [-b/2, b/2]$ and denote $p' = (x, y)$ (for simplicity assume that $x,y\ge 0$). Then $\ell_p$ is the Y-axes, and $p, p' \in Q'$ implies that $x, y \lesssim u$. We also have 
    \[
    x = d(p', \ell_p) \lesssim u |P' \cap Q'|^{-\gamma+o(1)}.
    \]
    Let $K = \max\{1, x/a, y/b\}$, then one can check that
    \[
    |R'| \sim K |R| \sim K a b \lesssim \max\{ab, x b, a y\}.
    \]
    Using the estimates $ab \sim w^2 m^{-1-\beta+o(1)}$, $m \gtrsim w^2 n$ and $|P'\cap Q'| \gtrsim u^2 n^{1+o(1)}$ (which holds since $u \gtrsim n^{-1/2+o(1)}$) and the bounds on $x$ and $y$ we obtain
    \begin{equation}\label{eq:R}
    |R'| \lesssim \max\{ w^{-2\beta} n^{-1-\beta+o(1)}, u^{1-2\gamma} n^{-\gamma} b, u a \}.    
    \end{equation}
    Recall that we chose $w = n^{-\frac{1}{2} +\frac{\gamma}{2(1+\beta)}}$ and given $b \in [w^{-\beta} n^{-\frac{1+\beta}{2}}, w]$ we defined $u = n^{-1/2} (b/a)^{\frac{1}{2\gamma}}$. The choice of $w$ implies that $u \in [n^{-1/2}, 1]$ holds so it is valid for our estimate. Taking such $u$ makes the last two terms in (\ref{eq:R}) agree and so we obtain
    \[
    |R'| \lesssim \max\{w^{-2\beta}n^{-1-\beta}, u^{1-2\gamma} n^{-\gamma} b\} \lesssim w^{-2\beta}n^{-1-\beta} + n^{-\frac{2\gamma-1}{2}} (ab)^{\frac{2\gamma-1}{2\gamma}} b^{\frac{1-\gamma}{\gamma}}
    \]
    the second term is the largest when $b = w$, in which case we get $u=1$ (by the definition of $w$) which simplifies the bound to
    \[
    |R'| \lesssim w^{-2\beta}n^{-1-\beta} + w n^{-\gamma} \lesssim n^{-1 - \frac{\beta\gamma}{1+\beta}} + n^{-\frac{1+\beta+\gamma+2\gamma\beta}{2(1+\beta)}}.
    \]
    The second term beats the trivial bound $\Delta_{k+1}(n) \lesssim n^{-1}$ if and only if $\gamma > \frac{1+\beta}{1+2\beta}$. The restriction $\gamma, \beta \le 1$ and a brief computation show that the second term is always larger than the first, leading to the claimed estimate on $\Delta_{k+1}(n)$.
\end{proof}

Let us list some corollaries of \Cref{prop:kgon}:
\begin{itemize}
    \item Since $\operatorname{PL}(\gamma)$ implies the bound $\Delta_3(n) \lesssim n^{-\frac{1}{2}-\gamma}$ for the usual Heilbronn's triangle function, one can check that this and \Cref{prop:kgon} imply a non-trivial bound $\Delta_{4}(n)\lesssim n^{-1-\varepsilon}$ assuming $\operatorname{PL}(\gamma)$ for any $\gamma > \gamma_4$, where $\gamma_4 = \frac{1+\sqrt{5}}{4} \approx 0.809$.
    \item Similarly, for any $k \ge 5$ let us define $\gamma_k$ as the `smallest' value so that $\operatorname{PL}(\gamma)$ for $\gamma>\gamma_k$ implies $\Delta_{k}(n) \lesssim n^{-1-\varepsilon}$. A straightforward computation then shows that $1-\gamma_k$ decays roughly as $2^{-k}$ and that the best possible bound coming from this method (i.e. if we assume $\operatorname{PL}(1)$) would give 
    \[
    \Delta_{k}(n)\lesssim n^{ - 1 -\frac{ 1}{2(2^{k-2}-1)}+o(1)}.
    \]
\end{itemize}

Any statement $\operatorname{PL}(\gamma)$ for $\gamma > 2/3$ seems hard and would be interesting. However, we believe that \Cref{prop:kgon} gives a natural line of attack on Heilbronn's problem for $k$-gons and is another motivation to study the point-line incidence problem.

\bibliographystyle{amsplain0.bst}
\bibliography{main}

\end{document}